\RequirePackage{fix-cm}
\documentclass[smallextended]{svjour3}       
 
\smartqed  
\usepackage{graphicx}

\usepackage{epsfig,epsf,fancybox}
\usepackage{amsmath}
\usepackage{mathrsfs}
\usepackage{amssymb}
\usepackage{amsfonts}
\usepackage{graphicx}
\usepackage{color}
\usepackage[linktocpage,colorlinks,linkcolor=blue,anchorcolor=blue,citecolor=magenta,urlcolor=cyan,hypertexnames=false]{hyperref}
\usepackage{boxedminipage}
\usepackage{stmaryrd}
\usepackage{multirow}
\usepackage{booktabs}
\usepackage{relsize}
\usepackage{accents}
\usepackage{cite}
\usepackage{float}
\usepackage[ruled,vlined,linesnumbered]{algorithm2e}
\usepackage{algpseudocode}
\usepackage{breqn}
\usepackage{lineno}
\usepackage{mathtools}
\usepackage{xcolor}
\usepackage{enumerate}
\usepackage{breqn}
\usepackage{rotating}
\usepackage{lscape}

\usepackage{stmaryrd}
\SetSymbolFont{stmry}{bold}{U}{stmry}{m}{n}

\newcommand{\vxi}{{\boldsymbol{\xi}}}

\usepackage{mathtools}

\usepackage[normalem]{ulem} 

\newcommand{\va}{{\mathbf{a}}}
\newcommand{\vb}{{\mathbf{b}}}
\newcommand{\vc}{{\mathbf{c}}}

\newcommand{\vf}{{\mathbf{f}}}
\newcommand{\vg}{{\mathbf{g}}}

\newcommand{\vl}{{\mathbf{l}}}

\newcommand{\vp}{{\mathbf{p}}}

\newcommand{\vu}{{\mathbf{u}}}
\newcommand{\vv}{{\mathbf{v}}}
\newcommand{\vw}{{\mathbf{w}}}
\newcommand{\vx}{{\mathbf{x}}}
\newcommand{\vy}{{\mathbf{y}}}
\newcommand{\vz}{{\mathbf{z}}}

\newcommand{\zz}{^{\top}}
\newcommand{\vA}{{\mathbf{A}}}

\newcommand{\vM}{{\mathbf{M}}}

\newcommand{\vQ}{{\mathbf{Q}}}

\newcommand{\cD}{{\mathcal{D}}}

\newcommand{\cL}{{\mathcal{L}}}
\newcommand{\cM}{{\mathcal{M}}}
\newcommand{\cN}{{\mathcal{N}}}
\newcommand{\cO}{{\mathcal{O}}}
\newcommand{\cP}{{\mathcal{P}}}

\newcommand{\cU}{{\mathcal{U}}}

\newcommand{\cX}{{\mathcal{X}}}

\newcommand{\vareps}{\varepsilon}

\newcommand{\RR}{\mathbb{R}} 
\newcommand{\vzero}{\mathbf{0}} 

\newcommand{\dist}{\mathrm{dist}}    
\newcommand{\prox}{{\mathbf{prox}}} 
\newcommand{\dom}{{\mathrm{dom}}} 

\newcommand{\st}{\mbox{ s.t. }}

\DeclareMathOperator*{\argmin}{arg\,min} 

\newcommand{\bc}{\begin{center}}
\newcommand{\ec}{\end{center}}

\newcommand{\bdm}{\begin{displaymath}}
\newcommand{\edm}{\end{displaymath}}

\newcommand{\beq}{\begin{equation}}
\newcommand{\eeq}{\end{equation}}

\newcommand{\bfl}{\begin{flushleft}}
\newcommand{\efl}{\end{flushleft}}

\newcommand{\bt}{\begin{tabbing}}
\newcommand{\et}{\end{tabbing}}

\newcommand{\beqn}{\begin{eqnarray}}
\newcommand{\eeqn}{\end{eqnarray}}

\newcommand{\beqs}{\begin{align*}} 
\newcommand{\eeqs}{\end{align*}}  

\newtheorem{assumption}{Assumption}

\begin{document}
\allowdisplaybreaks[4]

\title{Damped Proximal Augmented Lagrangian Method for weakly-Convex Problems with Convex Constraints}

 
\author{ Hari Dahal, Wei Liu, Yangyang Xu}

\institute{Hari Dahal \at
              Department of Mathematical Sciences, Rensselaer Polytechnic Institute, Troy, NY 12180 \\
              \email{dahalh@rpi.edu}          
           \and
           Wei Liu \at
              Department of Mathematical Sciences, Rensselaer Polytechnic Institute, Troy, NY 12180 \\
              \email{lwdsdqqb@gmail.com}
              \and
            Yangyang Xu \at 
            Department of Mathematical Sciences, Rensselaer Polytechnic Institute, Troy, NY 12180 \\
              \email{xuy21@rpi.edu}
              }

\date{Received: date / Accepted: }

\maketitle

\begin{abstract}
We present a damped proximal augmented Lagrangian method (DPALM) for solving problems with a weakly-convex objective and convex linear/non-linear constraints. Instead of taking a full stepsize, DPALM adopts a (possibly) damped dual stepsize. 
We show that DPALM can produce a (near) $\vareps$-KKT point within $\cO(\vareps^{-2})$ outer iterations  
if each DPALM subproblem is solved to a proper accuracy. In addition, we establish overall oracle complexity {(i.e., total number of evaluations of function values and (sub)gradients)} of DPALM when the objective is either a regularized smooth function or in a regularized compositional form. For the former case, DPALM achieves a complexity of $\widetilde{\mathcal{O}}\left(\varepsilon^{-2.5} \right)$ to produce an $\varepsilon$-KKT point by applying an accelerated proximal gradient (APG) method to each DPALM subproblem{\footnote{Throughout this paper, we use $\widetilde\cO(\cdot)$ to hide a logarithmic term.}}. For the latter case, the complexity of DPALM is $\widetilde{\mathcal{O}}\left(\varepsilon^{-3} \right)$ to produce a near $\varepsilon$-KKT point by using an {APG} {method} to solve a Moreau-envelope smoothed version of each subproblem.  
Our outer iteration complexity and the overall complexity either generalize existing best ones from unconstrained or linear-constrained problems  to convex-constrained ones, or improve over the best-known results on solving the same-structured problems. Furthermore, numerical experiments on linearly/quadratically constrained non-convex quadratic programs and linear-constrained robust nonlinear least squares are conducted to demonstrate the empirical efficiency of the proposed DPALM over several state-of-the-art methods {and a specialized solver}.
\end{abstract}
	
	\section{Introduction}
	Given the rapid increase of data volume in modern applications, there has been a substantial surge of interest in designing first-order methods (FOMs). Traditionally, a significant portion of research in optimization has focused on convex problems. However, there has been a noticeable and accelerating shift towards the investigation of non-convex optimization during the past decade. This trend is primarily attributable to the practical applications and the growing recognition that many contemporary optimization challenges indeed fall within the category of non-convex problems.
	
	In this paper, we consider to design new FOMs for non-convex constrained optimization {problems} in the form of
	\begin{equation}\label{problem:main}
		\begin{aligned}
			&\min_{ \vx \in \mathbb{R}^d} \,\, F( \vx):=f(\vx) + h(\vx),\\
                &\st \,\, \vA \vx = \vb, \quad \vg( \vx):=[g_1( \vx), g_2(\vx), \ldots, g_m( \vx)] \leq \mathbf{0},
		\end{aligned}
		\tag{P}
	\end{equation}
	where $\vA \in \mathbb{R}^{n \times d}, \vb \in \mathbb{R}^n$, $F$ is continuous on its domain $\mathcal{X}:= \mathrm{dom}(F)$, $h$ is closed convex, and
	$g_i: \mathbb{R}^d \rightarrow \mathbb{R}$ is closed convex for each $i \in [m] $ with $[m] := \{1, 2, \ldots, m\}$. We will assume that $h$ is simple and admits an easy proximal mapping, each $g_i$ is smooth 
	on an open set containing $\mathcal{X}$, and 
	$f$ is $\rho$-weakly convex and possibly non-smooth; see Definition~\ref{def:weak-cvx} below. 
	
	The problem~\eqref{problem:main} is rather general and has many interesting applications, such as non-convex quadratic programs with linear and/or nonlinear constraints, reformulation of the nonnegative matrix completion by variable splitting \cite{xu2012alternating}, hyperspectral image denoising \cite{chen2021hyperspectral}, classification problem with ROC-based constraints~\cite{huang2023single}, 
	and the Neyman-Pearson classification~\cite{rigollet2011neyman}. 
	 
	In the realm of non-convex optimization,  locating a global optimizer is usually a computationally intractable task \cite{danilova2022recent}. As a practical alternative, the goal is often directed towards identifying a stationary point. 
On solving~\eqref{problem:main}, we aim at finding a (near) $\varepsilon$-KKT point for a given $\varepsilon>0$; 
see Definition~\ref{def: epskkt}.

	\subsection{Algorithmic Framework} \label{sec: AlgFrame}
The first-order method (FOM) we design is based on the framework of an inexact damped proximal augmented Lagrangian method (DPALM), {in which a damped step size is used when updating the Lagrange multipliers.}
	The augmented Lagrangian (AL) function of~\eqref{problem:main} is 
	\begin{equation}\label{eq:aug-fun}
		\cL_\beta( \vx;\vy, \vz) = F( \vx)  + \vy^{\top}(\vA \vx-\vb)+ \frac{\beta}{2}\|\vA \vx-\vb\|^2 + \frac{\beta}{2}\left\|\left[ \vg( \vx)+\frac{ \vz}{\beta}\right]_+\right\|^2 - \frac{\| \vz\|^2}{2\beta},
	\end{equation}
	where $\beta>0$ is a penalty parameter, $[ \va]_+$ takes the component-wise positive part of a vector $ \va$, and $ \vy \in \mathbb{R}^n$ and $\vz \in \mathbb{R}^m$ are {Lagrange} multipliers. 
 Define {the} proximal AL function as 
		\begin{equation}\label{eq:aug-fun2}
			\widetilde{\cL}_{\beta}( \vx; \vy, \vz) = \cL_\beta( \vx;\vy, \vz) + \rho \| \vx-\vx^{k}\|^2,
		\end{equation}
		which is $\rho$-strongly convex due to the $\rho$-weak convexity of $F$. Notice that we use $\rho$ in~\eqref{eq:aug-fun2} for convenience. It can be any value that is strictly larger than~$\frac{\rho}{2}$.
	With 	$\widetilde{\cL}_{\beta}$, we present the DPALM framework in Alg.~\ref{alg:ialm}, where we adopt the convention of $\frac{v}{0}=+\infty$ for any $v>0$.
		
		\begin{algorithm}[htbp!]
			\caption{A Damped Proximal Augmented Lagrangian Method (DPALM) for 
			\eqref{problem:main}}\label{alg:ialm}
			\DontPrintSemicolon
				\textbf{Initialize} $ \vx^0\in\mathcal{X}=\dom(F)$, $\vy^0  \in \text{Range}(\vA)$, and $\vz^0\geq \vzero$. Let $k=0$.\; 
				\textbf{Choose} {two positive sequences} $\{\beta_k\}_{k\ge0}$ and$ \{v_k\}_{k\ge0}$.
				
				\While{a stopping criterion is not met,}{
					
					Step 1: Obtain an approximate solution $\vx^{k+1}$ of $\min_\vx \widetilde{\cL}_{\beta_k}( \vx; \vy^k, \vz^k )$ {such that \eqref{eq:approx-cond} in Remark \ref{rm:approx-sol} holds.}
					
					Step 2: Set $$\vy^{k+1} = \vy^k + \alpha_k(\vA \vx^{k+1}-\vb), \text{ and }\vz^{k+1} = \vz^k+ \alpha_k \max\left\{-\frac{ \vz^k}{\beta_k}, \vg\left(\vx^{k+1}\right)\right\}$$ with $\alpha_k := \min
                     \left\{\beta_k, \frac{v_k}{\sqrt{\| \vA \vx^{k+1}- \vb\|^2 + \|[\vg(\vx^{k+1})]_+ \|^2}}\right\}$.

				     Increase $k\gets k+1$.
				}
                
			\textbf{Output:} $\vx^{k}$.
		\end{algorithm}

In Alg.~\ref{alg:ialm}, we use a (possibly) damped stepsize $\alpha_k$ instead of a full stepsize~$\beta_k$ for the $\vy$- and $\vz$-updates. We mention here that our algorithm allows the choice of $\alpha_k=\beta_k$ but will have a higher worst-case complexity result; see our discussions in Remark~\ref{rem:1}. Below we give a remark on how to compute $\vx^{k+1}$.

\begin{remark}[approximate solution $\vx^{k+1}$]\label{rm:approx-sol}
The choice of $\vx^{k+1}$ in Step~1 of Alg.~\ref{alg:ialm} will be made clear in Sect.~\ref{Sec: ConvAnalysisII}, where we discuss three different cases of $f$: (i)~$f$ is smooth; (ii) $f$ is a composition of a convex function $l$ with a smooth mapping~$\vc$, i.e., $f = l\circ \vc$; (iii) $f$ is a general weakly-convex function. For all cases, $\vx^{k+1}$ satisfies the condition:
\begin{equation}\label{eq:approx-cond}\tag{ApproxCond}
\|\vx^{k+1} - \widehat{\vx}^{k+1} \| \leq \delta_k, \text{ with }\widehat{\vx}^{k+1} \in \cX \text{ and }\dist\left(\vzero, \partial_{\vx}\overline{\cL}_{\beta_k}(\widehat{\vx}^{k+1}; \vy^k, \vz^k)\right) \leq \varepsilon_k
\end{equation}
for some $\delta_k \geq 0$ and $\varepsilon_k\in (0, 1]$. Here, 
\begin{equation}\label{eq:barL-func}
\begin{aligned}
 \overline{\cL}_{\beta_k}(\vx; \vy^k, \vz^k) :=        &~   \overline{f}(\vx) + h(\vx) + (\vy^k)^\top (\vA\vx-\vb) + \frac{\beta_k}{2}\|\vA\vx-\vb \|^2 \\
  &~    + \frac{\beta_k}{2}\left\| \left[ \vg(\vx)
        + \frac{\vz^k}{\beta_k} \right]_+\right\|^2 - \frac{\|\vz^k\|^2}{2\beta_k} + \rho\|\vx-\vx^k\|^2,
    \end{aligned}
\end{equation} 
$ \overline{f}$ is $\rho$-weakly convex and either equals $f$ or is a smoothed approximation of $f$, and there is a constant $\overline{B}_f>0$ such that $\| \vxi_f \|\leq \overline{B}_f$ for any $\vxi_f\in\partial \overline{f}(\vx)$ and any $\vx\in \cX$. 
\end{remark}

		\subsection{Related Work}
        Significant efforts on FOMs for non-convex optimization problems have been dedicated to unconstrained or simple-constrained settings, as evidenced by a notable body of research~\cite{allen2017natasha, davis2018stochastic, davis2019stochastic, davis2019proximally, drusvyatskiy2019efficiency, ghadimi2013stochastic, ghadimi2016accelerated, lan2019accelerated, reddi2016stochastic, zhang2018convergence}. For problem~\eqref{problem:main}, these methods are inapplicable or inefficient as projecting onto the constraint set of~\eqref{problem:main} can be prohibitively expensive. Also, many existing FOMs for affine and/or nonlinear functional constrained optimization deal with the convex case; see, for example,~\cite{aybat2011first, lan2013iteration-penalty, lan2016iteration, liu2019nonergodic, lu2023iteration, necoara2019complexity, patrascu2017adaptive, xu2021iteration-ialm, bayandina2018mirror, lin2018level, tran2014primal, xu2021first-ALM, wei2018solving} for a deterministic case and~\cite{wei2018primal, lan2020algorithms, yu2017simple, xu2020primal} for a stochastic case.
		 {In what follows,} 
         we review existing FOMs for solving non-convex optimization problems in the form of~\eqref{problem:main}. 

FOMs for solving affinely constrained non-convex optimization problems, i.e., 	problem~\eqref{problem:main} with $\vg\equiv \vzero$, have been studied extensively such as in \cite{hajinezhad2019perturbed, goncalves2017convergence, hong2016decomposing, jiang2019structured, melo2017iteration, melo2024proximal, zhang2022global, zhang2020proximal}. 	
  Hajinezhad and Hong~\cite{hajinezhad2019perturbed} introduce a perturbed proximal primal-dual algorithm (PProx-PDA) with a complexity result of $\widetilde{\mathcal{O}}(\varepsilon^{-4})$.
The FOM in \cite{melo2024proximal} is based on the inexact proximal accelerated augmented Lagrangian (IPAAL) method.  
It can produce an $\varepsilon$-KKT point within {$\widetilde\cO\left(\varepsilon^{-2.5} \right)$ iterations under Slater's condition}. Kong et al.~\cite{kong2019complexity} give an FOM based on a quadratic penalty accelerated inexact proximal point method and show a complexity result of $\mathcal{O}\left(\varepsilon^{-3} \right)$. 
On a special class of affine-constrained non-convex optimization problems, where the regularizer is the indicator function of a polyhedral set, Zhang and Luo~\cite{zhang2020proximal, zhang2022global} introduce an FOM based on a proximal alternating direction method of multipliers (ADMM) and show that their method can generate an $\varepsilon$-KKT solution within $\mathcal{O}\left(\varepsilon^{-2} \right)$ iterations. 
The methods in \cite{goncalves2017convergence, hong2016decomposing, jiang2019structured, melo2017iteration} are all variants of ADMM and can produce an $\vareps$-KKT solution within $\mathcal{O}(\vareps^{-2})$ iterations under a certain assumption about the matrices in the affine constraint. Kong and Monteiro \cite{kong2023accelerated} introduce an accelerated inexact dampened ALM for non-convex composite problems with linear constraints and achieve a complexity result of $\widetilde{\mathcal{O}}(\varepsilon^{-2.5})$. By our notation and notating $(1-\theta)p_k$ in \cite{kong2023accelerated} as $\vy^{k+1}$, the dampened dual update in \cite{kong2023accelerated} becomes $\vy^{k+1} = (1-\theta) (\vy^{k} + \chi \beta_k(\vA\vx^{k+1}-\vb))$, where $\theta$ and $\chi$ are both in $(0,1)$ and satisfy $(1-\theta)(2-\theta)\chi \le \theta^2$. This is different from our dual update.

For regularized non-convex smooth optimization problems with convex nonlinear constraints, Li et al.~\cite{li2021augmented} design an FOM, called HiAPeM, by applying a hybrid of ALM  and a quadratic penalty method to a sequence of proximal point subproblems. Under Slater's condition, HiAPeM is able to produce an $\varepsilon$-KKT point with complexity of $\widetilde{\mathcal{O}}(\varepsilon^{-2.5})$. It is demonstrated in \cite{li2021augmented} that more frequent use of ALM {updates} can yield better empirical performance. However, obtaining the $\widetilde{\mathcal{O}}(\varepsilon^{-2.5})$ complexity result requires using the quadratic penalty method more frequently. This differs from our proposed DPALM, which is solely ALM-based and can yield better practical performance. Kong et al.~\cite{kong2022iteration} consider a convex cone-constrained regularized non-convex smooth optimization problem. With an appropriate convex cone, the problem considered in \cite{kong2022iteration} can have the same constraints as those in problem~\eqref{problem:main}. However, the objective function in \cite{kong2022iteration} is a special case of what we consider. Similar to our method, the FOM in \cite{kong2022iteration}, called NL-IAPIAL, is also based on the proximal ALM framework. 
Compared to our method, NL-IAPIAL tracks a condition (see Eqn.~(35) in \cite{kong2022iteration}) --- when the condition holds, it doubles the penalty parameter; otherwise the penalty parameter is kept unchanged. 
Another key difference from our method is 
that NL-IAPIAL always uses $\beta_k$ as the dual stepsize. 
Due to these differences, NL-IAPIAL has a higher complexity than ours. It requires $\widetilde{\mathcal{O}}( \varepsilon^{-3})$ first-order oracle {calls} to  produce an $\varepsilon$-KKT point. 

For non-convex optimization problems with non-convex constraints, Sahin et al.~\cite{sahin2019inexact} design an ALM-based FOM. Under a regularity condition that ensures near primal feasibility at a near stationary point of an ALM subproblem, their FOM can produce an $\varepsilon$-KKT point by $\widetilde{\mathcal{O}}(\varepsilon^{-4} )$ calls to the first-order oracle. The two works \cite{lin2022complexity} and \cite{li2021rate} assume the same regularity condition as that in \cite{sahin2019inexact}, but differently, their FOMs achieve an $\widetilde{\mathcal{O}}(\varepsilon^{-3} )$ complexity result, by adopting the framework of a proximal point penalty method and ALM respectively. 
The FOM in \cite{li2021rate} is also analyzed for problems with convex constraints, in which case its complexity becomes $\widetilde{\mathcal{O}}\big( \varepsilon^{-\frac{5}{2}}\big)$. Without a regularity condition but instead assuming a feasible initial point, the method in \cite{lin2022complexity} achieves a complexity result of $\widetilde{\mathcal{O}}( \varepsilon^{-4}) $. 
For solving non-smooth weakly convex problems with convex or weakly convex nonlinear constraint, Huang and Lin propose a single-loop switching subgradient method in \cite{huang2023single}. They introduce a switching stepsize rule to accompany the switching subgradient and show that their method can find a near $\varepsilon$-KKT point in $\mathcal{O}(\varepsilon^{-4})$ iterations by assuming a (uniform) Slater's condition. Curtis and Overton \cite{curtis2012sequential} give a method based on sequential quadratic programming (SQP)  for solving problems where the objective and constraint functions {may be non-convex and non-smooth.}  
{Global convergence to stationarity is established under  
a locally Lipschitz and continuous differentiability condition on open dense subsets of $\RR^d$ and a boundedness condition of first-order derivatives at algorithm-generated points.}

The two works \cite{drusvyatskiy2019efficiency} and \cite{zeng2022moreau} are most closely related to ours, though  {the problems they consider} are special cases of~\eqref{problem:main}. In \cite{drusvyatskiy2019efficiency},
Drusvyatskiy and Paquette consider the regularized compositional optimization problems in the form of $\min_{\vx} F(\vx):=  l(\vc(\vx)) + r(\vx)$, where $l$ and $r$ are closed convex functions, and $\vc$ is a smooth mapping. A Moreau-envelope-based smoothing prox-linear method is analyzed in \cite{drusvyatskiy2019efficiency}, and it reaches a complexity result of $\widetilde{\mathcal{O}}( \varepsilon^{-3})$ to produce a near $\vareps$-stationary point. 
We note that the constraints in~\eqref{problem:main} can be encoded into the objective by adding an indicator composed function $\iota_{\{\vzero\}}(\vA\vx-\vb) + \iota_{\RR_-^m}(\vg(\vx))$. Using Moreau-envelope smoothing, i.e., replacing the indicator functions $\iota_{\{\vzero\}}(\cdot)$ and $\iota_{\RR_-^m}(\cdot)$ by their Moreau envelope (see Definition~\ref{def:weak-cvx}), the composed function can be smoothed to $\frac{1}{2\nu}\|\vA\vx-\vb\|^2 + \frac{1}{2\nu}\|[\vg(\vx)]_+\|^2$, where $\nu>0$ is a smoothing parameter. Hence, the smoothing prox-linear method in \cite{drusvyatskiy2019efficiency} can be applied to~\eqref{problem:main} with a regularized compositional objective as we consider in Sect.~\ref{subsec: caseII}, but it is based on the quadratic penalty of the constraints. In contrast, our method is based on the proximal ALM framework and can perform significantly better; see the experimental results in Sect.~\ref{subsec: nllsq} and Sect.~\ref{subsec: fairnessROC}.
 Zeng et al.~\cite{zeng2022moreau} introduce a Moreau Envelope Augmented Lagrangian Method (MEAL) for problems with a weakly-convex objective and linear constraints. MEAL can be viewed as a gradient update method on the Moreau envelope of the AL function. Assuming an exact solution of each proximal ALM subproblem, MEAL can produce an $\vareps$-KKT point within $\mathcal{O}(\vareps^{-2})$ outer iterations when either an implicit Lipschitz subgradient property or an implicit bounded subgradient property on the objective function holds. An inexact version, named iMEAL, is also proposed in \cite{zeng2022moreau}, which only requires an $\vareps_k$-stationary solution for the $k$-th proximal ALM subproblem. The same-order outer iteration complexity results are shown for iMEAL, provided that $\sum_{k=0}^\infty \vareps_k^2 < \infty$. When the objective function is composite, i.e., a smooth term plus a convex regularizer, \cite{zeng2022moreau} also presents a linearized variant of MEAL, which has the same-order outer iteration complexity as MEAL and iMEAL. The (inexact) MEAL becomes an (inexact) proximal ALM when its stepsize $\eta=1$; see the updates in Eqn.~(5) and Eqn.~(7) of \cite{zeng2022moreau}. On the special linear-constrained case, iMEAL can achieve the {same order} outer iteration complexity result as our proposed DPALM. 
 However, to do so under the implicit bounded subgradient assumption, iMEAL requires a penalty parameter setting of $\beta_k=\Theta(\vareps^{-2})$. In contrast, our proposed DPALM only needs to set $\beta_k = \Theta(\sqrt{k+1})$ for all $k\ge0$, which increases to $\Theta(\vareps^{-1})$ eventually to produce a (near) $\vareps$-KKT point. Though an overall first-order oracle complexity result is not explicitly shown in \cite{zeng2022moreau} for iMEAL, due to the higher-order penalty parameters, it will be higher than our complexity by at least an order of $\vareps^{-\frac{1}{2}}$ if iMEAL applies the same first-order subroutine as our method.

		\subsection{Contributions}
Our contributions lie in both algorithm design and complexity analysis.
They are summarized as follows.
		\begin{enumerate}[(i)]
			\item We propose a damped proximal augmented Lagrangian method (DPALM) to solve problems in the form of~\eqref{problem:main}, which has a weakly-convex objective and linear and/or nonlinear convex constraints. At each iteration of DPALM, a strongly convex subproblem is formed by adding a proximal term with an appropriate proximal parameter to the AL function. The primal variable is updated to {a solution of the subproblem with a desired accuracy}, and the dual variables (or Lagrange multipliers) are then updated by performing a dual gradient ascent step but with (possibly) damped stepsizes instead of using the penalty parameter as a full stepsize.  
			\item Under Slater's condition, we show that  for any given $\vareps>0$, DPALM can produce a near $\vareps$-KKT point of problem~\eqref{problem:main} (see Definition~\ref{def: epskkt}) within $\mathcal{O}(\vareps^{-2})$ outer iterations,  when $f$ in~\eqref{problem:main} is weakly convex and has uniformly bounded subgradients, and if for each DPALM subproblem, a near-optimal solution is found with a desired accuracy. This result generalizes that in \cite{zeng2022moreau} from an affine-constrained problem to an affine and/or convex functional constrained one. In addition, the value of the penalty parameter that we require is in a lower order than that in \cite{zeng2022moreau} to ensure a (near) $\vareps$-KKT point under the non-smooth case. This leads our method to have a lower overall complexity.
			
			\item For the case where $f$ in~\eqref{problem:main} is smooth but may be non-convex, we apply Nesterov's APG method to find a near stationary solution of each DPALM subproblem and establish an $\widetilde{\mathcal{O}}\big(\varepsilon^{-\frac{5}{2}} \big)$ complexity result to produce an $\varepsilon$-KKT point. This result improves the $\widetilde{\mathcal{O}}\big(\varepsilon^{-3} \big)$ complexity obtained in \cite{kong2022iteration} for a proximal ALM based FOM. It matches the complexity results in \cite{li2021augmented} and \cite{lin2022complexity} for either an ALM-penalty-hybrid method or a quadratic penalty-based FOM, which often yields worse empirical performance than DPALM as demonstrated in Sect.~\ref{sec: NumericalExp}. 
			\item For the case where $f$ is in a compositional form, we apply Nesterov's APG method to a Moreau-envelope smoothed version of each DPALM subproblem and establish an $\widetilde{\mathcal{O}}\big(\varepsilon^{-3} \big)$ complexity result to produce a near $\varepsilon$-KKT point. This result generalizes that in \cite{drusvyatskiy2019efficiency} for solving an unconstrained compositional problem. 
            {Though it is possible to encode the constraints in~\eqref{problem:main} via 
            an appropriate outer smooth convex function and an appropriate inner vector function, doing so within the framework of~\cite{drusvyatskiy2019efficiency} leads to a quadratic penalty-based method, which performs significantly worse than DPALM, as demonstrated in Sections~\ref{subsec: nllsq} and~\ref{subsec: fairnessROC}.}
		\end{enumerate}

		\subsection{Notations and Definitions}\label{subsec: NotationsAndDef}
		The interior and 
  boundary of a set $\mathcal{X}$ are denoted by 
  $\mathrm{int}(\cX)$ and 
  $\mathrm{bd} (\mathcal{X})$, respectively. 
		We use $\mathcal{N}_{\mathcal{X}}(\vx)$ to denote the normal cone of $\mathcal{X}$ at $\vx$ 
  and define $\mathcal{B}_r=\left\{\vx \in \mathbb{R}^d:\|\vx\| \leq r\right\}$ for some $r>0$, where $\|\cdot\|$ denotes the Euclidean norm. For a symmetric matrix $\vM$, its smallest positive eigenvalue is denoted as $\lambda_{\min}^+ (\vM)$; for a positive integer $m$, we define $[m]=\{1,\ldots,m\}$. {The range of the matrix $\vA$ is denoted as
  Range$(\vA)$.} 
		\begin{definition}[Weakly convex function and Moreau envelope\cite{drusvyatskiy2019efficiency}]\label{def:weak-cvx} A function $f$ is called $\rho$-weakly convex for some $\rho \ge 0$, if $f(\cdot) + \frac{\rho}{2}\|\cdot\|^2$ is convex.
			For a $\rho$-weakly convex function $f$, its Moreau envelope and the proximal mapping for any $\nu\in(0, \frac{1}{\rho})$ are defined by
			$
			f_\nu({\vx}) :=\min_{\vz}\left\{f({\vz})+\frac{1}{2 \nu}\|{\vz}-{\vx}\|^2\right\}
			$ and $\operatorname{prox}_{\nu f}({\vx}) :=\argmin_\vz\{f({\vz})+\frac{1}{2 \nu}\|{\vz}-{\vx}\|^2\},
			$ respectively.
		\end{definition}

{\begin{definition}[Subdifferential of weakly convex functions]\label{def: subdiff}
Let $f$ be a $\rho$-weakly convex function for some $\rho\ge 0$. If $\rho=0$, i.e., $f$ is convex, 
the subdifferential of~$f$ at $\vx \in \dom(f)$ is defined as~\cite{clarke1990optimization}
\begin{align*}
    \partial f(\vx) := \{\vxi \in \RR^n: f(\vx') - f(\vx) \geq \langle \vxi, \vx' - \vx \rangle \ \text{ for all } \vx' \in \dom(f)\};
\end{align*}
otherwise, if $\rho>0$, the subdifferential of~$f$ at $\vx \in \dom(f)$ is defined as
$$\partial f(\vx) := \textstyle \partial \big(f(\cdot) + \frac{\rho}{2}\|\cdot\|^2\big)(\vx) - \rho \vx.$$
	\end{definition}
}        
  
		\begin{definition}[(near) $\vareps$-KKT point]\label{def: epskkt}
			Given $\varepsilon > 0$, a point ${\vx}\in \cX$ is called an $\varepsilon$-KKT point of~\eqref{problem:main}, if there exist ${\vy} \in \mathbb{R}^n$ and ${\vz} \in \mathbb{R}_+^m$ such that 
            \begin{equation}\label{eq: epskkt}
			\begin{aligned}
				\max\bigg\{ &\dist(\mathbf{0}, \partial F(\vx) + J_\vg({\vx})^\top {\vz} + \vA^\top {\vy} ), \\ &\quad \sqrt{\|\vA\vx - \vb\|^2 + \|[\vg(\vx)]_+ \|^2},\  \sum_{i=1}^m |{z}_i g_i({\vx}) |\bigg\}
				\leq \varepsilon,
			\end{aligned}        
            \end{equation}
		where 
        $J_{\vg}({\vx})$ denotes the Jacobian of $\vg$ at ${\vx}$. We say that $\overline{\vx}$ is a near $\varepsilon$-KKT point of~\eqref{problem:main}, if it is $\varepsilon$-close to an $\varepsilon$-KKT point $\vx$, i.e., $\|\overline\vx-\vx\|\le \vareps$.
		\end{definition}
  
  The three quantities in~\eqref{eq: epskkt} respectively measure the violation to the dual feasibility (DF), primal feasibility (PF), and complementary slackness (CS) conditions in the KKT system \cite{bazaraa2006nonlinear}.

		\subsection{Organization}
		The rest of this paper is organized as follows.
		Some preparatory lemmas are shown in Sect.~\ref{sec: ConveAnalysis}. 
Iteration complexity results are established in  Sect.~\ref{Sec: ConvAnalysisII} for three cases of $f$, and we present our numerical experiments in Sect.~\ref{sec: NumericalExp}. Finally, the paper is concluded in Sect.~\ref{sec:conclusion}.
		
		\section{Preparatory Lemmas} \label{sec: ConveAnalysis}		
		In this section, we show some results that 
		will be used in Sect.~\ref{Sec: ConvAnalysisII} to establish 
		iteration complexity results of  
		DPALM.  		
		Throughout the paper, we make the following assumptions.
		\begin{assumption}\label{Assump1}
		In~\eqref{problem:main}, the domain of $F$, denoted by $\mathcal{X}:=\dom(F)$, is compact, and its diameter is denoted by $D = \max_{\vx_1,\vx_2 \in \cX} \|\vx_1-\vx_2\| < \infty$. 
        Moreover,  $f$ is $\rho$-weakly convex on $\cX$ for some $\rho>0$, and $h$ is a closed convex function {that admits an efficient proximal mapping.}
		\end{assumption}
        
		\begin{assumption} \label{Assump2}
			For each $i = 1, \dots, m$, $g_i$ in~\eqref{problem:main} is convex and $L_g$-smooth on an open set $\cU$ containing $\mathcal{X}$, i.e., $\left\|\nabla g_i(\vx)-\nabla g_i(\vy)\right\| \leq L_{g}\|\vx-\vy\|\text{ for all }$ $\vx, \vy \in \mathcal{U}$.  
		\end{assumption}
		
		\begin{assumption}\label{Assump3}
			There exists ${\vx}_{\mathrm{feas}} \in \mathrm{int}\left(\mathcal{X}\right)$ such that $\vA \vx_{\mathrm{feas}} = \vb$ and \\
			$g_{\min} := \min_{i\in[m]} -g_i(\vx_{\mathrm{feas}}) >0.$ 
		\end{assumption}
  \begin{assumption}\label{Assump5}
There exists $r_h>0$ such that $\partial h( \vx) \subseteq \mathcal{N}_{\mathcal{X}}( \vx) + \mathcal{B}_{r_h} $, and $\partial h(\vx) \neq \emptyset$ for all $\vx \in \mathcal{X}$.
		\end{assumption}
Under Assumptions~\ref{Assump1} and~\ref{Assump2}, {$F$ is bounded on $\cX$, i.e., $\max_{\vx\in \cX}|F(\vx)| < \infty$}, and there must exist a positive constant $B_g$ such that
		\begin{equation} \label{eq: boundg}
			\begin{aligned}
			&\max \left\{\left|g_i(\vx)\right|,\left\|\nabla g_i(\vx)\right\|\right\} \leq B_{g}, \quad |g_i({\vx}) - g_i(\widetilde{\vx})| \leq B_g \|{\vx} - \widetilde{\vx} \|, 
            \\ & \quad \quad \quad \quad  \text{ for all }  {\vx}, \widetilde{\vx}  \in \mathcal{X},  i \in [m].
            \end{aligned}
		\end{equation}
If $\cX$ is a subset of a subspace, we can replace the condition of ${\vx}_{\mathrm{feas}} \in \mathrm{int}\left(\mathcal{X}\right)$ in Assumption~\ref{Assump3} to ${\vx}_{\mathrm{feas}} \in \mathrm{relint}\left(\mathcal{X}\right)$, where $\mathrm{relint}\left(\mathcal{X}\right)$ denotes the relative interior of $\cX$. All our results will still hold.
{  
Assumption~\ref{Assump5} has appeared 
in literature 
such as in \cite{lin2022complexity}. It holds 
when \( h \) is the sum of a convex Lipschitz continuous \emph{soft} regularizer and the indicator function of the \emph{hard} constraint set \( \cX \), such as $h(\vx) = \lambda\|\vx\|_1$ if $\vx\in [l_1, u_1]\times\cdots \times [l_d, u_d]$ and $+\infty$ otherwise, where $\lambda>0$, and $l_i$'s and $u_i$'s are real numbers. 
}
        

We show the boundedness of $\{\vp^k:=(\vy^k, \vz^k)\}$ in the lemma below.  
Its proof is given in Appendix~\ref{Appendix: multipbound}.
		\begin{lemma}
			\label{lem:boundxz}
			Under Assumptions~\ref{Assump1}--\ref{Assump5}, let $\{\vx^k, \vy^k,\vz^k\}$ be generated by Alg.~\ref{alg:ialm} such that \eqref{eq:approx-cond} holds for each $k\ge0$. 
If $\sum_{k\ge0} \beta_k \delta_k < \infty$, then 
\begin{align}
\label{eq:bpini}
    \|\vp^k\| \le B_0 + \sum_{j = 0}^{k-1} \widetilde{\delta}_j \leq B_p:=B_0 + \sum_{j \ge 0} \widetilde{\delta}_j \text{{, }for all } k \geq 0,
\end{align}
where \begin{align}
    &B_0 := \max \{\|\vp^0\|, \theta\}, \label{eq:b0}\\
    &\widetilde{\delta}_j := \alpha_j \delta_j(\|\vA\| + \sqrt{m} B_g), \label{eq:delta}\\
    &\theta := \frac{\overline{B}}{\min\left\{ \tau \sqrt{\lambda_{\min}^+(\vA\vA^\top)} , \frac{g_{\min}}{2} \right\}}, \label{eq:theta}\\
    &\overline{B} :=  D(r_h + \overline{B}_f)+2\rho D^2 + D+ \tau(1 + r_h + \overline{B}_f + 2\rho D ),\\
    &\tau := \min\left\{\frac{g_{\min}}{2\sqrt{m}B_g},
        \dist(\vx_{\mathrm{feas}},
        \mathrm{bd}(\cX))\right\}.
\end{align}
		\end{lemma}
  
 The next lemma gives the bound for the violation to PF. Its proof is given in Appendix~\ref{section: primalboundproof}.
\begin{lemma}
		\label{lem:tilderhodeltak}
		Under Assumptions of Lemma~\ref{lem:boundxz} and with $\vareps_k \le \sqrt{\frac{\rho}{2\beta_k}}$ for all $k\ge0$, it holds
  \begin{align}\label{eq: diffL}
       {\left\|\vA {\vx}^{k+1}- \vb\right\|^2 + \left\|\left[\vg\left( {\vx}^{k+1}\right)\right]_+\right\|^2} \leq \frac{C_p^2}{\beta_k^2}
  \end{align}
  for all $k\ge 0$, where 
		\begin{align}
                & C_p := 2\widehat{C}_p + \sum_{j=0}^{\infty}\beta_j \delta_j(\|\vA\| + \sqrt{m} B_g) < \infty, \label{eq:cp}\\
			&   \widehat{C}_p := 2\left({B_p} + \sqrt{ \frac{{Q}^2}{g_{\min}^2 } + Q^2\|( \vA \vA^{\top})^{\dagger} \vA\|^2C_1^2}\right)+1,\label{eq:C-P-Case1}\\
			&    Q := D(\overline{B}_f + 2\rho D + r_h), \quad  \quad C_1 := \frac{1}{D} + \frac{1}{\mathrm{dist}( \vx_{\mathrm{feas}},\mathrm{bd}(\mathcal{X}))} + \frac{B_{g}}{g_{\min}}. \label{eq:Q-C-1-Case1}
		\end{align}
	\end{lemma}
		
		\begin{lemma} \label{lem: CompSlack}
			Suppose Assumptions of Lemma~\ref{lem:boundxz} hold. Then for any $k\ge 0$ and any $\vx\in \mathcal{X}$, it holds
			\begin{equation}\label{eq:CS-k+1-x} 
		 	\sum_{i=1}^m\left|[ z_i^k + \beta_k g_i\left( \vx\right)]_+g_i\left( \vx\right)\right| \leq  \frac{B_p^2}{\beta_k} + \frac{5\beta_k}{4}\sum_{i =1}^m [g_i(\vx)]^2_+,
			   \end{equation} where $B_p$ is defined in Lemma~\ref{lem:boundxz}.
   In addition, it holds 
   \begin{equation}\label{eq:CS-k+1} 
     \sum_{i=1}^m\left|[ z_i^k + \beta_k g_i\left( \vx^{k+1}\right)]_+g_i\left( \vx^{k+1}\right)\right| \leq \frac{1}{\beta_k}\left(B_p^2 + \frac{5C_p^2}{4}\right).
   \end{equation}
		\end{lemma}
		\begin{proof} 
	For a given $\vx$, let $J_+ := \left\{i : g_i\left( \vx\right) \geq 0\right\},$ $ J_- := \left\{i : g_i\left( \vx\right) < 0\right\},$ \\$ J_3^k := \{i: -\frac{z_i^k}{\beta_k} \geq g_i\left(\vx\right)\}, $ and $J_4^k := \{i: -\frac{z_i^k}{\beta_k} < g_i\left(\vx\right)\}.
			$
			We then have
			\begin{align}\label{eq:comp-bd}
				&\quad\,\,\sum_{i=1}^m\left|[ z_i^k + \beta_k g_i\left( \vx\right)]_+g_i\left( \vx\right)\right| \notag \\
				&= \sum_{i\in J_4^k \cap J_+}\left( z_i^k + \beta_kg_i\left( \vx\right)\right) g_i\left( \vx\right) - \sum_{i\in J_4^k \cap J_-}\left( z_i^k + \beta_kg_i\left( \vx\right)\right) g_i\left( \vx\right) \notag\\
				&{\leq} \sum_{i\in J_4^k \cap J_+}\left( z_i^k g_i\left(\vx\right) + \beta_kg_i^2\left( \vx\right)\right) + \sum_{i\in J_4^k \cap J_-}\frac{ (z_i^k)^2}{\beta_k} \notag \\
				&\leq \sum_{i\in J_4^k \cap J_+}\frac{\left(z_i^k\right)^2}{\beta_k}  + \frac{\beta_k}{4} \sum_{i \in J_4^k \cap J_+} g_i^2\left(\vx\right) + \beta_k \sum_{i\in J_4^k \cap J_+} g_i^2\left( \vx\right) + \sum_{i\in J_4^k \cap J_-}\frac{\left( z_i^k\right)^2}{\beta_k} \notag\\
				&= \frac{1}{\beta_k} \sum_{i \in J_4^k} \left(z_i^k\right)^2 + \frac{5\beta_k}{4}\sum_{i\in J_4^k \cap J_+} g_i^2\left(\vx\right) \leq \frac{1}{\beta_k} \sum_{i =1}^m \left(z_i^k\right)^2 + \frac{5\beta_k}{4}\sum_{i =1}^m [g_i(\vx)]^2_+ ,
			\end{align}
			where the first inequality holds because $$-( z_i^k + \beta_kg_i( \vx)) g_i( \vx)\leq -z_i^kg_i( \vx)\leq \frac{( z_i^k)^2}{\beta_k}$$ for all $i\in J_4^k \cap J_-$, the second inequality is from Young's inequality, 
			the second equality is obtained by combining $J_+$ and $J_-$, and the last inequality holds by definition of $J_+$. We can now prove~\eqref{eq:CS-k+1-x} by applying Lemma~\ref{lem:boundxz} to~\eqref{eq:comp-bd}.
           
            When $\vx = \vx^{k+1}$, we use~\eqref{eq: diffL} to further bound $\sum_{i =1}^m [g_i(\vx)]^2_+$ and complete the proof. 
		\end{proof}
		
	The next lemma will be used to show the cumulative change of the AL functions in Lemma~\ref{lem: ProxErr}. 
	Its proof is given in Appendix~\ref{AppendixA}. 
		\begin{lemma}
			\label{lem:boundggg}
			Suppose Assumptions of Lemma~\ref{lem:tilderhodeltak} hold. For any integer $K>0$,
			it holds that 
			\begin{equation}\label{eq:lem:boundggg}
			\begin{aligned}
				&  \sum_{k=0}^{K-1}\left(\frac{\beta_{k+1}}{2}\left\|\left[\vg\left(\vx^{k+1}\right)+\frac{ \vz^{k+1}}{\beta_{k+1}}\right]_+\right\|^2 - \frac{\beta_k}{2}\left\|\left[\vg\left(\vx^{k+1}\right)+\frac{ \vz^k}{\beta_k}\right]_+\right\|^2\right)\\
				\leq  &   (C_p^2 + B_p^2) \sum_{k=0}^{K-1} \frac{\beta_{k+1}-\beta_k}{2\beta_k^2} + \left(\frac{3}{2}C_p + B_p \right)\sum_{k=0}^{K-1} \frac{v_k}{\beta_k},
			\end{aligned}
			\end{equation}
			where $C_p$ is defined in Lemma~\ref{lem:tilderhodeltak}. In addition, 
			\begin{equation}\label{eq:lem: yyybound}
   \begin{aligned}
			&\sum_{k=0}^{K-1} \left\langle \vy^{k+1}- \vy^k, \vA \vx^{k+1}-\vb \right\rangle \leq C_p\sum_{k=0}^{K-1} \frac{v_k}{\beta_k},\\
			&\sum_{k=0}^{K-1} \frac{\beta_{k+1}-\beta_k}{2} \left\|\vA \vx^{k+1}- \vb\right\|^2 \leq C_p^2\sum_{k=0}^{K-1} \frac{\beta_{k+1}-\beta_k}{2\beta_k^2}.
   \end{aligned}
			\end{equation}
		\end{lemma}

  \begin{remark}
  \label{rem:1}
      The summability of the terms on the RHS of~\eqref{eq:lem:boundggg} and~\eqref{eq:lem: yyybound} depends on the choice of $v_k$ and $\beta_k$. 
      Below, we give two choices of these parameters to ensure the summability.
      \begin{itemize}
          \item One choice that will be used to obtain our complexity result is 
          \begin{equation}\label{eq:choice-v-beta}
          \beta_k = \beta_0 \sqrt{k+1}, \quad v_k = \frac{v_0}{\sqrt{k+1}(\log ({k+1}))^2}, \text{ for all } k\ge1,
          \end{equation} 
          for some positive constants $\beta_0$ and $v_0$. In this case, we have 
          \begin{equation}
              \begin{aligned}\label{lem: ineq: seriesbound}
                        \sum_{k=0}^{K-1}\frac{\beta_{k+1}-\beta_k}{\beta_k^2}
				&= \sum_{k=0}^{K-1} \frac{1}{\beta_0(k+1)(\sqrt{k+2} + \sqrt{k+1})} \\
                    &\leq \frac{1}{2\beta_0} + \int_1^{\infty} \frac{1}{2\beta_0 x^{\frac{3}{2}}}dx = \frac{3}{2\beta_0},
              \end{aligned}
          \end{equation}
          \begin{equation}
              \begin{aligned}\label{eq:bd-v-to-beta}
                   		 \sum_{k=0}^{K-1} \frac{v_k}{\beta_k} 
				 &\leq \frac{v_0}{\beta_0} \left(1 + \frac{1}{2(\log 2)^2} + \int_2^\infty \frac{1}{x (\log x)^2}\right)\\
                    &= \frac{v_0}{\beta_0} \left(1 + \frac{1}{2(\log 2)^2} + \frac{1}{\log 2}\right) \le \frac{4 v_0}{\beta_0}.
              \end{aligned}
          \end{equation}
          \item Another choice is $\beta_k = \beta_0(k+1)(\log(k+1))^2, v_k = v_0,$ for all $ k\ge1$ and for some $\beta_0 > 0$ and $ v_0 > 0$. In this case, the RHS of~\eqref{eq:lem:boundggg} and~\eqref{eq:lem: yyybound} is still summable. In addition, by~\eqref{eq: diffL}, we have that if $v_0 \ge C_p$, then $\alpha_k = \beta_k, $ for all $ k\ge0$, i.e., we can take a full (instead of damped) dual stepsize. 
          By this choice, our analysis can still go through. However, the resulting iteration complexity will be $\mathcal{O}\left(\sum_{k=0}^{K-1}\sqrt{\beta_k}\right)= \widetilde{\mathcal{O}}(\vareps^{-3})$ as $K=\Theta(\vareps^{-2})$, which is worse than $\widetilde{\mathcal{O}}(\vareps^{-\frac{5}{2}})$ obtained by taking the choice in~\eqref{eq:choice-v-beta}.
      \end{itemize}
  \end{remark}
  \begin{lemma}\label{lem: ProxErr}
Under the same assumptions of Lemma~\ref{lem:boundggg} and with $\{\beta_k, v_k\}$ set as in~\eqref{eq:choice-v-beta}, it holds 
			\begin{align}\label{eq:sm-L-value}
             \sum_{k = \widetilde{K}}^{K-1}& \left({{\mathcal{L}}}_{\beta_k}( \vx^k; \vy^k, \vz^k)-{{\mathcal{L}}}_{\beta_k}( \vx^{k+1}; \vy^k, \vz^k) \right) \leq C_{\vx}, \\
             &\text{ for all integers } K > 	\widetilde{K}\ge0, \notag
			\end{align} 
			where
			\begin{align}
            \notag
			    C_{\vx}&:=    2 \max_{\vx \in \mathcal{X}} |F(\vx)| 
			 + \frac{1}{4\beta_0} (14C_p^2 + 8 B_p C_p + 7B_p^2) + (10 C_p + 4B_p)\frac{v_0}{\beta_0} \\
    & \quad \quad \quad    + \frac{\beta_0}{2} \left(\| \vA \vx^{0}-\vb\|^2 +\left\|\left[g( \vx^0)+\frac{ \vz^0}{\beta_0}\right]_+\right\|^2\right) 
     + \left| \langle \vy^0, \vA \vx^0- \vb\rangle \right|
     \label{eq:cx}
			\end{align}
        with $B_p$ and $C_p$ defined in Lemmas~\ref{lem:boundxz} and~\ref{lem:tilderhodeltak} respectively.
\end{lemma}
		\begin{proof}
			From the definition of $\mathcal{L}_{\beta_k}( \vx^k; \vy^k, \vz^k)$, it follows that
			\begin{align*}
   &  \quad\,\, \sum_{k = \widetilde{K}}^{K-1} \left({{\mathcal{L}}}_{\beta_k}( \vx^k; \vy^k, \vz^k)-{{\mathcal{L}}}_{\beta_k}( \vx^{k+1}; \vy^k, \vz^k) \right)\notag \\
				&  = F( \vx^{\widetilde{K}}) - F\left(\vx^{K}\right) +  \left\langle \vy^{\widetilde{K}}, \vA \vx^{\widetilde{K}}- \vb\right\rangle - \left\langle \vy^{K-1}, \vA \vx^{K}-\vb\right\rangle \\
                & \quad +\underbrace{  \sum_{k=\widetilde{K}}^{K-2} \left\langle \vy^{k+1}- \vy^k, \vA \vx^{k+1}-\vb \right\rangle}_{\textbf{term 1}} \label{eq:normDifference_of_x}  +\underbrace{\sum_{k=\widetilde{K}}^{K-2} \frac{\beta_{k+1}-\beta_k}{2} \left\|\vA \vx^{k+1}- \vb\right\|^2}_{\textbf{term 2}} \\
                & \quad + \underbrace{\sum_{k=\widetilde{K}}^{K-2}\left(\frac{\beta_{k+1}}{2}\left\|\left[\vg\left(\vx^{k+1}\right)+\frac{ \vz^{k+1}}{\beta_{k+1}}\right]_+\right\|^2 - \frac{\beta_k}{2}\left\|\left[\vg\left(\vx^{k+1}\right)+\frac{ \vz^k}{\beta_k}\right]_+\right\|^2\right)}_{\textbf{term 3}}  \notag\\
				&  \quad + \underbrace{   \frac{\beta_{\widetilde{K}}}{2}\| \vA \vx^{\widetilde{K}}-\vb\|^2 + \frac{\beta_{\widetilde{K}}}{2}\left\|\left[g( \vx^{\widetilde{K}})+\frac{ \vz^{\widetilde{K}}}{\beta_{\widetilde{K}}}\right]_+\right\|^2}_{\textbf{term 4}} -  \frac{\beta_{K-1}}{2} \| \vA \vx^{K}-\vb\|^2  \\
                & \quad - \frac{\beta_{K-1}}{2}\left \|\left[\vg\left(\vx^{K}\right)  + \frac{ \vz^{K-1}}{\beta_{K-1}}\right]_+\right\|^2
				\notag\\
                &\leq 2 \max_{\vx \in \mathcal{X}} |F(\vx)| + \left|\langle \vy^0, \vA \vx^0- \vb\rangle \right| +  \frac{B_pC_p}{\beta_0} + \frac{B_pC_p}{\beta_0} +  C_p\frac{4 v_0}{\beta_0} + \frac{3C_p^2}{4\beta_0} \\
                & \quad \quad \quad + \frac{3C_p^2 + 3B_p^2}{4\beta_0}
                 + \left(\frac{3}{2}C_p + B_p\right) \frac{4 v_0}{\beta_0} + \frac{\beta_0}{2}\| \vA \vx^{0}-\vb\|^2 \\
                 & \quad \quad \quad + \frac{\beta_{0}}{2}\left\|\left[g( \vx^0)+\frac{ \vz^0}{\beta_0}\right]_+\right\|^2 + \frac{2C_p^2}{\beta_0} + \frac{B_p^2}{\beta_0}.
			\end{align*}
Below we explain how the second inequality is obtained. By Lemma~\ref{lem:boundxz} and~\eqref{eq: diffL}, we have 
  $\langle -\vy^{K-1}, \vA\vx^{K} - \vb \rangle \leq \| \vy^{K-1}\| \| \vA\vx^{K} - \vb \| \leq \frac{B_pC_p}{\beta_0}.$ Similarly, it holds $\left\langle \vy^{\widetilde{K}}, \vA \vx^{\widetilde{K}}- \vb\right\rangle \leq \left|\langle \vy^0, \vA \vx^0- \vb\rangle\right| +  \frac{B_pC_p}{\beta_0}$ by discussing the cases of $\widetilde{K}=0$ and $\widetilde{K}>0$. 
\textbf{term 1} and \textbf{term 2} are bounded by  using~\eqref{eq:lem: yyybound},~\eqref{lem: ineq: seriesbound}, and~\eqref{eq:bd-v-to-beta}; 
 \textbf{term 3} is bounded by using~\eqref{eq:lem:boundggg},~\eqref{lem: ineq: seriesbound}, and~\eqref{eq:bd-v-to-beta}; \textbf{term 4} is upper bounded by $ \frac{2C_p^2}{\beta_0} + \frac{B_p^2}{\beta_0}$ for $\widetilde{K} > 0$ 
 by~\eqref{eq: diffL}, $\|[\va+\vb]_+\|^2 \leq \|\va+\vb\|^2 \leq 2\|\va\|^2 + 2\|\vb\|^2$, and $\|\vz^k\| \le B_p$ from Lemma~\ref{lem:boundxz}. 
     Adding all the obtained upper bounds and simplifying the summation gives the desired result.
		\end{proof}
		
	\section{Iteration Complexity Results for Three Cases}\label{Sec: ConvAnalysisII}
	In this section, we assume a certain structure on $f$ in~\eqref{problem:main} and specify how to compute $\vx^{k+1}$ in Alg.~\ref{alg:ialm} so that~\eqref{eq:approx-cond} holds.  
Throughout this section, we set $\{\beta_k, v_k\}$ as those in~\eqref{eq:choice-v-beta}.
	
	\subsection{Regularized Smooth Objective} 
	\label{subsec: caseI} 
	We first consider the case where $f$ is smooth. 
	\begin{assumption} \label{Assump4}
	In~\eqref{problem:main}, $f$ is $L_f$-smooth in an open set that contains $\cX$. 
	\end{assumption}
    
Under Assumptions $\ref{Assump1}$ and~\ref{Assump4}, there must exist a positive constant $B_f$ such that 
$ \|\nabla f(\vx)\| \leq B_f, $ for all $ \vx \in \mathcal{X}.$
At the $k$-th iteration of Alg.~\ref{alg:ialm}, 
we 
apply Nesterov's APG method, i.e., Alg.~\ref{alg:acceleratedNesterov} in Appendix~\ref{AppendixB}, to the following subproblem
	\begin{equation}\label{eq:subproblem1}
		\begin{aligned}
			\min_{ \vx\in \mathcal{{X}}} \,\, \widetilde{\mathcal{L}}_{\beta_k}( \vx; \vy^k, \vz^k) = \widetilde{f}^k(\vx)+\widetilde{h}^k(\vx),
		\end{aligned}
	\end{equation}
	where $\widetilde{\mathcal{L}}_{\beta_k}$ is defined in~\eqref{eq:aug-fun2}, and we set 
	\begin{align} 
		\label{eq:tildefh}
		&\widetilde{f}^k(\vx) = {\mathcal{L}}_{\beta_k}( \vx; \vy^k, \vz^k) + \frac{\rho}{2}\|\vx - \vx^k \|^2 - h(\vx),\quad 
		\widetilde{h}^k(\vx) = h(\vx) + \frac{\rho}{2}\|\vx-\vx^k\|^2.
	\end{align}
By \cite[Eqn.~(3.10)]{li2021augmented}, we have the results in the next lemma.
  \begin{lemma} \label{lem: LftildecaseI}
      The function $\widetilde{f}^k$ is $L_{\widetilde{f}^k}$-smooth with $$L_{\widetilde{f}^k} =  L_{f} + \rho + \sqrt{m}L_gB_p +   \beta_k\big(\|\vA^\top\vA\| 
      + m B_g(B_g + L_g)\big).$$ In addition, $\widetilde{f}^k$ is convex, and $\widetilde{h}^k$ is $\rho$-strongly convex.
  \end{lemma}
	
Given $\vareps > 0$, we compute $\vx^{k+1}$ by Alg.~\ref{alg:acceleratedNesterov} as a near-stationary point of subproblem~\eqref{eq:subproblem1} such that 
	\begin{equation}\label{eq:terminate1}
  \dist\left(\mathbf{0},\partial_{\vx}\widetilde{\mathcal{L}}_{\beta_k}( \vx^{k+1}; \vy^k, \vz^k)\right) \leq \varepsilon_k := \min\left\{\frac{\varepsilon}{8}, \sqrt{\frac{\rho}{2\beta_k}}, 1 \right\}, \text{ for all }  k\ge0.
	\end{equation}
Then $\vx^{k+1}$ satisfies~\eqref{eq:approx-cond} with $\overline{f}=f$, $\overline{B}_f=B_f$, $\widehat{\vx}^{k+1} = \vx^{k+1}$, and $\delta_k = 0$. Thus $\sum_{k\ge0}\beta_k\delta_k=0$, and all the lemmas in Sect.~\ref{sec: ConveAnalysis} hold in this case.

  The following lemma shows how to achieve a desired bound on the stationarity at the iterates.
 
		\begin{lemma} \label{lem: dualFeasCaseI}
		Given $\vareps > 0$,	under Assumptions~\ref{Assump1}-\ref{Assump4},
			 let $\{\vx^k, \vy^k, \vz^k\}$ be generated by Alg.~\ref{alg:ialm} such that the condition in~\eqref{eq:terminate1} is satisfied.
			Then for $K_2 := \left \lceil 5 C_{\vx}\rho \varepsilon^{-2}\right \rceil$ and any integer $ \widetilde{K}_1 \geq 0$, where  
			$C_{\vx}$ is defined in Lemma~\ref{lem: ProxErr}, it must hold that $$\underset{\widetilde{K}_1 \leq k \leq \widetilde{K}_1 + K_2-1}\min \dist(\mathbf{0}, \partial_{\vx}\mathcal{L}_{\beta_k}(\vx^{k+1}; \vy^{k}, \vz^{k})) \leq \varepsilon.$$
		\end{lemma}
		\begin{proof}
                Denote $\vx^{k+1}_*=\argmin_{\vx}\widetilde{\mathcal{L}}_{\beta_k}\left( \vx; \vy^k, \vz^k\right)$. By~\eqref{eq:terminate1}, there is \\ $\vxi\in\partial_{\vx}\widetilde{\mathcal{L}}_{\beta_k}( \vx^{k+1}; \vy^k, \vz^k)$ such that $\|\vxi\| \le \vareps_k$.
Then we have
	\begin{align}
			&  {\mathcal{L}}_{\beta_k}( \vx^{k+1}; \vy^k, \vz^k) +\rho \|\vx^{k+1}-\vx^{k}\|^2=\widetilde{\mathcal{L}}_{\beta_k}( \vx^{k+1}; \vy^k, \vz^k) \notag \\
			\leq & \widetilde{\mathcal{L}}_{\beta_k}( \vx^{k+1}_*; \vy^k, \vz^k)+ \frac{ \varepsilon_k^2}{\rho}\notag\\ 
			\leq &   \widetilde{\mathcal{L}}_{\beta_k}( \vx^k; \vy^k, \vz^k) - \frac{\rho}{2} \left\|\vx^{k+1}_* - \vx^k\right\|^2 + \frac{ \varepsilon_k^2}{\rho} \notag \\
			= &   {\mathcal{L}}_{\beta_k}( \vx^k; \vy^k, \vz^k) - \frac{\rho}{2} \left\|\vx^{k+1}_* - \vx^k\right\|^2 + \frac{ \varepsilon_k^2}{\rho} \notag \\
			\leq& {\mathcal{L}}_{\beta_k}( \vx^k; \vy^k, \vz^k) - \frac{\rho}{4} \left\|\vx^k - \vx^{k+1}\right\|^2 + \frac{\rho}{2}\left\|\vx^{k+1} - \vx^{k+1}_* \right\|^2 + \frac{ \varepsilon_k^2}{\rho} \notag \\
			\leq &   {\mathcal{L}}_{\beta_k}( \vx^k; \vy^k, \vz^k) - \frac{\rho}{4} \left\|\vx^k - \vx^{k+1}\right\|^2 + \frac{\rho}{2}\frac{\|\vxi - \mathbf{0}\|^2}{\rho^2} + \frac{ \varepsilon_k^2}{\rho} \notag \\
			\leq & {\mathcal{L}}_{\beta_k}( \vx^k; \vy^k, \vz^k) - \frac{\rho}{4} \left\|\vx^k - \vx^{k+1}\right\|^2 + \frac{ 3\varepsilon_k^2}{2\rho},\label{eq:change-L-beta-k}
		\end{align}
		where in the first inequality, we have used~\eqref{eq:obj-opt-error} and $\overline{\mathcal{L}}_{\beta_k}=\widetilde{\mathcal{L}}_{\beta_k}$, the second inequality follows from the $\rho$-strong convexity of $\widetilde{\mathcal{L}}_{\beta_k}$, the fourth one holds by the $\rho$-strong convexity of $\widetilde{\mathcal{L}}_{\beta_k}(\cdot; \vy^k, \vz^k)$ and because {$\mathbf{0} \in \partial_\vx \widetilde{\mathcal{L}}_{\beta_k}(\vx_*^{k+1}; \vy^k, \vz^k)$ and $\vxi\in\partial_{\vx}\widetilde{\mathcal{L}}_{\beta_k}( \vx^{k+1}; \vy^k, \vz^k)$}, and the last inequality holds since $\|\vxi\| \le \vareps_k$.

  Using the sum of the inequality in~\eqref{eq:change-L-beta-k} over $ k = \widetilde{K}_1, \dots, \widetilde{K}_1 + K_2 -1$, we obtain
            \begin{align}\notag
                &\quad\,\,\min_{\widetilde{K}_1 \leq k \leq \widetilde{K}_1 + K_2 - 1} \|\vx^{k+1}-\vx^k\|^2  \notag \\   
                &\leq \frac{1}{K_2}\sum_{k = \widetilde{K}_1}^{\widetilde{K}_1+K_2-1}\|\vx^{k+1}-\vx^k\|^2 \notag \\ \label{eq: sufficientdescentsum}
                &  \leq 
            \frac{4}{5 \rho K_2}\sum_{k = \widetilde{K}_1}^{\widetilde{K}_1+K_2-1} \left({{\mathcal{L}}}_{\beta_k}( \vx^k; \vy^k, \vz^k)-{{\mathcal{L}}}_{\beta_k}( \vx^{k+1}; \vy^k, \vz^k)+\frac{3\varepsilon_k^2}{2\rho} \right).
            \end{align}
			Since $\varepsilon_{k} \leq \frac{\varepsilon}{8}, $ for all $  k\ge0$,
			it holds that $\frac{4}{5\rho K_2}\sum_{k = \widetilde{K}_1}^{\widetilde{K}_1 + K_2 -1} \frac{3\varepsilon_k^2}{2\rho} \leq \frac{\varepsilon^2}{32\rho^2} \text{ for all } \\ \widetilde{K}_1 \geq 0.$ Hence, we have from Lemma~\ref{lem: ProxErr} and~\eqref{eq: sufficientdescentsum} that  $$\underset{\widetilde{K}_1 \leq k \leq \widetilde{K}_1 + K_2-1}{\text{min}} \left\|\vx^{k+1}-\vx^k\right\| \leq \sqrt{\frac{4C_{\vx}}{5\rho K_2} + \frac{\varepsilon^2}{32\rho^2}}.$$ 
   Now notice $\partial_{\vx}\mathcal{L}_{\beta_k}(\vx^{k+1}; \vy^{k}, \vz^{k})=\partial_{\vx}\widetilde{\mathcal{L}}_{\beta_k}(\vx^{k+1}; \vy^{k}, \vz^{k})- 2\rho \left( \vx^{k+1} - \vx^{k}\right)$. We  have 
			\begin{align*}
				&\quad\,\,\underset{\widetilde{K}_1 \leq k \leq \widetilde{K}_1 + K_2-1}{\text{min}} \dist\left(\mathbf{0}, \partial_{\vx}\mathcal{L}_{\beta_k}(\vx^{k+1}; \vy^{k}, \vz^{k})\right)\\
    &= \underset{\widetilde{K}_1 \leq k \leq \widetilde{K}_1 + K_2 - 1}{\text{min}} \dist\left(\mathbf{0}, \partial_{\vx}\widetilde{\mathcal{L}}_{\beta_k}(\vx^{k+1}; \vy^{k}, \vz^{k})-2\rho \left( \vx^{k+1} - \vx^{k}\right)\right) \\
				&\overset{\eqref{eq:terminate1}}\leq \underset{\widetilde{K}_1 \leq k \leq \widetilde{K}_1 + K_2 - 1}{\text{min}}\left(\varepsilon_k+ 2\rho \left\| \vx^{k+1} - \vx^{k}\right\|\right) \leq \frac{\varepsilon}{8}+ 2 \rho \sqrt{\frac{4C_{\vx}}{5\rho K_2} + \frac{\varepsilon^2}{32\rho^2}} \leq \varepsilon,
			\end{align*}
			where the last inequality holds because $K_2 \ge 5 C_{\vx}\rho \varepsilon^{-2}$. This completes the proof.
		\end{proof}
 Now we are ready to show the outer iteration complexity of  Alg.~\ref{alg:ialm} when $f$ 
 satisfies Assumption~\ref{Assump4}. 
		\begin{theorem}[Outer iteration complexity result I] \label{th:caseIOuterResult}
			Given $\vareps>0$, under Assumptions~\ref{Assump1}-\ref{Assump4},
			let $\{\vx^k,\vy^k, \vz^k\}$ be generated by Alg.~\ref{alg:ialm} such that~\eqref{eq:terminate1} holds, where $\{\beta_k, v_k\}$ are set as in~\eqref{eq:choice-v-beta}.
			Then for some $k < K = K_1 + K_2$, $\vx^{k+1}$ is an $\varepsilon$-KKT point of problem~\eqref{problem:main}, 
			where $K_1 := \left\lceil \max\left\{ \frac{C_p^2}{\beta_0^2\varepsilon^2}, \frac{(4B_p + 5C_p^2)^2}{16\beta_0^2\varepsilon^2} \right\}\right\rceil$, $K_2 := \left \lceil 5 C_{\vx}\rho \varepsilon^{-2}\right \rceil$, 
			$C_p$ is given in Lemma~\ref{lem:tilderhodeltak}, $B_p$ is given in Lemma~\ref{lem:boundxz}, {and $C_{\vx}$ is defined in Lemma~\ref{lem: ProxErr}}.
		\end{theorem}
		\begin{proof}
    From Lemma~\ref{lem:tilderhodeltak}, we have~\eqref{eq: diffL} and thus by the update of $\beta_k$, it holds 
			\begin{align}\label{eq: priamlcaseI}
		 	\sqrt{{\left\|\vA \vx^{k+1}- \vb\right\|^2 + \left\|\left[\vg\left( \vx^{k+1}\right)\right]_+\right\|^2}} \leq \frac{C_p}{\beta_{K_1}}=\frac{C_p}{\beta_0 \sqrt{K_1+1}}\leq \varepsilon, \text{ for all }  k \geq K_1,
			\end{align}
			where the second inequality holds because 
			$
			K_1 \geq \frac{C_p^2}{\beta_0^2\varepsilon^2}.
			$
			Denote $$ \overline{\vy}^k := \vy^k + \beta_k(\vA\vx^{k+1}-\vb), \text{ and }\overline{\vz}^k := [\vz^k + \beta_k \vg(\vx^{k+1})]_+.$$ 
			Then by the update of $\beta_k$ and~\eqref{eq:CS-k+1}, 
			it holds
			\begin{align}\label{eq: compslackcaseI}
		 	\sum_{i = i}^m |\overline{z}_i^{k}g_i(\vx^{k+1})| \leq& \frac{1}{\beta_{K_1}} \left( B_p^2 + \frac{5C_p^2}{4}\right)\notag \\ = &\frac{1}{\beta_0 \sqrt{K_1+1}}\left( B_p^2 + \frac{5C_p^2}{4}\right) \leq \varepsilon, \text{ for all }  k \geq K_1, 
			\end{align}
			where we have used 
			$
			K_1 \geq \frac{(4B_p + 5C_p^2)^2}{16\beta_0^2\varepsilon^2}
			$
			to obtain the second inequality.
   
                Moreover, notice $ \partial_{\vx}\mathcal{L}_{\beta_k}(\vx^{k+1}; \vy^{k}, \vz^{k}) = \partial_{\vx}\mathcal{L}_0(\vx^{k+1}; \overline{\vy}^{k}, \overline{\vz}^{k})$. Thus  letting $\widetilde{K}_1 = K_1$ in Lemma~\ref{lem: dualFeasCaseI} yields
            \begin{align}\label{eq: dualcaseI}
                \underset{K_1 \leq k \leq K_1 + K_2-1}{\text{min}} \dist\left(\mathbf{0}, \partial_{\vx}\mathcal{L}_{0}(\vx^{k+1}; \overline{\vy}^{k}, \overline{\vz}^{k})\right) \leq \varepsilon. 
            \end{align}
            Now let $k' = \underset{K_1 \leq k \leq K_1 + K_2-1}{\text{arg min}} \dist\left(\mathbf{0}, \partial_{\vx}\mathcal{L}_{0}(\vx^{k+1}; \overline{\vy}^{k}, \overline{\vz}^{k})\right)$. We conclude from~\eqref{eq: priamlcaseI}, \eqref{eq: compslackcaseI}, and~\eqref{eq: dualcaseI} that $\vx^{k'+1}$ is an $\varepsilon$-KKT point of problem~\eqref{problem:main} with multipliers $\overline{\vy}^{k'}$ and $\overline{\vz}^{k'}$ by Definition~\ref{def: epskkt}.
		\end{proof}
    
		To obtain the total complexity of Alg.~\ref{alg:ialm} for solving problem~\eqref{problem:main}, we still need to evaluate the number of inner iterations for solving subproblem~\eqref{eq:subproblem1} by using Alg.~\ref{alg:acceleratedNesterov} as the subroutine,  such that~\eqref{eq:terminate1} is met.  We have the following lemma directly from Theorem~\ref{Th: TboundNesterov}. 
		
		\begin{lemma}
			\label{lem:innercase1}
			Under Assumptions~\ref{Assump1}-\ref{Assump4} and for a given $\varepsilon_k>0$,  
			 Alg.~\ref{alg:acceleratedNesterov} with $\gamma_u=2$ applied to subproblem~\eqref{eq:subproblem1}  can find a solution $\vx^{k+1}$ that satisfies the criteria in~\eqref{eq:terminate1} within
			\begin{align}\label{eq: innerComplexityOurProblem}
			 	T^k_1 = \left\lceil\max\left\{\frac{1}{\log 2}, {2}{\sqrt{\frac{ L_{\widetilde{f}^k}}{\rho}}}\right\} \log \frac{9 DL_{\widetilde{f}^k}\sqrt{\frac{L_{\widetilde{f}^k}}{\rho}}}{\varepsilon_k}\right\rceil + 1
			\end{align}
			 iterations, where $L_{\widetilde{f}^k} := C_3 +   \sqrt{k+1}\beta_0 C_2$ is the Lipschitz constant of $\nabla \widetilde{f}^k$ with $C_2 := \|\vA^\top\vA\| + m B_g(B_g + L_g)$ and $C_3 = L_{f} + \rho + \sqrt{m}L_gB_p$. 
		\end{lemma}
Combining Theorem~\ref{th:caseIOuterResult} and Lemma~\ref{lem:innercase1}, we are ready to show the total complexity of  Alg.~\ref{alg:ialm}, {namely, the total number of gradient and function evaluations}.
		\begin{theorem}[Total complexity result I]\label{thm:totalcase1}
For a given $\vareps>0$,	under Assumptions~\ref{Assump1}-\ref{Assump4}, Alg.~\ref{alg:ialm}, with $\{\beta_k, v_k\}$ set as in~\eqref{eq:choice-v-beta} and with Alg.~\ref{alg:acceleratedNesterov} as a subroutine to compute $\vx^{k+1}$, can produce an $\varepsilon$-KKT point of problem~\eqref{problem:main} by $T^{\mathrm{total}}_1$ proximal gradient steps. Here, $T^{\mathrm{total}}_1$ satisfies 
			\begin{dmath}\label{eq:total-complex-I}
			T^{\mathrm{total}}_1  \leq 2K+K\left(2\sqrt{\frac{C_3}{\rho}} + 2\sqrt{\frac{\sqrt{K} \beta_0C_2}{\rho}}
                + \frac{1}{\log 2}\right)\log \left( 9 {\varepsilon^{-1}_K} D\sqrt{\frac{1}{\rho}}(C_3 +  \sqrt{K}\beta_0 C_2)^{\frac{3}{2}}\right),
			\end{dmath}
			where $K$ is given in Theorem~\ref{th:caseIOuterResult}, $\vareps_k$ is defined in~\eqref{eq:terminate1}, and $C_2$ and $C_3$ are given in Lemma~\ref{lem:innercase1}. 
		\end{theorem}
		\begin{proof}
	By Theorem~\ref{th:caseIOuterResult}, Alg.~\ref{alg:ialm} can find an $\varepsilon$-KKT point of problem~\eqref{problem:main} within $K$ outer iterations.
			Hence, by Lemma~\ref{lem:innercase1}, the total number $T^{\mathrm{total}}_1$ of inner iterations 
satisfies 
			\begin{align*}\label{eq: totalcompbasic}
				&   T^{\mathrm{total}}_1 \leq \sum_{k=0}^{K-1} T^k_1 \leq \sum_{k=0}^{K-1} \left(\left\lceil \left(2\sqrt{\frac{ L_{\widetilde{f}^k}}{\rho}}+\frac{1}{\log 2}\right)\log \frac{9L_{\widetilde{f}^k}D\sqrt{\frac{L_{\widetilde{f}^k}}{\rho}}}{\varepsilon_k} \right\rceil + 1 \right) \\
				&\leq     2K + K \left(2\sqrt{\frac{ L_{\widetilde{f}^{K-1}}}{\rho}}+\frac{1}{\log 2}\right)\log \frac{9L_{\widetilde{f}^{K-1}}D\sqrt{\frac{L_{\widetilde{f}^{K-1}}}{\rho}}}{\varepsilon_K}\\
				&\leq    2K \\ &+K\left(2\sqrt{\frac{C_3}{\rho}} + 2\sqrt{\frac{\sqrt{K} \beta_0C_2}{\rho}}+ \frac{1}{\log 2}\right)\log \left( 9 \varepsilon_K^{-1} D\sqrt{\frac{1}{\rho}}(C_3 + \sqrt{K}\beta_0 C_2)^{\frac{3}{2}}\right) ,
			\end{align*}
   where the second inequality comes from $L_{\widetilde{f}^k} \leq L_{\widetilde{f}^{K-1}}, $ for all $k \leq K-1$ by the definition of $L_{\widetilde{f}^k}$ in Lemma~\ref{lem:innercase1}, and $\varepsilon_K \leq \varepsilon_k,$ for all $ k \leq K$ from~\eqref{eq:terminate1} and the update of $\beta_k$, the third inequality holds by $\sqrt{a + b} \leq \sqrt{a} + \sqrt{b}, $ for all $ a, b\geq 0$. This completes the proof.
		\end{proof}
   \begin{remark}
   \label{rem:rhoep}
   {We show how the total complexity  depends on the weak convexity parameter $\rho$ and the target accuracy $\vareps$.
   As seen in Theorem~\ref{th:caseIOuterResult}, \(K\) scales with the constants \(C_p\), \(C_\vx\), $\rho$, and $\vareps$. From~\eqref{eq:cp}, \(C_p\) is governed by \(\widehat{C}_p\), which involves~\(B_p\) and \(Q\) (cf.~\eqref{eq:bpini} and \eqref{eq:Q-C-1-Case1}). Since both of \(B_p\) and \(Q\) grow linearly with  \(\rho\), we obtain \(\widehat{C}_p = \mathcal{O}(1+\rho)\) and consequently \(C_p = \mathcal{O}(1+\rho)\). Similarly,~\eqref{eq:cx} implies \({C}_\vx = \mathcal{O}(1+\rho^2)\). Choose $\beta_0 = \Theta(\frac{B_p + C_p + C_p^2}{\sqrt{C_\vx \rho}})$ and substitute these quantities  into the formula for $K_1, K_2$ and \(K\); we deduce that $K_1$ and $K_2$ are in the same order and \(K = \cO(C_\vx\rho\vareps^{-2})=\mathcal{O}(\rho(1+\rho^2) \varepsilon^{-2})\). Hence, retaining the dominating term $K\sqrt{\frac{\sqrt{K}\beta_0 C_2}{\rho}}$ from \eqref{eq:total-complex-I}, we have 
   the total complexity 
\[
T_1^{\mathrm{total}} = \widetilde{\mathcal{O}}\left(K\sqrt{\frac{\sqrt{K}\beta_0 C_2}{\rho}}\right) = \widetilde{\mathcal{O}}\left(\sqrt{\rho}(1+\rho^3) \varepsilon^{-2.5}\right),
\]
which indicates higher complexity or more difficulty of solving the problem~\eqref{problem:main} to $\vareps$-stationarity as $\rho$ increases.
}
   \end{remark}

		\subsection{Regularized Compositional Objective}\label{subsec: caseII} In this subsection, we make the following structural assumption on the function~$f$ in~\eqref{problem:main}.
		\begin{assumption}
			\label{ass:case2}
		In~\eqref{problem:main}, $f$ is in a compositional form of 
  $f=l\circ \vc$, 
		where $\vc: \mathbb{R}^d \rightarrow \mathbb{R}^p$ is a $L_c$-smooth mapping, 
		i.e.,
			\mbox{$
			\|J_{\vc}(\vx_1)-J_{\vc}(\vx_2)\|_{F} \leq L_c\|\vx_1-\vx_2\|
			$}  for all   $\vx_1, \vx_2 \in \mathbb{R}^d,$
			and $l: \mathbb{R}^p \rightarrow \mathbb{R}$ is a convex, potentially non-smooth, $M_l$-Lipschitz continuous function. 
		\end{assumption}
  
 The next lemma is from \cite[Lemma 4.2]{drusvyatskiy2019efficiency}. Notice that the tightest weak convexity constant $\rho$ can be smaller than $M_lL_c$.

\begin{lemma}
    Under Assumption~\ref{ass:case2}, $f$ is $M_lL_c$-weakly convex.
\end{lemma}

	Without the smoothness of $f$, a FOM may not produce a near-stationary point of problem~\eqref{eq:subproblem1} as claimed in \cite{drusvyatskiy2019efficiency}.
 Hence, we aim to find a point that is close to a near-KKT point of~\eqref{problem:main}, and we utilize the smoothing strategy 
	adopted in~\cite{drusvyatskiy2019efficiency}. Details are described below. 
		
		Given a point $\overline{\vx} \in \mathbf{R}^d$, we define the prox-linear function $f^0:\RR^{d}\mapsto \RR$ and a smoothing function
        $f^{\nu}:\RR^{d}\mapsto \RR$ of $f$ by
		\begin{equation}\label{eq:def-f0-fnu}
			f^0( \vx; \overline{\vx} ) :=l(\vc(\overline{\vx})+J_{\vc}(\overline{\vx})({\vx}-\overline{\vx})), \quad
			f^{\nu}(\vx; \overline{\vx} ) :=l^{\nu}(\vc(\overline{\vx})+J_{\vc}(\overline{\vx})({\vx}-\overline{\vx})), \\
		\end{equation}
		where $l^{\nu}$ is the Moreau envelope of $l$ with $0<\nu\le 1$. 
		Then, $f^{\nu}$ is a smooth function  and $ f^0(\vx; {\vx})=f(\vx)$. 
		Since $\cX = {\dom(F)}$ is bounded, we have from Assumptions~\ref{Assump1},~\ref{Assump3} and~\ref{ass:case2} that 
		\begin{equation}
  \label{eq:jacobinc}
			\|J_{\vc} (\vx)\|_F \leq \|\nabla \vc\|:= \|J_{\vc} (\vx_{\mathrm{feas}})\|_F + L_c D, \text{ for all }  \vx\in \cX.
		\end{equation}
		 The following properties hold directly from Lemmas~2.1 and 3.2 of \cite{drusvyatskiy2019efficiency}.
		\begin{lemma}
			\label{lem:1}
			Let $f^0$ and $f^\nu$ be defined in~\eqref{eq:def-f0-fnu}. It holds that
			\begin{itemize}
				\item[\textnormal{(a)}] $f^{\nu}$ is $M_l \|\nabla \vc\|$ Lipschitz continuous;
				\item[\textnormal{(b)}] $\nabla f^{\nu}$ is $\frac{\|\nabla \vc\|}{\nu}$ Lipschitz continuous;
				\item[\textnormal{(c)}] $0\leq f^0( \vx; \overline{\vx} )-f^{\nu}({\vx} ; \overline{\vx}) \leq \frac{ M_{l}^2\nu}{2}$, $-\frac{\rho}{2}\|{\vx} - \overline{\vx}\|^2\leq f(\vx)-f^0( \vx; \overline{\vx} ) \leq \frac{\rho}{2}\|{\vx} - \overline{\vx}\|^2$.
			\end{itemize}
		\end{lemma}
	With the setting in~\eqref{eq:def-f0-fnu} and the properties in Lemma~\ref{lem:1}, we compute $\vx^{k+1}$ in Alg.~\ref{alg:ialm} by applying Alg.~\ref{alg:acceleratedNesterov} to solve the   subproblem
		\begin{equation}\label{eq:sub-3}
			\min_\vx \widetilde{\cL}^{\nu_k}_{\beta_k}( \vx; \vy^k,\vz^k):
        =\widehat{f}^k(\vx)+\widehat{h}^k(\vx),
		\end{equation}
such that $\vx^{k+1}$ is an ${\varepsilon}_k$-stationary point  of problem~\eqref{eq:sub-3}, i.e.,
\begin{equation}
			\label{eq:terminate2}
            \begin{aligned}
		&\dist(\vzero,\partial_{\vx} \widetilde{\cL}^{\nu_k}_{\beta_k}( \vx^{k+1}; \vy^k, \vz^k)) \leq \varepsilon_k,
            \end{aligned}
		\end{equation}  		
		where 
		\begin{align} 
			\label{eq:tildefh2}
			\widehat{f}^k(\vx) &=    f^{\nu_k}(\vx ; \vx^k) + (\vy^k)^{\top}(\vA 
			\vx-\vb)+ \frac{\beta_k}{2}\|\vA \vx-\vb\|^2+ \frac{\beta_k}{2}\left\|[ \vg( \vx)+\frac{ \vz^k}{\beta_k}]_+\right\|^2 \notag \\ & \quad \quad \quad - \frac{\|\vz^k\|^2}{2\beta_k}, \text{ and}\\
			  \widehat{h}^k(\vx) &= h(\vx) + \frac{\rho }{2}\| \vx-\vx^{k}\|^2.
		\end{align}
  
  By \cite[Eqn.~(3.10)]{li2021augmented} and Lemma~\ref{lem:1}(b), we have the next lemma.
  \begin{lemma}\label{lem: caseIIfLsmoothness}
      Under Assumption~\ref{ass:case2}, $\widehat{f}^k(\vx)$ is $(\frac{\|\nabla \vc\|}{\nu_k} + \sqrt{m}L_gB_p + \beta_k C_2)$-smooth with $C_2$ given in Lemma~\ref{lem:innercase1}. Also,  $\widehat{h}^k(\vx)$ is ${\rho }$-strongly convex, and $\widehat{f}^k(\vx)$ is convex.
  \end{lemma}
  
	In addition, by 
 Lemma~\ref{lem:1}(a), it holds that 
 \begin{align}
     \label{eq:tildebf}
     \|\nabla f^{\nu_k}(\vx ; \vx^k)\|\leq \widetilde{B}_f:=M_l \|\nabla \vc\|, \text{ for all }  \vx \in \mathcal{X}.
 \end{align}
Hence, $\vx^{k+1}$ satisfies~\eqref{eq:approx-cond} with $\overline{f}(\vx) = f^{\nu_k}(\vx;\vx^k) - \frac{\rho }{2}\| \vx-\vx^{k}\|^2$,   $\overline{B}_f=\widetilde{B}_f$, $\widehat{\vx}^{k+1} = \vx^{k+1}$, and $\delta_k = 0$ for all $k\ge 0$. Then $\sum_{k\ge0}\beta_k\delta_k=0$. Thus all the lemmas in Sect.~\ref{sec: ConveAnalysis} hold in this case.
  

            With Lemma~\ref{lem:tilderhodeltak}, we can show the outer iteration complexity of Alg.~\ref{alg:ialm}. Define 
            \begin{equation}
                \begin{aligned}
                    \quad \quad \quad \quad \quad  C_4:=\max\left\{\frac{3}{2}, \frac{3}{2\rho}\sqrt{\|\vA\|^2+mB_g^2}, \frac{4B_g}{\rho} \sqrt{m\beta_0\sqrt{C_5}}, \right. \\ \left. \frac{64}{\rho^2}mB_g^2 \beta_0 \sqrt{\rho C_{\vx}},  \frac{8B_g}{\rho} \sqrt{m\beta_0}, \frac{3}{4\rho}\right\}, \text{ and} \label{eq:def-C-4}
                \end{aligned}
            \end{equation}
            \begin{equation}
                \begin{aligned}
                    C_5:=\max\left\{4C_p^2\beta_0^{-2}, 4(B_p^2+\frac{5C_p^2}{2})^2\beta_0^{-2}\right\}, \label{eq:def-C-5}
                \end{aligned}
            \end{equation}    
  where $C_p$ is given in Lemma~\ref{lem:tilderhodeltak} and $C_{\vx}$  in Lemma~\ref{lem: ProxErr}. Then   
            we set $\vareps_k$ as follows:
            \begin{equation}
			\label{eq:terminate2ep}
            \begin{aligned}
    		&  \varepsilon_k:= \min\left\{\frac{\varepsilon}{16C_4}, \sqrt{\frac{\rho}{2\beta_k}}, 1 \right\}.
            \end{aligned}
		\end{equation} 		
In addition, we define		
		\begin{equation}\label{eq:def-x-hat-+}
		\begin{aligned}
			&\vx^{k}_{+} =  \argmin_{\vx}  \widetilde{\cL}^{0}_{\beta_k}({\vx}; \vy^k,\vz^k), 
		\quad \widehat{\vx}^{k} =\argmin_{\vx}  \widetilde{\cL}^{\nu_k}_{\beta_k}({\vx}; \vy^k,\vz^k),\\ 
			&\mathcal{G}_{\frac{1}{\rho}}\left(\vx^{k}\right) =\rho\left(\vx^{k}_{+}-\vx^{k}\right), \quad	\mathcal{G}_{\frac{1}{\rho}}^{\nu_k}\left(\vx^{k}\right) =\rho(\widehat{\vx}^{k}-\vx^{k}) .
		\end{aligned}
		\end{equation}	
The next lemma is directly from \cite[Theorem 6.5]{drusvyatskiy2019efficiency}.
\begin{lemma}\label{lem:diff-smooth-mv}
Let $\mathcal{G}_{\frac{1}{\rho}}$ and $\mathcal{G}_{\frac{1}{\rho}}^{\nu}$ be defined in \eqref{eq:def-x-hat-+}. It holds   $$\left\|\mathcal{G}_{\frac{1}{\rho}}\left(\vx^{k}\right)\right\| \leq\left\|\mathcal{G}_{\frac{1}{\rho}}^{\nu}\left(\vx^{k}\right)\right\|+\sqrt{{\frac{M_l^2 \nu \rho}{2} }}.$$
\end{lemma}
		In addition, by 
  \cite[Eqn. (4.10)]{drusvyatskiy2019efficiency}, with  
		$\widetilde{\vx}^{k}  =\operatorname{prox}_{ \frac{{\cL}_{\beta_k}( \cdot; \vy^k,\vz^k)}{2 \rho}}\left(\vx^{k}\right)$, it holds 
		\begin{equation}
			\label{eq:importantine}
			\begin{aligned}
			  &\|\widetilde{\vx}^{k}-\vx^{k}\|  \leq \frac{2}{\rho }\big\|\mathcal{G}_{ \frac{1}{\rho} }\left(\vx^{k}\right)\big\|, \ {\cL}_{\beta_k}( \widetilde{\vx}^{k}; \vy^k,\vz^k)  \leq {\cL}_{\beta_k}( {\vx}^{k}; \vy^k,\vz^k) , 
				\\
				&\operatorname{dist}(\vzero , \partial_\vx {\cL}_{\beta_k}( \widetilde{\vx}^{k}; \vy^k,\vz^k))
				 \leq 4\big\|\mathcal{G}_{ \frac{1}{\rho}}\left(\vx^{k}\right)\big\| .
			\end{aligned}
		\end{equation}
		Thus if $\|\mathcal{G}_{ \frac{1}{\rho}}\left(\vx^{k}\right)\|$ is small, $\widetilde{\vx}^{k}$ is nearly stationary for $\min_{\vx} {\cL}_{\beta_k}( \vx; \vy^k, \vz^k)$, and $\vx^{k}$ is close to $\widetilde{\vx}^{k}$. 
Notice that we do not compute $\widetilde{\vx}^{k}$. The sole purpose of introducing $\widetilde{\vx}^{k}$ is to certify the quality of $\vx^{k}$. 
		\begin{theorem}[Outer iteration complexity result II]
			\label{lem:2}
Given $\vareps>0$, under Assumptions~\ref{Assump1}-\ref{Assump5} and~\ref{ass:case2}, let $\{\vx^k, \vy^k, \vz^k\}$ be generated by Alg.~\ref{alg:ialm} with $\{\beta_k, v_k\}$ set as in~\eqref{eq:choice-v-beta}, such that the condition in~\eqref{eq:terminate2} is satisfied, where $\varepsilon_k$ is given in~\eqref{eq:terminate2ep} and $$\nu_{k}=\nu := \min\left\{1,  \frac{\varepsilon^2}{64 C_4^2\rho M_l^2}\right\},$$ with $C_4$ defined in~\eqref{eq:def-C-4}.
			Then for some $k\leq K:=\overline{K}_1+\overline{K}_2$,  
			$\vx^{k}$ is a near $\varepsilon$-KKT point of problem~\eqref{problem:main}, 
			where $\overline{K}_1:=\lceil C_5 \varepsilon^{-2}\rceil$, $\overline{K}_2 := \left \lceil 64\rho C_{\vx}C_4^2 \varepsilon^{-2}\right \rceil$  with $ C_{\vx}$ and $C_5$ given in Lemma~\ref{lem: ProxErr} and~\eqref{eq:def-C-5}, respectively. 
		\end{theorem}
		\begin{proof}
From $f^{\nu}(\vx^k ; \vx^k)\leq f^{0}(\vx^k ; \vx^k) = f(\vx^k)$, it holds $${\cL}_{\beta_k}( \vx^k; \vy^k,\vz^k)\ge\widetilde{\cL}^{\nu}_{\beta_k}( \vx^k; \vy^k,\vz^k).$$ Hence, by the $\rho$-strong convexity of	$\widetilde{\cL}^{\nu}_{\beta_k}(\cdot; \vy^k, \vz^k)$, the definition of $\widehat{\vx}^{k}$ in~\eqref{eq:def-x-hat-+}, and the condition in~\eqref{eq:terminate2},	we have 
			\begin{equation}\label{eq:chain-ineq}
			\begin{aligned}
				{\cL}_{\beta_k}( \vx^k; \vy^k,\vz^k) 
				  &\geq \widetilde{\cL}^{\nu}_{\beta_k}( \widehat{\vx}^{k}; \vy^k,\vz^k)+\frac{\rho}{2 }\left\|\vx^k-\widehat{\vx}^{k}\right\|^2 \\ &\ge \widetilde{\cL}^{\nu}_{\beta_k}( \vx^{k+1}; \vy^k,\vz^k)+\frac{\rho}{2 }\left\|\vx^k-\widehat{\vx}^{k}\right\|^2-\frac{\varepsilon_{k}^2}{\rho},
			\end{aligned}
			\end{equation}
where in the second inequality, we have used~\eqref{eq:obj-opt-error} with $\overline{\mathcal{L}}_{\beta_k}=\widetilde{\mathcal{L}}_{\beta_k}^\nu$.
In addition, by Lemma~\ref{lem:1}(c), it holds			{\small$$f^{\nu}(\vx^{k+1} ; \vx^k)+\frac{\rho}{2}\|\vx^{k+1}-\vx^{k}\|^2\geq f^{0}(\vx^{k+1} ; \vx^k)- \frac{M_l^2\nu}{2}+\frac{\rho}{2}\|\vx^{k+1}-\vx^{k}\|^2 \geq  f(\vx^{k+1})- \frac{M_l^2\nu}{2}.$$}This together with~\eqref{eq:chain-ineq} indicates $${\cL}_{\beta_k}( \vx^k; \vy^k,\vz^k) \ge {\cL}_{\beta_k}( \vx^{k+1}; \vy^k,\vz^k)-\frac{M_l^2\nu}{2}+\frac{\rho}{2 }\left\|\vx^k-\widehat{\vx}^{k}\right\|^2 -\frac{\varepsilon_{k}^2}{\rho}.$$ Hence by~\eqref{eq:def-x-hat-+}, we have 
$${\cL}_{\beta_k}( \vx^k; \vy^k,\vz^k)\ge {\cL}_{\beta_k}( \vx^{k+1}; \vy^k,\vz^k)-\frac{M_l^2\nu}{2}+
\frac{1}{2\rho }\|\mathcal{G}^{\nu}_{\frac{1}{\rho}}(\vx^{k})\|^2-\frac{\varepsilon_{k}^2}{\rho}.	$$
Summing up this inequality over $k$ from $\overline{K}_1$ to ${K-1}$ and using Lemma~\ref{lem: ProxErr} yields
			\begin{align*}
			     &\sum_{k = \overline{K}_1}^{K-1}\|\mathcal{G}^{\nu}_{\frac{1}{\rho}}(\vx^{k})\|^2 \\
            \leq &  
            2\rho\sum_{k=\overline{K}_1}^{K-1}\left({{\mathcal{L}}}_{\beta_k}( \vx^k; \vy^k, \vz^k)-{{\mathcal{L}}}_{\beta_k}( \vx^{k+1}; \vy^k, \vz^k)+\frac{M_l^2\nu}{2}+ \frac{\varepsilon_{k}^2}{\rho} \right) \\ \leq &    2{{\rho}}C_{\vx}+2{{\rho}}\sum_{k=\overline{K}_1}^{K-1}\left(\frac{M_l^2\nu}{2}+ \frac{\varepsilon_{k}^2}{\rho}\right), \text{ for all }  K > \overline{K}_1.
			\end{align*}
			The above inequality together with $\nu \le \frac{\varepsilon^2}{64 C_4^2\rho M_l^2}$ and $\varepsilon_k\le \frac{\varepsilon}{16C_4}$ implies that for all $k \geq 0$,
			\begin{equation}
				\label{eq:g1}
				\min_{\overline{K}_1\leq k\leq K-1} \|\mathcal{G}^{\nu}_{\frac{1}{\rho}}(\vx^{k})\|\leq   \sqrt{\frac{2\rho C_{\vx}}{K-\overline{K}_1} + \frac{\varepsilon^2}{32C_4^2}}.
			\end{equation}
			Moreover, by Lemma~\ref{lem:diff-smooth-mv} and $\nu \le \frac{\varepsilon^2}{64 \rho C_4^2 M_l^2}$, it holds that 
			\begin{equation}
				\label{eq:g2}
				  \left\|\mathcal{G}_{\frac{1}{\rho}}\left(\vx^{k}\right)\right\| \leq\left\|\mathcal{G}_{\frac{1}{\rho}}^{\nu}\left(\vx^{k}\right)\right\|+\sqrt{{\frac{M_l^2 \nu \rho}{2} }}\leq \left\|\mathcal{G}_{\frac{1}{\rho}}^{\nu}\left(\vx^{k}\right)\right\|+ \frac{\varepsilon}{8C_4}.
			\end{equation}
Now from~\eqref{eq:g1} and~\eqref{eq:g2}, we get $$\min_{\overline{K}_1\leq k\leq K-1} \|\mathcal{G}_{\frac{1}{\rho}}(\vx^{k})\| \leq \frac{\varepsilon}{8C_4} + 
        \sqrt{\frac{2\rho C_{\vx}}{K-\overline{K}_1} + \frac{\varepsilon^2}{32C_4^2}}.$$ Noting $K=\overline{K}_1+\overline{K}_2\geq \overline{K}_1+64\rho C_{\vx}C_4^2\varepsilon^{-2}$, we obtain $$ \min_{\overline{K}_1\leq k\leq K-1} \|\mathcal{G}_{\frac{1}{\rho}}(\vx^{k})\| \leq \frac{3\varepsilon}{8C_4}.$$ Let $k'=\argmin_{\overline{K}_1\leq k\leq {K}-1} \|\mathcal{G}_{\frac{1}{\rho}}(\vx^{k})\|$. Then from~\eqref{eq:importantine} and $C_4\ge \frac{3}{2}$, it holds  
			$\operatorname{dist}(\vzero , \partial_\vx {\cL}_{\beta_{k'}}( \widetilde{\vx}^{k'}; \vy^{k'},\vz^{k'})) \leq \varepsilon$
			and $\|\widetilde{\vx}^{k'}- \vx^{k'}\|\leq \frac{3\varepsilon}{4\rho C_4}$.  
			Denote $$ \overline{\vy}^{k'} = \vy^{k'-1} + \beta_{k'}(\vA\widetilde{\vx}^{{k'}}-\vb), \text{ and }\overline{\vz}^{k'} = [\vz^{k'-1} + \beta_{k'-1} \vg(\widetilde{\vx}^{{k'}})]_+. $$
			Then noticing $ \partial_{\vx}\mathcal{L}_{\beta_{k'}}(\widetilde{\vx}^{{k'}}; \vy^{{k'}}, \vz^{{k'}}) = \partial_{\vx}\mathcal{L}_0(\widetilde{\vx}^{{k'}}; \overline{\vy}^{{k'}}, \overline{\vz}^{{k'}})$, we have 
            \begin{equation}
				\label{eq:condi1}
				\dist\left(\vzero, \partial_{\vx}\mathcal{L}_0(\widetilde{\vx}^{{k'}}; \overline{\vy}^{{k'}}, \overline{\vz}^{{k'}}) \right) \leq \varepsilon.
			\end{equation}
			Furthermore, since $k'\ge \overline{K}_1$, it follows from Young's inequality,~\eqref{eq: diffL}, and update of $\beta_k$ that  
				\begin{align}\label{eq:condi2}
					&  {\left\|\vA \widetilde{\vx}^{k'}- \vb\right\|^2 + \left\|[\vg( \widetilde{\vx}^{k'})]_+\right\|^2}\notag\\ \leq &2\left\|\vA {\vx}^{k'}- \vb\right\|^2 + 2\left\|[\vg( {\vx}^{k'})]_+\right\|^2 + 2\left\|\vA (\widetilde{\vx}^{k'}-{\vx}^{k'})\right\|^2 \notag \\ 
                    & \quad \quad \quad \quad + 2\left\|[\vg( \widetilde{\vx}^{k'})]_+-[\vg( {\vx}^{k'})]_+\right\|^2\notag\\
					\leq &   \frac{2C^2_p}{\beta_{\overline{K}_1}^2} + 2(\|\vA\|^2+mB_g^2)\frac{9\varepsilon^2}{16\rho^2C_4^2} \leq \varepsilon^2,
				\end{align}
			where 
			the last inequality follows from $C_4\geq \frac{3\sqrt{\|\vA\|^2+mB_g^2}}{2\rho}$ and $\overline{K}_1\geq 4C_p^2\beta_0^{-2}\varepsilon^{-2}$.
		Finally, we have 
				\begin{align}\label{eq:condi3}
		  	&\sum_{i=1}^m\big|\overline{z}_i^{k'}g_i(\widetilde{\vx}^{k'}) \big|=\sum_{i=1}^m\left| \big[ z_i^{k'-1} + \beta_{k'-1} g_i( \widetilde{\vx}^{k'})\big]_+g_i(\widetilde{\vx}^{k'}) \right| \notag \\
                \leq&  \frac{1}{\beta_{k'-1}} \sum_{i=1}^m (z_i^{k'-1})^2 + \frac{5\beta_{k'-1}}{4}\sum_{i=1}^m [g_i(\widetilde{\vx}^{k'})]_+^2 \notag\\
            \leq 
             &  \frac{1}{\beta_{k'-1}} \sum_{i=1}^m (z_i^{k'-1})^2 + \frac{5\beta_{k'-1}}{2}\sum_{i=1}^m [g_i({\vx}^{k'})]_+^2 \notag \\
             & \quad \quad \quad \quad + \frac{5\beta_{k'-1}}{2}\left\|[\vg( \widetilde{\vx}^{k'})]_+-[\vg( {\vx}^{k'})]_+\right\|^2 \notag\\
             \leq &  \frac{1}{\beta_{k'-1}} \left( B_p^2 + \frac{5C_p^2}{2}\right)+ \frac{45\sqrt{K}m B_g^2\beta_0 \varepsilon^2}{2(4\rho C_4)^2} \notag \\
            \leq& \frac{\varepsilon}{2} + \frac{45mB_g^2\beta_0 \varepsilon}{64\rho^2 C_4^2}\left(\sqrt{C_5} + 8\sqrt{\rho C_{\vx}}C_4+2\right)\leq \varepsilon.
				\end{align}
   Here, the first inequality holds by letting
   $\vx = \widetilde{\vx}^{k'}$ and $k=k'-1$ in
~\eqref{eq:comp-bd};  
			 the third inequality follows from Lemma~\ref{lem:boundxz}, the inequality in~\eqref{eq: diffL}, $\beta_{k'}\leq \sqrt{K+1}\beta_0$,  the Lipschitz continuity of $[\vg]_+$, and $\|\widetilde{\vx}^{k'}- \vx^{k'}\|\leq \frac{3\varepsilon}{4\rho C_4}$; the fourth inequality results from  $k'\geq \overline{K}_1\geq 4(B_p^2+\frac{5C_p^2}{2})^2\beta_0^{-2}\varepsilon^{-2}$, $\beta_{k'-1}=\beta_0\sqrt{k'}\geq \beta_0\sqrt{\overline{K}_1}$, and the definition of $K$; the last inequality holds because of the definition of $C_4$ in~\eqref{eq:def-C-4}. 
			
		We obtain from~\eqref{eq:condi1}--\eqref{eq:condi3} that $\widetilde{\vx}^{k'}$ is an $\varepsilon$-KKT point of problem~\eqref{problem:main}. 
		Since $\|\widetilde{\vx}^{k'}- \vx^{k'}\|\leq \frac{3\varepsilon}{4\rho C_4}\le \varepsilon$, then $\vx^{k'}$ is a near $\varepsilon$-KKT point of problem~\eqref{problem:main} by Definition~\ref{def: epskkt}. This completes the proof.
		\end{proof}
		
		Below we give the number of iterations for solving~\eqref{eq:sub-3} by Alg.~\ref{alg:acceleratedNesterov}  such that~\eqref{eq:terminate2} is met. 
		 The following lemma directly follows from Lemma~\ref{lem: caseIIfLsmoothness} and Theorem~\ref{Th: TboundNesterov}.

		\begin{lemma}
			\label{lem:innercase2}
		Given $\varepsilon_k>0$ and $\nu_k = \nu> 0$, under Assumptions~\ref{Assump1}--\ref{Assump5} and~\ref{ass:case2},    
			Alg.~\ref{alg:acceleratedNesterov} with $\gamma_u=2$ applied to~\eqref{eq:sub-3} can find a solution $\vx^{k+1}$ that satisfies the criteria in~\eqref{eq:terminate2} within
			\begin{align}\label{eq: innerComplexityOurProblem2}
		 		T^k_2 = \left\lceil\max\left\{\frac{1}{\log 2}, 2 {\sqrt{\frac{ L_{\widehat{f}^k}}{\rho}}}\right\} \log \frac{9 DL_{\widehat{f}^k}\sqrt{\frac{L_{\widehat{f}^k}}{\rho}}}{\varepsilon_k}\right\rceil + 1
			\end{align}
			 iterations, where  $L_{\widehat{f}^k} = C_6 + \sqrt{k+1}\beta_0 C_2$, with $C_6:=\frac{\|\nabla \vc\|}{\nu} + \sqrt{m}L_gB_p$, $C_2$ given in Lemma~\ref{lem:innercase1} and {$\|\nabla \vc\|$ given in~\eqref{eq:jacobinc}.}
		\end{lemma}
		
		Combining Theorem~\ref{lem:2} and Lemma~\ref{lem:innercase2}, we are ready to show the total complexity of  Alg.~\ref{alg:ialm}.
		\begin{theorem}[Total complexity result II]
  \label{thm:totalcase2}
			For a given $\vareps>0$, under Assumptions~\ref{Assump1}--\ref{Assump5} and~\ref{ass:case2}, Alg.~\ref{alg:ialm}, with $\{\beta_k, v_k\}$ set as in~\eqref{eq:choice-v-beta} and with Alg.~\ref{alg:acceleratedNesterov} as a subroutine to compute $\vx^{k+1}$ by solving~\eqref{eq:sub-3}, can find a near $\varepsilon$-KKT point of problem~\eqref{problem:main} by $T^{\mathrm{total}}_2$ proximal gradient steps. Here, $T^{\mathrm{total}}_2$ satisfies 
			\begin{dmath}\label{eq:total-comp-II}
			T^{\mathrm{total}}_2 \leq 2K+K\left(2\sqrt{\frac{C_6}{\rho}} + 2\sqrt{\frac{\sqrt{K} \beta_0C_2}{\rho}}+ \frac{1}{\log 2}\right)\log \left( 9 \varepsilon_K^{-1} D\sqrt{\frac{1}{\rho}}(C_6 +  \sqrt{K}\beta_0 C_2)^{\frac{3}{2}}\right),
			\end{dmath}
			where $C_2$ and $C_6$ are defined in Lemmas~\ref{lem:innercase1} and~\ref{lem:innercase2}, 
			$K$ is given in Theorem~\ref{lem:2}, and $\vareps_k$ is in~\eqref{eq:terminate2ep}. 
		\end{theorem}
\begin{proof}
Set $\nu_{k}=\nu$ as in Theorem~\ref{lem:2}. Notice that $T^{\mathrm{total}}_2 \le \sum_{k=0}^{K-1} T^k_2$, where $T^k_2$ is given in~\eqref{eq: innerComplexityOurProblem2}.  We obtain the desired result by the same arguments in the proof of Theorem~\ref{thm:totalcase1}.
\end{proof}
        \begin{remark}
           \label{rem:rhoep2}
           By $K=\cO(\vareps^{-2})$, $\nu_k=\nu=\cO(\vareps^2)$,  we have from \eqref{eq:total-comp-II} and the definition of $C_6$ in Lemma~\ref{lem:innercase2}  that $$T_2^{\mathrm{total}} = \widetilde{\mathcal{O}}\left(K \sqrt{ \frac{\|\nabla \vc\|}{\nu}}+K^{\frac{5}{4}}\right)= \widetilde{\mathcal{O}}(\varepsilon^{-3}).$$ 
        \end{remark}   

\subsection{General Weakly-Convex Objective} \label{subsec: caseIII}
In this subsection, we make the following structural assumption on the function~$f$ in~\eqref{problem:main}.
	\begin{assumption}
		\label{ass:case3}
In~\eqref{problem:main}, $f$ satisfies $\|\vxi\| \le \widehat{B}_f$, for any $\vxi\in \partial f(\vx)$ and any $\vx\in \cX$.  
	\end{assumption}

Without a smoothness structure, we do not expect an FOM to produce an $\vareps_k$-stationary point of the problem $\min_\vx\widetilde{\mathcal{L}}_{\beta_k}({\vx}; \vy^k, \vz^k)$ as in~\eqref{eq:terminate1}. Instead, 
we compute a near-optimal solution ${\vx}^{k+1}$ satisfying 
\begin{equation}\label{eq:cond-general-x}
  \widetilde{\mathcal{L}}_{\beta_k}({\vx}^{k+1}; \vy^k, \vz^k)- \min_{\vx}\widetilde{\mathcal{L}}_{\beta_k}\left( \vx; \vy^k, \vz^k\right) \le \frac{\vareps_k^2}{\rho}.
\end{equation}

Let $\vx^{k+1}_*=\argmin_{\vx}\widetilde{\mathcal{L}}_{\beta_k}\left( \vx; \vy^k, \vz^k\right)$. Then by the $\rho$-strong convexity of $\widetilde{\mathcal{L}}_{\beta_k}({\vx}; \vy^k, \vz^k)$, it holds 
\begin{equation}\label{eq:dist-x-k-xstar}
\begin{aligned}
  \|\vx^{k+1}_* - {\vx}^{k+1}\|^2 \le \frac{2}{\rho}\left(\widetilde{\mathcal{L}}_{\beta_k}({\vx}^{k+1}; \vy^k, \vz^k)- \widetilde{\mathcal{L}}_{\beta_k}({\vx}^{k+1}_*; \vy^k, \vz^k)\right) \overset{\eqref{eq:cond-general-x}}\le \frac{2\vareps_k^2}{\rho^2},\\
\text{ for all } k\ge0.
\end{aligned}
\end{equation}
Hence, $\vx^{k+1}$ satisfies~\eqref{eq:approx-cond} with $\overline{f} = f$,   $\overline{B}_f=\widehat{B}_f$, $\widehat{\vx}^{k+1} = \vx_*^{k+1}$, $\delta_k = \frac{\sqrt{2}\varepsilon_k}{\rho}$ for all $k\ge 0$.
  Let $\varepsilon_k \leq \frac{1}{\beta_k(k+2) (\log (k+2))^2}$.
  Then $\sum_{k \ge 0} \beta_k {\delta}_k < \infty$. Thus all the lemmas in Sect.~\ref{sec: ConveAnalysis} hold in this case.

    \begin{theorem}[Outer iteration complexity result III]
    \label{thm:outer3}
         Given $\vareps>0$, under Assumptions~\ref{Assump1}-\ref{Assump5} and~\ref{ass:case3},
         let $\{\vx^k,\vy^k,\vz^k\}$ be generated by Alg.~\ref{alg:ialm} such that~\eqref{eq:obj-opt-error} holds
         with $\varepsilon_k:=\min\Big\{\frac{\varepsilon}{4}, \frac{\rho\vareps}{\sqrt{2}}, 1, 
         \sqrt{\frac{\rho}{2\beta_k}}, \frac{1}{\beta_k(k+2) (\log (k+2))^2}\Big\} $ for all $k\ge0$  and with $\{\beta_k, v_k\}$ set as in~\eqref{eq:choice-v-beta}. 
			Then for some $k < {K} := \widetilde{K}_1 + \widetilde{K}_2$, $\vx^{k+1}$ is a near $\varepsilon$-KKT point of problem~\eqref{problem:main},  
			where $\widetilde{K}_1 := \left\lceil \max\{C_p^2\beta_0^{-2}, (B_p^2+\frac{5C_p^2}{4})^2\beta_0^{-2}\} \varepsilon^{-2}\right\rceil$ and $\widetilde{K}_2 := \left \lceil 16 C_{\vx}\rho \varepsilon^{-2}\right \rceil$, with $B_p$, $C_p$, and $C_{\vx}$  given in Lemma~\ref{lem:boundxz}, Lemma~\ref{lem:tilderhodeltak}, and Lemma~\ref{lem: ProxErr}, respectively. 
    \end{theorem}
    \begin{proof}
    Apply Lemma~\ref{lem:tilderhodeltak} to $\vx^{k+1}_*$, i.e., $\delta_k=0, \vareps_k=0$ for all $k\ge0$. We have 
	\begin{equation}\label{eq: genercondi1}
	 	{\left\|\vA {\vx}^{k+1}_*- \vb\right\|^2 + \left\|\left[\vg\left( {\vx}^{k+1}_*\right)\right]_+\right\|^2} \leq \frac{C_p^2}{\beta_k^2},\ \forall\, k\ge0. 
	\end{equation}
Then by~\eqref{eq:CS-k+1-x} in Lemma~\ref{lem: CompSlack} and~\eqref{eq: genercondi1}, 
	it holds that 
	\begin{equation}\label{eq: genercondi2}
	 	\sum_{i=1}^m|[ z_i^{k} + \beta_{k} g_i({\vx}^{k+1}_*)]_+g_i({\vx}^{k+1}_*) |\,\,\leq \frac{1}{\beta_k}\left( B_p^2 + \frac{5C_p^2}{4}\right),\ \forall\, k\ge0. 
			\end{equation} 
In addition, by the $\rho$-strong convexity of $\widetilde{\mathcal{L}}_{\beta_k}$, we have	 $$\widetilde{\mathcal{L}}_{\beta_k}( {\vx}^{k+1}_*; \vy^k, \vz^k) \leq  {\mathcal{L}}_{\beta_k}( \vx^k; \vy^k, \vz^k) - \frac{\rho}{2} \left\|\vx^k - {\vx}^{k+1}_*\right\|^2,$$ which together with~\eqref{eq:cond-general-x}, gives	
			\begin{align}\label{eq:strongConvLgener}
			&  {\mathcal{L}}_{\beta_k}( {\vx}^{k+1}; \vy^k, \vz^k) + \rho \|{\vx}^{k+1}-\vx^{k}\|^2 + \frac{\rho}{2} \|{\vx}^{k+1}_*-\vx^{k}\|^2 
			\leq  {\mathcal{L}}_{\beta_k}( \vx^k; \vy^k, \vz^k) + \frac{\vareps_k^2}{\rho}.
			\end{align}			 
{Sum up~\eqref{eq:strongConvLgener} over $k$ from $\widetilde{K}_1$ to $K-1$ and use~\eqref{eq:sm-L-value} to have
\begin{align*}
    \frac{\rho}{2} \sum_{k=\widetilde{K}_1}^{K - 1}\|{\vx}^{k+1}_*-\vx^{k}\|^2 \le\sum_{k= \widetilde{K}_1}^{K - 1} \frac{\vareps_k^2}{\rho} + C_\vx.
\end{align*}
Since $\vareps_k \le \frac{\vareps}{4}$, it holds $\sum_{k= \widetilde{K}_1}^{K - 1} \frac{\vareps_k^2}{\rho} \le \frac{\widetilde{K}_2\vareps^2}{16\rho}$. 
Let
\begin{align*}
    k'=\argmin_{\widetilde{K}_1\le k \le K - 1} \|{\vx}^{k+1}_*-\vx^{k}\|^2.
\end{align*}
Then $\|{\vx}^{k'+1}_*-\vx^{k'}\| \le \sqrt{\frac{\vareps^2}{8 \rho^2 } + \frac{2 C_\vx}{\rho \widetilde{K}_2}} \le \frac{\vareps}{2\rho}$ by $\widetilde{K}_2\ge 16 C_{\vx}\rho \varepsilon^{-2}$. 
Now notice
\begin{align*}
    \vzero\in \partial_{\vx}\widetilde{\cL}_{\beta_k}({\vx}^{k+1}_*; \vy^{k}, \vz^{k}) = \partial_{\vx}\mathcal{L}_{\beta_k}({\vx}^{k+1}_*; \vy^{k}, \vz^{k}) + 2\rho({\vx}^{k+1}_* - \vx^k),
 \end{align*}   
 and
 \begin{align*}
    \partial_{\vx}\mathcal{L}_{\beta_k}({\vx}^{k+1}_*; \vy^{k}, \vz^{k}) = \partial_{\vx}\mathcal{L}_0({\vx}^{k+1}_*; \overline{\vy}^{k}, \overline{\vz}^{k})
\end{align*}
with 
\begin{align*}
    \overline{\vy}^k := \vy^k + \beta_k(\vA{\vx}^{k+1}_*-\vb),\ 
    \overline{\vz}^k := [\vz^k + \beta_k \vg({\vx}^{k+1}_*)]_+.
\end{align*}
We obtain that
           $\dist\big(\mathbf{0}, \partial_{\vx}\mathcal{L}_{0}({\vx}^{k'+1}_*; \overline{\vy}^{k'}, \overline{\vz}^{k'})\big) \leq \varepsilon.$ 
This claim, together with 
\eqref{eq: genercondi1}, \eqref{eq: genercondi2}, and the choice of $\widetilde{K}_1$, indicates that ${\vx}^{k'+1}_*$ is an $\vareps$-KKT point of problem~\eqref{problem:main}. Moreover, from~\eqref{eq:dist-x-k-xstar} and $\vareps_k \le \frac{\rho\vareps}{\sqrt{2}}$, it follows that $\|\vx^{k'+1}_* - {\vx}^{k'+1}\| \le \vareps$. Therefore, ${\vx}^{k'+1}$ is a near $\vareps$-KKT point of problem~\eqref{problem:main} by Definition~\ref{def: epskkt}, and we complete the proof.}
    \end{proof}

\begin{remark}
We make a few remarks about Theorem~\ref{thm:outer3} and its implications. First, in a general case, one can apply a subgradient method \cite{nesterov2003introductory} to find ${\vx}^{k+1}$ such that~\eqref{eq:cond-general-x} holds, due to the strong convexity of $\widetilde{\mathcal{L}}_{\beta_k}(\,\cdot\,; \vy^k, \vz^k)$. 
    Our $\mathcal{O}(\vareps^{-2})$ outer iteration complexity result matches with that in \cite{zeng2022moreau}, but we allow a smaller penalty parameter for nondifferentiable problems and thus can potentially achieve a lower overall complexity result. Second, when there are certain special structures on $f$ such as those in Sect.~\ref{subsec: caseI} and Sect.~\ref{subsec: caseII}, one can apply a more efficient way to obtain ${\vx}^{k+1}$ and achieve a lower complexity. The best way to compute ${\vx}^{k+1}$ will depend on the structure on $f$. 
    For example, 
    when $f = \max\{f_1,f_2\}$ where $f_1$ and $f_2$ are both $\rho$-weakly convex and smooth, one can apply the Moreau-envelope based smoothing approach in Sect.~\ref{subsec: caseII} and achieve an overall complexity of $\widetilde{\mathcal{O}}(\vareps^{-3})$ to produce a near $\vareps$-KKT point. However, a potentially better way is to have a more efficient subroutine to solve each strongly convex subproblem $\min_\vx\widetilde{\mathcal{L}}_{\beta_k}( \vx; \vy^k, \vz^k)$ by exploiting the special structure of $f$.
    {Notice that 
    \begin{align*}
    &\,f(\vx) + \frac{\rho}{2}\|\vx - \vx^k\|^2 \\
    = &\, \max_{\lambda\in [0,1]} \lambda \left(f_1(\vx)+ \frac{\rho}{2}\|\vx - \vx^k\|^2\right) + (1-\lambda) \left(f_2(\vx)+ \frac{\rho}{2}\|\vx - \vx^k\|^2\right).
    \end{align*}
    Let 
    \begin{align*}
        H_k(\vx):= & h(\vx) +\frac{\rho}{2}\|\vx - \vx^k\|^2+ (\vy^k)\zz(\vA\vx-\vb) \\ & +\frac{\beta_k}{2}\|\vA\vx-\vb\|^2+\frac{\beta_k}{2}\|[ \vg( \vx)+\frac{ \vz^k}{\beta_k}]_+\|^2-\frac{\|\vz^k\|^2}{2\beta_k}.
    \end{align*}
     It then holds
    \begin{align*}
        \min_\vx\widetilde{\mathcal{L}}_{\beta_k}( \vx; \vy^k, \vz^k) = \max_{\lambda\in [0,1]} \min_\vx & \lambda \big(f_1(\vx)+ \frac{\rho}{2}\|\vx - \vx^k\|^2\big) + (1-\lambda) \big(f_2(\vx) \\ 
        &+ \frac{\rho}{2}\|\vx - \vx^k\|^2\big) + H_k(\vx).
    \end{align*}   
By exploiting the 1-dimension of $\lambda$, we can follow \cite{xu2022first} and apply a bisection method to search for the optimal~$\lambda$, by which we can then find the desired $\vx^{k+1}$ in $\widetilde{\mathcal{O}}(\sqrt{\beta_k})$ iterations. Hence, by Theorem~\ref{thm:outer3}, the total iteration complexity is $\widetilde{O}(\varepsilon^{-2.5})$ to produce a near $\vareps$-KKT point. We leave details to interested readers.}
\end{remark}
  
 \section{Numerical Experiments} \label{sec: NumericalExp}
		In this section, we conduct numerical experiments to demonstrate the effectiveness of our algorithm, named DPALM. We apply it to the non-convex linearly constrained quadratic problem (LCQP), non-convex quadratically constrained quadratic problem (QCQP), linearly constrained robust nonlinear least square, and fairness classification.  
  All of these tests are performed in MATLAB 2022a on an iMAC with 40GB memory. 
	\subsection{Data generation and parameter setting}\label{sec:data-gen}
    Synthetic data are used for the experiments in Sect.~\ref{subsec: QP}, Sect.~\ref{subsec: QCQP}, and Sect.~\ref{subsec: nllsq}. 
Details on generating the data are described below.
        
For the linear constraint $\vA\vx=\vb$ in Sect.~\ref{subsec: QP} and Sect.~\ref{subsec: nllsq}, $\vA$ and $\vb$ are generated by the MATLAB command \verb|A = [randn(n,d-n), eye(n)]| and \verb|b = randn(n,1) + 0.1|. The box constraint $[l_1,u_1] \times\cdots\times [l_d, u_d]$ is set with $l_i=-5$ and $u_i=5$ for each $i\in [d]$. For every generated instance, we observe that all entries of $\vb$ are in the box constraint so all instances are feasible at $\vx=[\vzero;\vb]$.        
Each of the vectors $\{\vc_i\}_{i=0}^m$ is generated by the command \verb|randn(d,1)|. 
The matrix $\vQ_0$ in Sect.~\ref{subsec: QP} and Sect.~\ref{subsec: QCQP} is generated by the command \verb|U*diag(max(0, randn(d,1)*5))*U' - rho*eye(d)|, where \verb|U| is generated as an orthonormal matrix by the command \verb|[U,~]=qr(randn(d))|, and \verb|rho| varies in $\{0.1,1,10\}$. Hence the minimum eigenvalue of $\vQ_0$ is set as desired. All matrices $\{\vQ_j\}$ in Sect.~\ref{subsec: nllsq} are generated in the same way. In Sect.~\ref{subsec: QCQP}, each of the matrices $\{\vQ_j\}_{j=1}^m$ is positive semidefinite and generated by the command \verb|U*diag(rand(d,1)*5 + 1)*U'|, where \verb|U| has size of $d\times (d-5)$ and orthonormal columns and is generated by the command \verb|[U,~]=qr(randn(d)), U = U(:,1:d-5)|; each $\gamma_j$ is positive and generated by the command \verb|max(0, randn*2) + 0.1|, and thus Slater's condition holds at the origin for instances of~\eqref{problem: qcqp}.
        
In each test, all compared methods use the same initial point, and the initial dual iterate (if needed) is always set to a zero vector. A warm-start strategy (for initial point and smoothness constant) is always adopted between subproblems for all methods. Except HiAPeM for which we use its default settings, we apply Alg.~\ref{alg:acceleratedNesterov} as the subroutine and set $\gamma_u = 3$ and $\gamma_d = 5$ when the line search technique is used. For the test on LCQP, we use the fixed smoothness constant $L_k = \|\vQ_0\| + \beta_k\|\vA^\top\vA\|$ for the $k$-th subproblem; for all other tests, we apply the line search technique, and the initial smoothness constant $L_0$ is set to $\|\vQ_0\|$ for QCQP and $10^2$ for the tests in Sect.~\ref{subsec: nllsq} and Sect.~\ref{subsec: fairnessROC}. 
The parameter $\sigma$ for NL-IAPIAL is set to $\frac{1}{\sqrt{2}}$.    The proximal parameter of the Prox-linear method applied to~\eqref{problem:l-rnls} and~\eqref{exp4: main}  is set to $\frac{1}{2\rho}$, where $\rho$ is the weak convexity constant. Hence, we use the same proximal parameter for DPALM and the Prox-linear method. The values of other parameters that are used will be given below. 

		\subsection{Non-convex Linearly-Constrained Quadratic Program (LCQP)}\label{subsec: QP}
	In this subsection, we apply different methods to solve LCQP in the form of 
		\begin{align}\label{problem: qp} 
			&\min_{\vx\in \mathbb{R}^d} \frac{1}{2}\vx^{\top}\vQ_0\vx + \vc_0^{\top}\vx, \notag \\ &\st \vA\vx = \vb,\ x_i \in [l_i, u_i] = [-5, 5],\, \text{ for all }  i \in [d],
		\end{align}
	where the data $\{\vQ_0, \vc_0, \vA, \vb\}$ are generated in the way as described in Sect.~\ref{sec:data-gen}.  

{We first compare the performance of DPALM with damped dual stepsize and that with full dual stepsize. We set $\beta_k=\beta_0\sqrt{k+1}$ and \mbox{$v_k=\frac{v_0}{\sqrt{k+1}},\forall\, k\ge0$} in Alg.~\ref{alg:ialm}. For both variants, we simply choose $\beta_0=0.1$. For the former variant, we choose $v_0\in\{0.1,1,10\}$, and for the latter, we set $v_0=+\infty$, thus \mbox{$\alpha_k=\beta_k,\forall\, k\ge0$} in Alg.~\ref{alg:ialm}.  Ten independent instances of \eqref{problem: qp} are generated with $n=10$, $d=100$, and the weak convexity constant $\rho=1$. The left plot in Figure~\ref{fig:damp-vs-full} shows the total number of gradient evaluations by DPALM  with different values of $v_0$ to produce an $\vareps$-KKT solution of each random instance, where $\vareps=10^{-3}$. From the plot, we observe that DPALM with $v_0=10$ performs exactly the same as that with $v_0=+\infty$. This indicates that a full dual stepsize is adopted at each outer iteration for all tested instances when $v_0=10$. 
In addition, we observe that DPALM with $v_0=1$ takes slightly fewer gradient evaluations than that with a full dual stepsize, in average, 9816 versus 9966, but significantly more gradients are evaluated when $v_0=0.1$. This indicates that applying overly damped dual stepsize may significantly slow down the empirical convergence of DPALM. In the middle and right plots of Figure~\ref{fig:damp-vs-full}, we show the curve of the adopted dual stepsize $\alpha_k$ and that of  $\beta_k$ for \mbox{instance \#1} and instance \#9, where most significant difference is observed between the two variants of DPALM with $v_0=1$ and $v_0=+\infty$. From the plots, we observe that though in the beginning, damped stepsize is adopted, eventually it holds $\alpha_k=\beta_k$ even for $v_0=0.1$. Based on these observations, in all the subsequent experiments, we use $v_k = \frac{v_0}{\sqrt{k+1}},\forall\, k\ge0$ and set $v_0$ to a relatively large value, i.e.,  
$v_0=200$, to achieve good empirical performance as well as to provide a theoretical safeguard.} 
\begin{figure}
    \centering
    \includegraphics[width=0.3\linewidth]{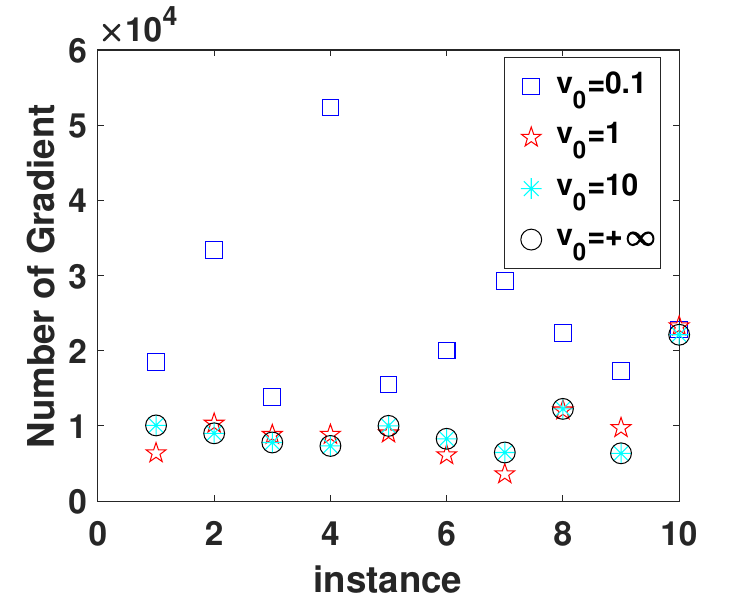}
    \includegraphics[width=0.3\linewidth]{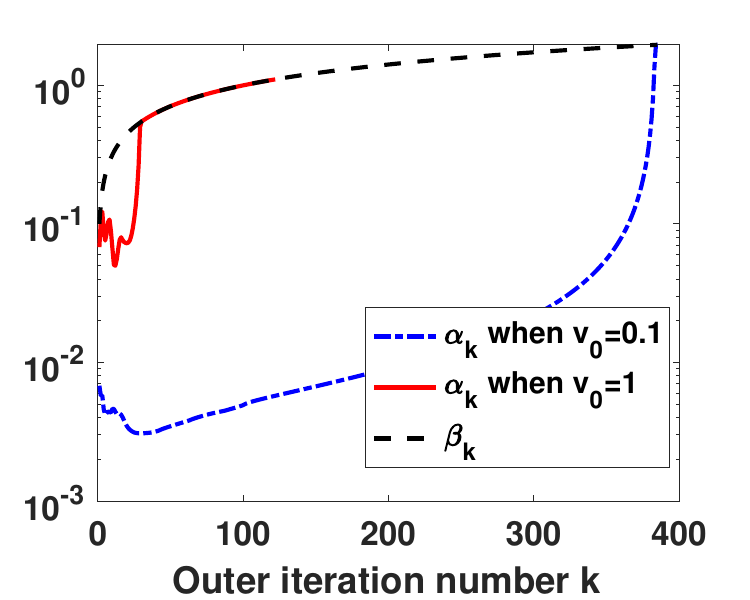}
    \includegraphics[width=0.3\linewidth]{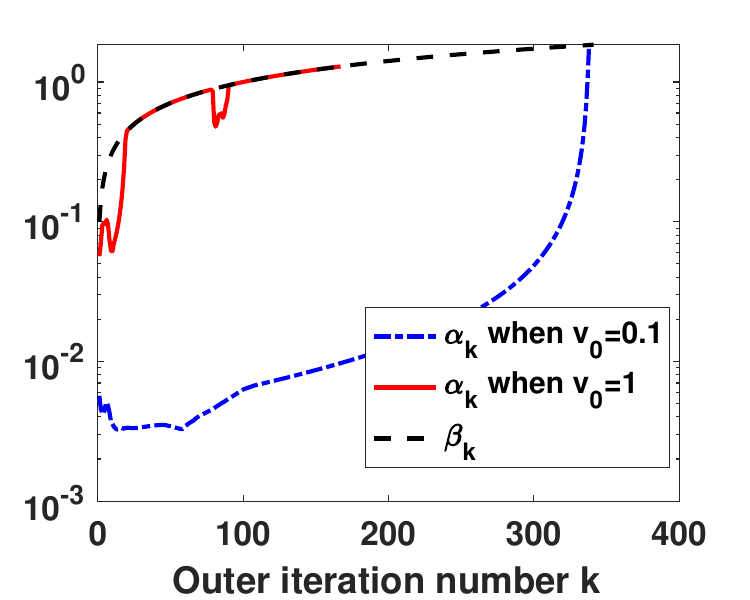}
    \caption{Left: the number of gradient evaluations by DPALM with different values of $v_0$ on solving 10 independent random instances of \eqref{problem: qp} with $n=10$, $d=100$, and $\rho=1$; Middle and Right: values of the adopted dual stepsize $\alpha_k$ and $\beta_k$ on instance \#1 and instance \#9.}
    \label{fig:damp-vs-full}
\end{figure}

      Then we compare DPALM to HiAPeM~\cite{li2021augmented}, LiMEAL~\cite{zeng2022moreau}, NL-IAPIAL~\cite{kong2022iteration}, and {IPL(A), as a variant of NL-IAPIAL in \cite{kong2022iteration}}.
We set $n=10$ and  $d = 1000$, and 
vary the weak convexity constant $\rho\in\{0.1, 1, 10\}$. Each algorithm is terminated if an $\varepsilon$-KKT solution is found or after $10^4$ outer iterations, where we set $\varepsilon=10^{-3}$. 
For DPALM and NL-IAPIAL, we tune the initial penalty parameter $\beta_0$ by picking the best one from $\{10^{-4}, 10^{-3}, 10^{-2}, 10^{-1}, 1, 10, 10^2 \}$ for each value of~$\rho$. This way, we have $\beta_0 = 10^{-3}, 10^{-4}, 10$ and $\beta_0 = 10^{-3}, 10^{-4}, 100$ corresponding to $\rho = 0.1, 1, 10$ respectively for DPALM and NL-IAPIAL. For LiMEAL, we set its parameters to $\gamma = \frac{1}{2\|\vQ_0\|}$ and $ \eta = 1.5$. We tune its penalty parameter from the same set of values that we use to tune for DPALM and obtain \mbox{$\beta = 0.1, 1, 100$} corresponding to  $\rho = 0.1, 1, 10$. For HiAPeM, we use its default value $\beta_0=0.01$ but set its parameter $N_0$ to $10^4$. Here, $N_0$ is the number of calls to ALM as subroutine in the initial stage of HiAPeM. Since we set the maximum number of outer iterations to $10^4$, using $N_0=10^4$ means that we run HiAPeM by solely using ALM to solve its proximal point subproblems. This yields the best performance for HiAPeM, as demonstrated in \cite{li2021augmented}. {
For IPL(A), we follow \cite{kong2022iteration} and set its parameters\footnote{We refer readers to \cite{kong2022iteration} for the meanings of its parameters.} to 
\begin{align*}
\beta_0 = \max\left\{1, \frac{L_f}{B_g^2} \right\}, \quad  \lambda = \frac{1}{\rho}, \quad \sigma = \sqrt{0.3}, \quad \nu = \sqrt{\sigma(\lambda L_f + 1)},
\end{align*}
and the update condition for doubling $\beta$ is changed to 
    $\Delta_k \leq  \frac{\lambda(1-\sigma^2)\varepsilon^2}{4(1+2\nu)^2}.$
}

For all the compared methods, we start from the same feasible point~$\vx^0$.
All algorithms use Nesterov's APG \cite{nesterov2013gradient} to solve their core strongly convex subproblems by warm start. Notice that our implementation of LiMEAL is different from that in the numerical experiment of \cite{zeng2022moreau} but instead we follow its update given in Eqn.~(9). 

For each value of $\rho$, we generate 10 independent random LCQP instances. In Table~\ref{table:qprecent}, we present the violation to PF and DF conditions, running time (in seconds), and the number of gradient evaluations (shortened by \verb|pres|, \verb|dres|, \verb|time|, \verb|#Grad|, respectively), averaged over 10 instances, 
for each method.  We also track which method performs the best in terms of \#Grad for each value of $\rho$, and the percentage is reported in Table~\ref{table:qprecent}. In Fig.~\ref{figure: qpfig}, we plot violation to PF and DF at each outer iteration, from the first random instance for each $\rho$, versus the number of gradient evaluations for all compared methods.
From the results in the table and the figure, we see that our method takes fewer gradient evaluations (and less running time), in average, than all other methods to produce the same-accurate KKT point. 
A larger value of $\rho$ means more non-convexity and thus a harder instance. However, we notice that all compared methods take fewer gradient evaluations for $\rho=10$ than $\rho=0.1$ and $1$. 
 This is possible because we tune the algorithm parameter for each $\rho$, or because we aim at producing a near-KKT point instead of a {globally} optimal solution. 

\begin{landscape}
		\begin{table}[htp]
           \vspace{4cm}
				\resizebox{1.5\textwidth}{!}{
					\begin{tabular}{|c||cccc|cccc|cccc|cccc|cccc|}
						\hline
						 & pres & dres & time & \#Grad & pres & dres & time & \#Grad & pres & dres & time & \#Grad & pres & dres & time & \#Grad & pres & dres & time & \#Grad\\\hline\hline
						 & \multicolumn{4}{|c|}{HiAPeM with $N_0 = 10^4$} & \multicolumn{4}{|c|}{NL-IAPIAL} & \multicolumn{4}{|c|}{IPL(A)} & \multicolumn{4}{|c|}{LiMEAL} & \multicolumn{4}{|c|}{DPALM (proposed)} \\\hline
						\multicolumn{18}{|c|}{weak convexity constant: $\rho=0.1$}\\\hline
						avg. & 3.94e-4 & 9.99e-4 & 221 & 1.29e5 & 9.26e-4 & 1.62e-4 & 6.35 & 2.21e4 & 1.77e-5 & 9.73e-4 & 33.1 & 1.53e5 & 8.68e-4 & 0.0528 & 68.7 & 1.55e5 & 4.53e-4 & 8.72e-4 & 5.99 & \textbf{2.11e4} \\\hline
						var. & 1.11e-7 & 4.88e-11 & 2.21e4 & 7.06e9  & 4.78e-9 & 1.18e-10 & 5.27 & 6.14e7 & 1.48e-11 & 4.42e-10 & 77.9 & 2.18e9 & 1.48e-8 & 0.0058 & 65.5 & 7.66e8 & 9.29e-9 & 3.16e-9 & 5.81 & 6.62e7\\\hline
                        best(\%) &  & 0 & & &  & 40 & & &  & 0 & &  & & 0 & & & &  60 & & \\\hline
						\multicolumn{18}{|c|}{weak convexity constant: $\rho=1.0$}\\\hline
						avg. & 1.64e-4 & 9.93e-4 & 588 & 3.93e5 & 9.08e-4 & 3.78e-4 & 11.2 & 3.44e4 & 1.46e-5 & 9.79e-4 & 33.5 & 1.55e5 & 1.44e-6 & 9.98e-4 & 116 & 3.91e5 & 9.55e-4 & 1.77e-4 & 6.81 & \textbf{1.97e4}\\\hline
						var. & 4.22e-8 & 1.73e-11 & 7.3e4 & 3.6e10  & 8.14e-10 & 3.88e-9 & 25.9 & 2.66e8 & 4.10e-12 & 5.75e-11 & 140 & 3.26e9 & 1.43e-13 & 8.28e-13 & 1.01e3 & 1.18e10 & 7.77e-10 & 3.73e-9 & 5.62 & 5.33e7 \\\hline
                        best(\%)  &  & 0 & & &  & 20 & & &  & 0 & &   & & 0 & & & &  80 & & \\\hline
						\multicolumn{18}{|c|}{weak convexity constant: $\rho=10$}\\\hline
						avg. & 1.16e-4 & 8.27e-4 & 53.6 & 7.97e4  & 1.19e-4 & 4.49e-4 & 11.4 & 5.45e4 & 4.24e-6 & 4.46e-4 & 12.9 & 5.69e4 & 5.80e-5 & 4.90e-4 & 11.3 & 4.61e4 & 1.53e-4 & 6.66e-4 & 6.09 & \textbf{2.95e4}\\\hline
						var. & 2.12e-8 & 7.33e-9 & 185 & 5.91e8 & 1.46e-8 & 7.46e-8 & 3.76 & 8.17e7 & 5.25e-12 & 1.20e-7 & 0.716 & 1.14e7 & 8.86e-9 & 9.89e-8 & 5.13 & 9.41e7 & 6.07-9 & 9.04e-8 & 6.74 & 1.73e8\\\hline
                        best(\%)  &  & 0 & & &  & 0 & & &  & 10 & &   & & 0 & & & &  90 & & \\\hline
					\end{tabular}
				}
        \caption{Average results and variance by the proposed algorithm DPALM, HiAPeM in~\cite{li2021augmented}, LiMEAL in~\cite{zeng2022moreau}, and NL-IAPIAL 
        and \\ IPL(A) in~\cite{kong2022iteration} on solving 10 instances of $\rho$-weakly convex LCQP in~\eqref{problem: qp} of size $n=10$ and $d=1000$, 
		where $\rho \in\{0.1,1.0,10\}$.}\label{table:qprecent}
		\end{table}
\end{landscape}

  \begin{figure}[htbp!]
  \begin{center}
      \begin{tabular}{ccc}
          $\rho = 0.1$ & $\rho = 1$ & $\rho = 10$   \\
          \includegraphics[width=0.29\textwidth]{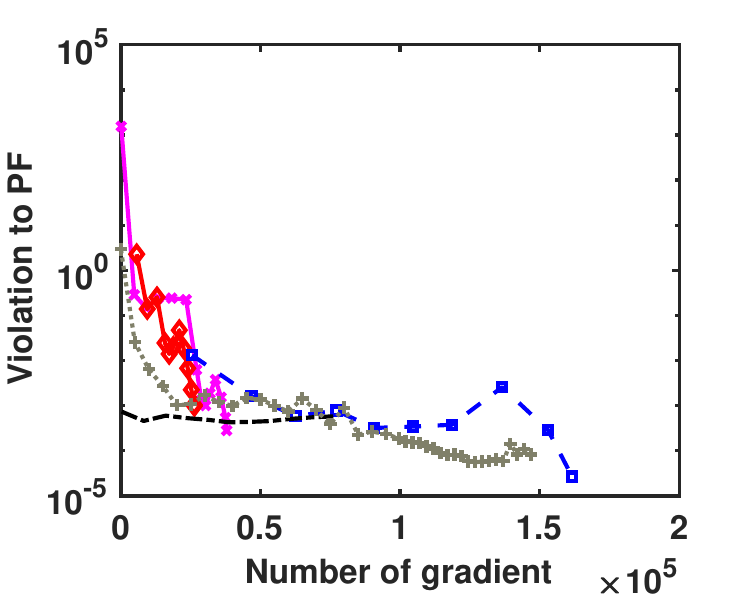} &
          \includegraphics[width=0.29\textwidth]{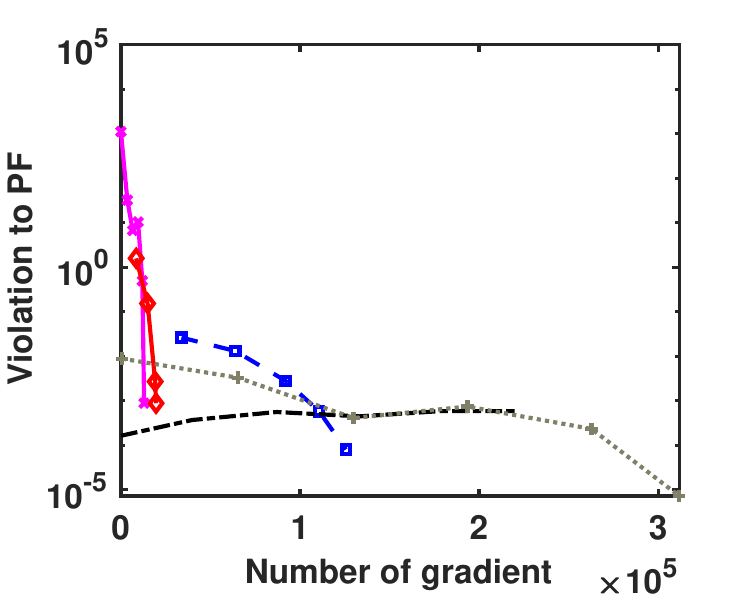} & 
          \includegraphics[width=0.29\textwidth]{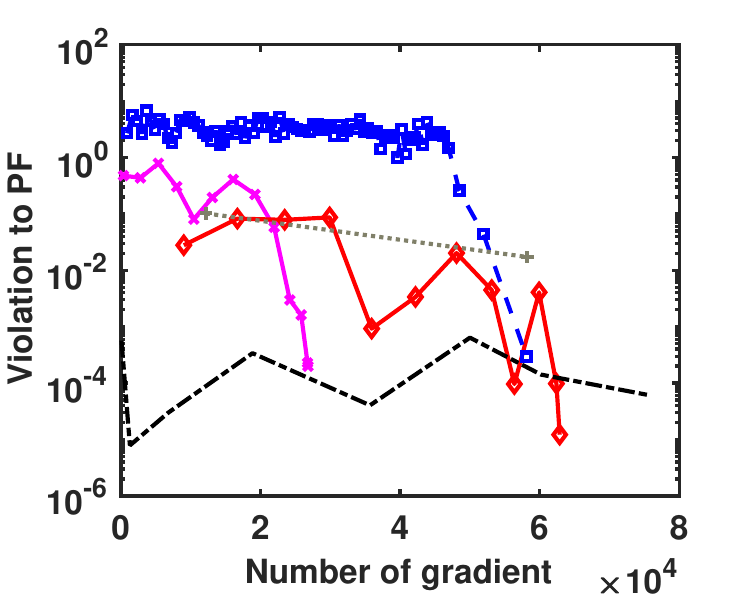} \\ 
          \includegraphics[width=0.29\textwidth]{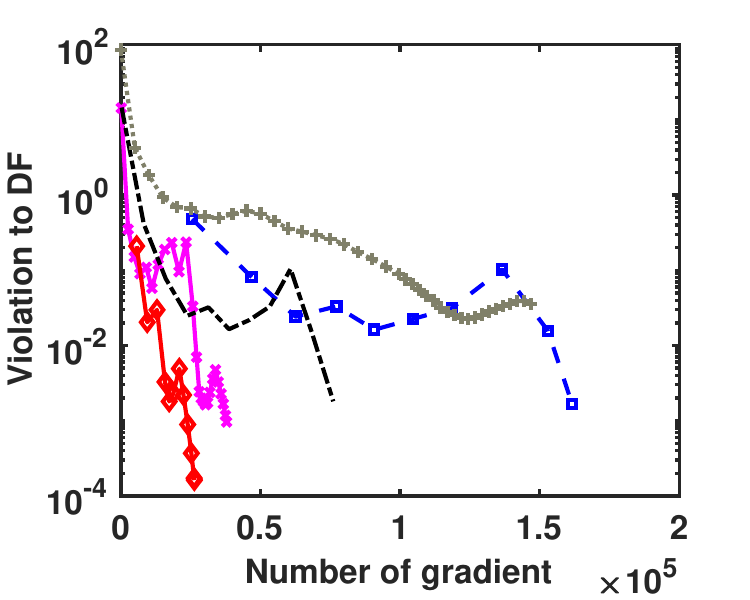} &
          \includegraphics[width=0.29\textwidth]{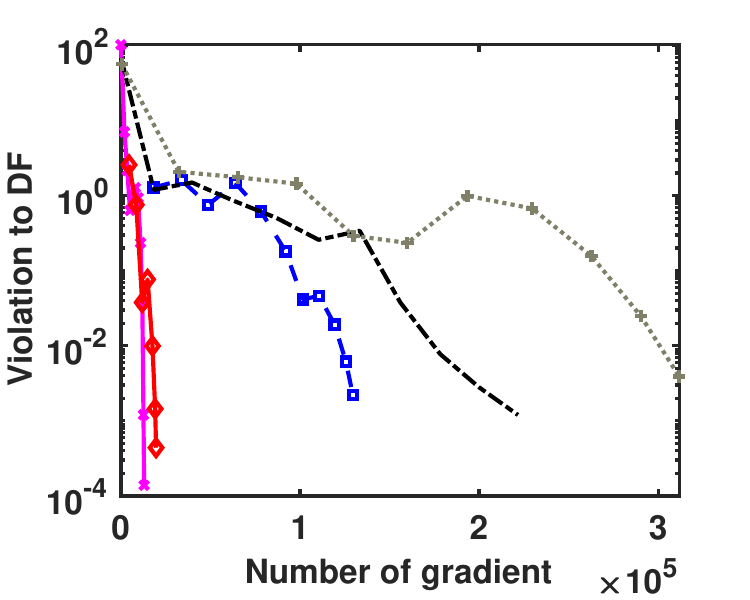} &
          \includegraphics[width=0.29\textwidth]{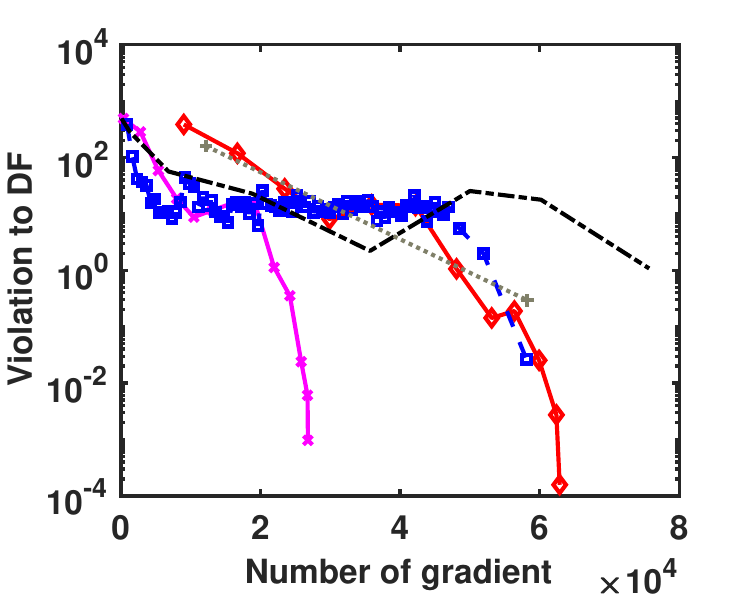} \\
      \end{tabular}
     \includegraphics[width=0.4\textwidth]{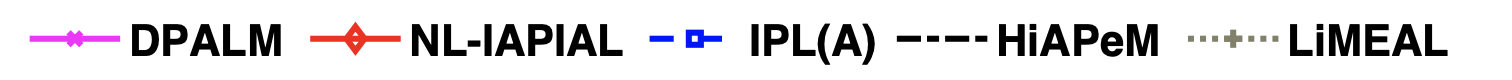} 
  \end{center}
  \caption{Violation to primal feasibility (PF) and dual feasibility (DF) conditions after each outer iteration vs. Number of gradient evaluations by the proposed DPALM, the HiAPeM method in~\cite{li2021augmented}, the LiMEAL in \cite{zeng2022moreau}, and the NL-IAPIAL method in~\cite{kong2022iteration} on solving instances of~\eqref{problem: qp} with different weak convexity constant $\rho\in \{0.1, 1, 10\}$.} \label{figure: qpfig}
  \end{figure}

\textbf{Evolution of penalty parameter.} 
In Fig.~\ref{figure: betacompare}, we visualize the evolution of the penalty parameter,~$\beta$, for DPALM and NL-IAPIAL applied to randomly generated LCQP instances for different values of $\rho$. DPALM uses the best $\beta_0$ that is selected from a set of values described above, while we run NL-IAPIAL with its own best $\beta_0$ and the $\beta_0$ that is used by DPALM. We see that DPALM increases $\beta$ slowly as designed. With a good-tuned $\beta_0$, NL-IAPIAL can achieve the desired solution without increasing $\beta$, while it does need to increase the value of $\beta$ a few times for some other values of $\beta_0$. 
\begin{figure}[ht]
  \begin{center}
      \begin{tabular}{ccc}
            $\rho = 0.1$ & $\rho = 1$ & $\rho = 10$   \\
           \includegraphics[width= 0.29\textwidth]{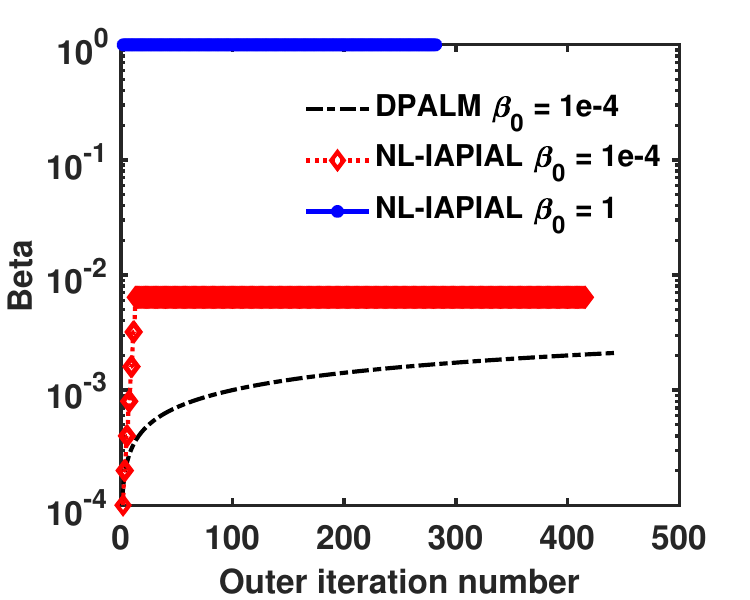} &
          \includegraphics[width= 0.29\textwidth]{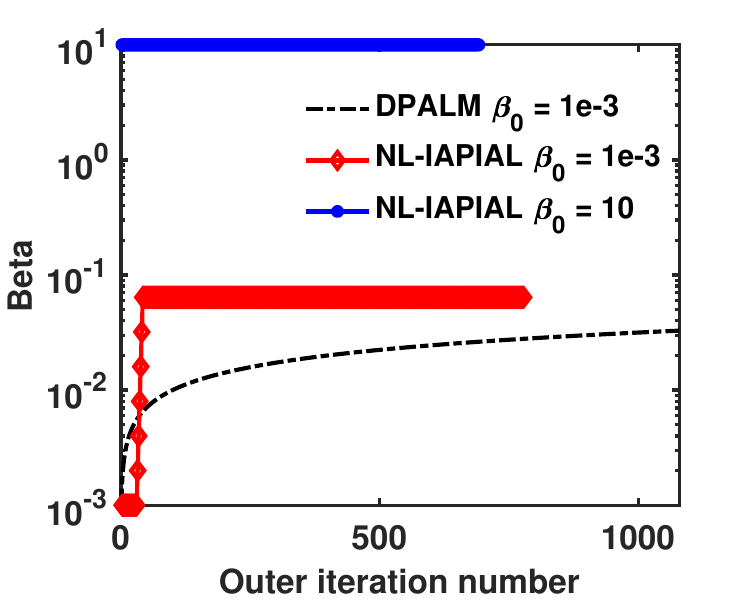} & 
          \includegraphics[width= 0.29\textwidth]{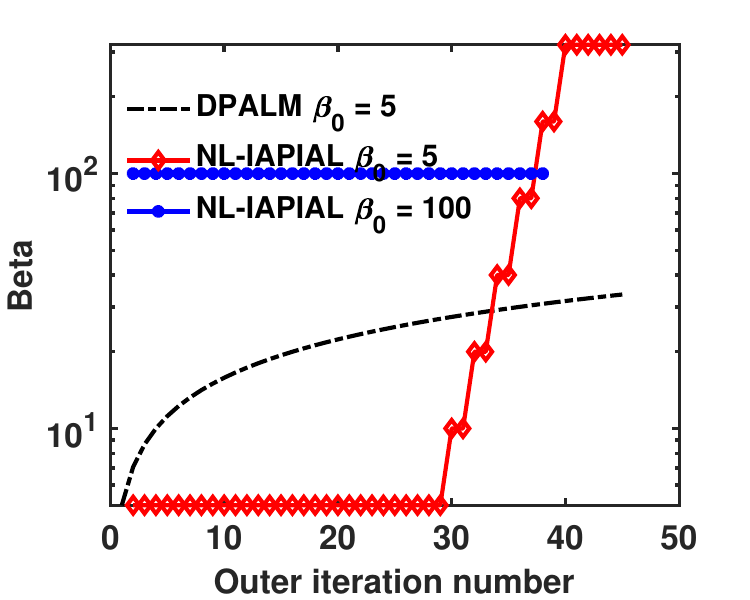}
      \end{tabular}
  \end{center}
  \caption{Evolution of the $\beta$-values for DPALM and NL-IAPIAL \cite{kong2022iteration} on LCQP instances after each outer iteration.} \label{figure: betacompare}
  \end{figure}

{\textbf{Comparison to specialized LCQP solvers.} 
We also compare the proposed method DPALM with a specialized  LCQP solver based on an active-set method~\cite{wong2011active} and MATLAB's built-in solver quadprog. For the active-set method, we follow the implementation as in GitHub\footnote{The PYTHON code is available from \url{https://github.com/JimVaranelli/ActiveSet/tree/master}. For fair {comparison}, we implement it in MATLAB.}. We run 10 independent instances for each value of
$\rho$ in $\{0.1, 1, 10 \}$ and also strongly-convex instances. To ensure a fair comparison, we project the  final iterates obtained by both methods onto the feasible region 
before computing and reporting the objective values at the projected feasible points. In Table~\ref{tab:sqp}, we report the average and variance of the final objective values and computational time in all runs, as well as those by MATLAB's solver quadprog.  We observe that for strongly-convex LCQP instances, both the active set method and our method DPALM achieve globally optimal solutions while the former takes significantly longer time; for weakly convex instances, DPALM produce better-quality solutions, in terms of the objective value, than the active set method and MATLAB's solver quadprog. 
\begin{table}[ht]
    \centering
    \begin{tabular}{|c|cc|cc|cc|}\hline
        & obj. & time & obj. & time & obj. & time \\\hline
        & \multicolumn{2}{|c|}{Active set method in \cite{wong2011active}} & \multicolumn{2}{|c|}{DPALM} & \multicolumn{2}{|c|}{MATLAB's quadprog}\\\hline
        \multicolumn{7}{|c|}{strong convexity constant $\mu=10^{-3}$}\\\hline
        avg. & -2.55e3 & 75.71 & -2.55e3 & 0.90 & -2.55e3 & 0.17 \\
        var.& 1.46e4 & 82.43 & 1.46e4 & 0.02 & 1.46e4 & 6.39e-4 \\\hline
        \multicolumn{7}{|c|}{weak convexity constant $\rho=0.1$}\\\hline
        avg. & 1.39e3 & 29.22 & -3.49e3 & 5.84 & -2.89e3 & 2.65 \\
        var.& 1.14e4 & 14.17 & 1.33e4 & 6.81 & 6.66e4 & 0.1030 \\\hline
        \multicolumn{7}{|c|}{weak convexity constant $\rho=1$}\\\hline
        avg. & 218.13 & 2.31 & -1.33e4 & 8.78 & 7.73e3 & 0.03 \\
        var.& 3.09e3 & 1.38 & 2.26e4 & 7.18 & 1.30e8 & 5.26e-5 \\\hline
        \multicolumn{7}{|c|}{weak convexity constant $\rho=10$}\\\hline
        avg. & 58.50 & 0.14 & -1.17e5 & 6.63 & -1.08e5 & 0.03 \\
        var.& 1.49e3 & 0.009 & 4.04e5 & 4.10 & 9.48e10 & 1.25e-5 \\\hline
    \end{tabular}
    \caption{Comparisons of the objective values at feasible points and running time (in second) by the proposed method DAPLM, the active set method in~\cite{wong2011active}, and MATLAB's solver quadprog.}
    \label{tab:sqp}
\end{table} 
}  
  
		\subsection{Non-convex Quadratically-Constrained Quadratic Program (QCQP)}\label{subsec: QCQP}
		In this subsection, we compare the proposed DPALM method in Alg.~\ref{alg:ialm} to HiAPeM  in~\cite{li2021augmented}, and NL-IAPIAL {and its variant IPL(A)} in~\cite{kong2022iteration} on solving non-convex instances of QCQP in the form of 
		\begin{align}\label{problem: qcqp} 
			\min_{\vx\in \mathbb{R}^d} \frac{1}{2}\vx^{\top}\vQ_0\vx + \vc_0^{\top}\vx, \notag &\st \frac{1}{2}\vx^{\top}\vQ_j\vx + \vc_j^{\top}\vx \le \gamma_j, \\  &\text{ for all }  j \in [m]; x_i \in [l_i, u_i] = [-5,5], i \in [d].
		\end{align}

In the experiment, we set $m =10$ and $d = 1000$.		We 
  vary the weak convexity constant $\rho\in \{0.1, 1, 10\}$ and for each value of $\rho$, we generate 10 instances independently at random, with the data $\{\vQ_j, \vc_j\}_{j=0}^m$ and $\{\gamma_j\}_{j=1}^m$ generated in the way as described in Sect.~\ref{sec:data-gen}. Each algorithm starts from the strict feasible point $\vx^0=\vzero$ and is terminated if an $\varepsilon$-KKT solution is found or after $10^4$ outer iterations, where we set $\varepsilon=10^{-3}$. 
  For DPALM and NL-IAPIAL, we pick the best $\beta_0$ from $\{10^{-4}, 10^{-3}, 10^{-2}, 10^{-1}, 1, 10, 10^2\}$, resulting in $\beta_0 = 10^{-4}$ for the former, $\beta_0 = 10^{-4}$ for $\rho = \{0.1, 1\}$ and $\beta_0 = 10^{-3}$ for $\rho = 10$ for the latter. The setting of HiAPeM is the same as that in the previous test with $\beta_0 = 0.01$ and $N_0 = 10^4$, {and the parameter setting for IPL(A) is the same as that in Section~\ref{subsec: QP}.}

  \begin{figure}
  \begin{center}
      \begin{tabular}{ccc}
          $\rho = 0.1$ & $\rho = 1$ & $\rho = 10$   \\
          \includegraphics[width= 0.29\textwidth]{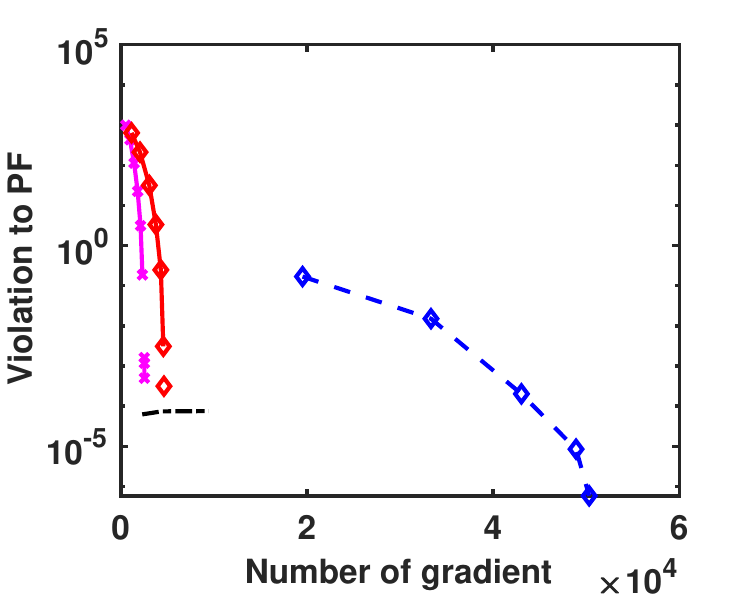} & \includegraphics[width= 0.29\textwidth]{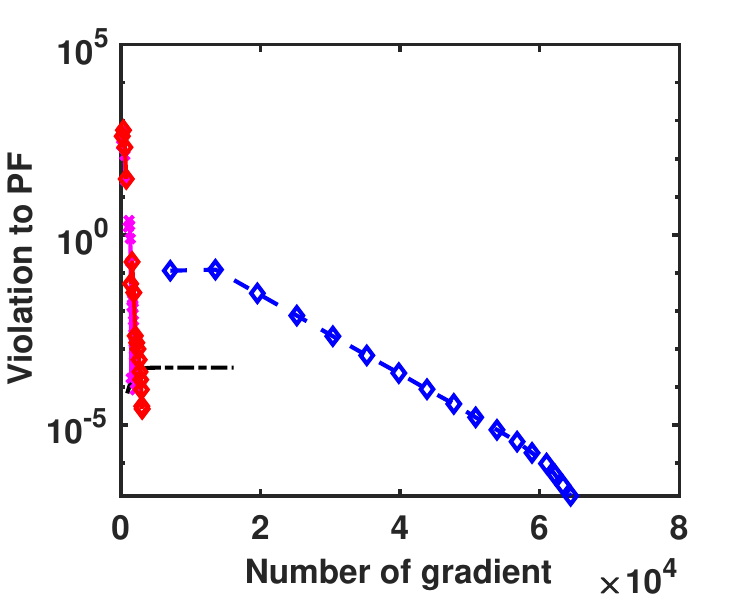} & 
          \includegraphics[width= 0.29\textwidth]{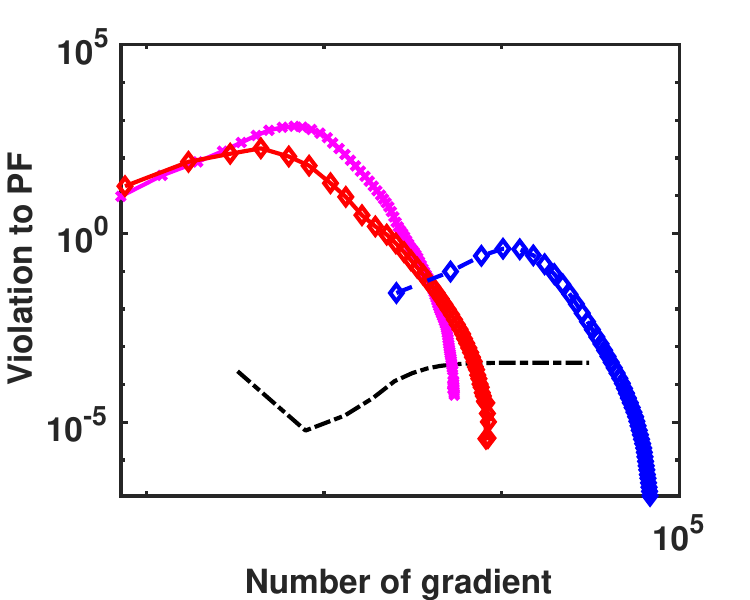}\\
          \includegraphics[width= 0.29\textwidth]{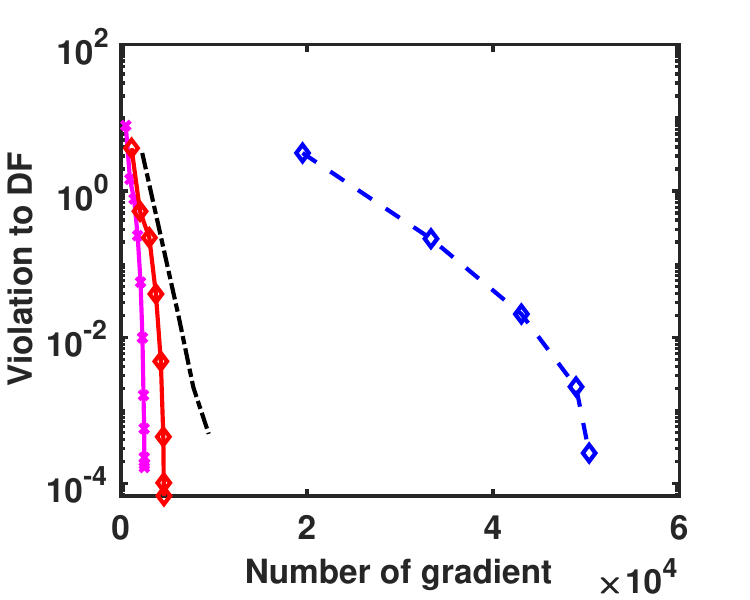} &
          \includegraphics[width= 0.29\textwidth]{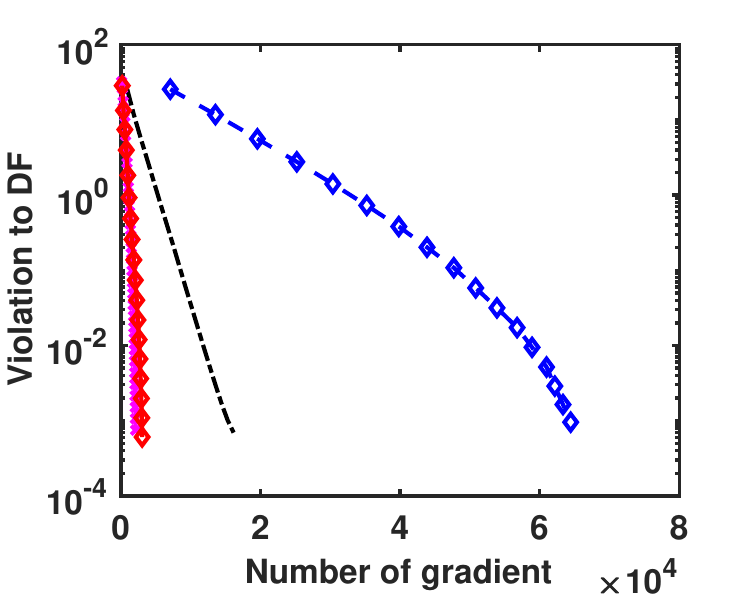} &
          \includegraphics[width= 0.29\textwidth]{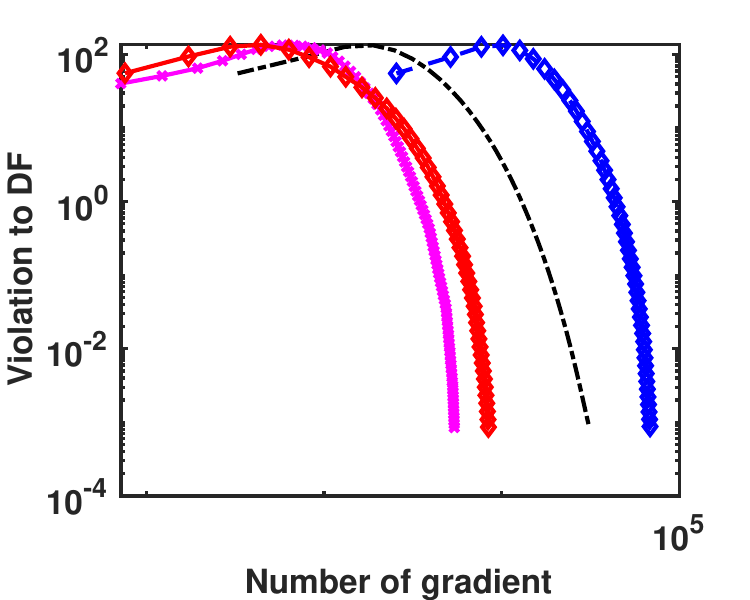} \\
          \includegraphics[width= 0.29\textwidth]{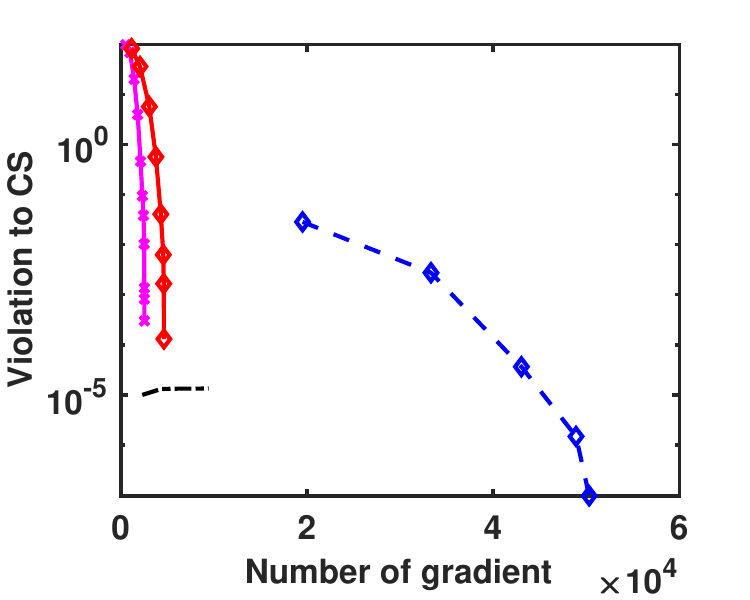} &
          \includegraphics[width= 0.29\textwidth]{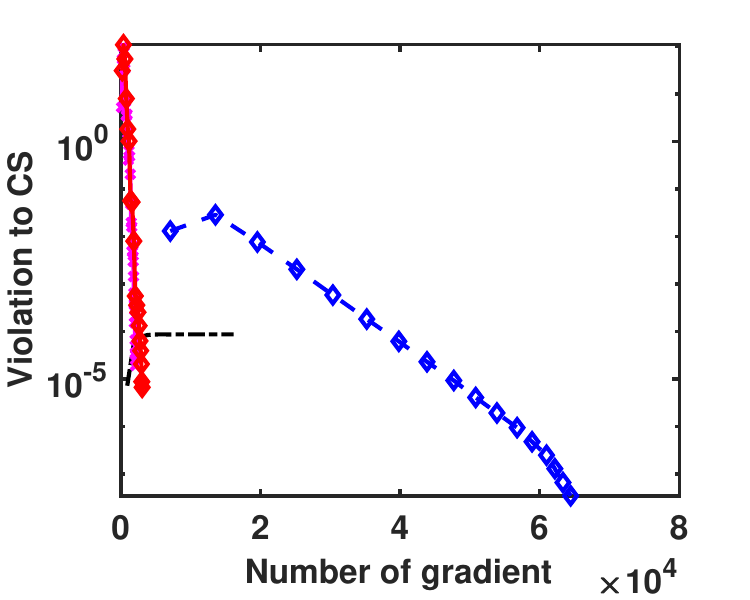} &
          \includegraphics[width= 0.29\textwidth]{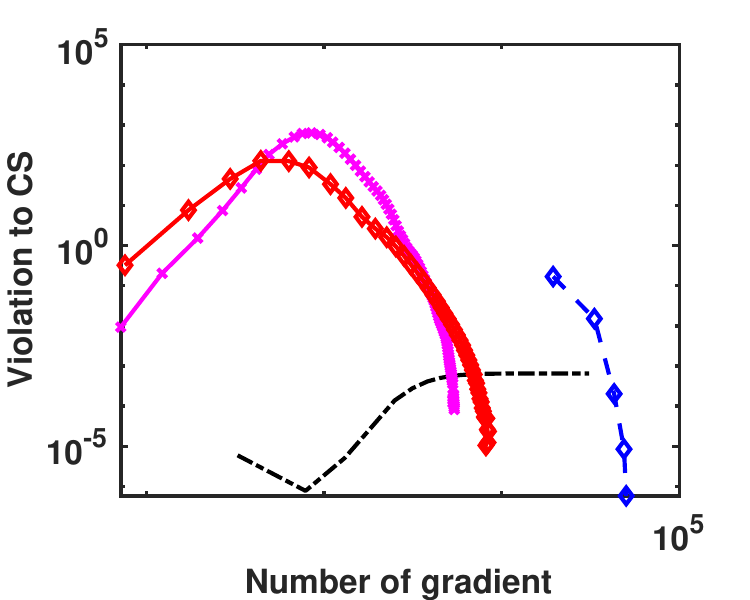} \\
      \end{tabular}
  \includegraphics[width=0.4\textwidth]{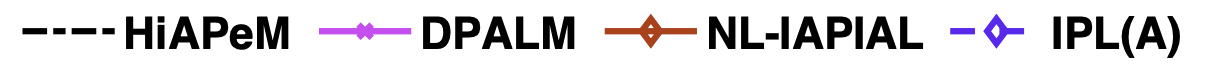}    
  \end{center}
  \caption{Violation to the primal feasibility (PF), dual feasibility (DF), and complementary slackness (CS) conditions vs. Number of gradient evaluations by the proposed DPALM, the HiAPeM method in~\cite{li2021augmented}, and NLIAPIAL and IPL(A) method in~\cite{kong2022iteration} on solving instances of~\eqref{problem: qcqp} with weak convexity constant $\rho\in \{0.1, 1, 10\}$.} \label{table: qcqpfig}
  \end{figure}

  In addition to \verb|pres|, \verb|dres|, \verb|time|, and \verb|#Grad|, we also report violation to CS condition, shortened as \verb|compslack|. Average results with variance are shown in Table~\ref{table:qcqp}. We observe that the proposed DPALM performs the best for all instances. In addition, we plot the results in Fig.~\ref{table: qcqpfig} for the first random instance for each $\rho$.
From the results, we see that to produce a near-KKT point at the same accuracy, our algorithm takes fewer gradient evaluations and less time than all other compared methods. 
This advantage becomes more significant as~$\rho$ increases. 

  \begin{landscape}
\begin{table}
		\vspace{4cm}
				\resizebox{1.5\textwidth}{!}{ 
					\begin{tabular}{|c||ccccc|ccccc|ccccc|ccccc|}
						\hline
						& pres & dres & compslack & time & \#Grad & pres & dres & compslack & time & \#Grad & pres & dres & compslack & time & \#Grad & pres & dres & compslack & time & \#Grad\\\hline\hline
						&\multicolumn{5}{|c|}{HiAPeM with $N_0=10^4$} & \multicolumn{5}{|c|}{NL-IAPIAL} & \multicolumn{5}{|c|}{IPL(A)} & \multicolumn{5}{|c|}{DPALM (proposed)}\\\hline
						\multicolumn{21}{|c|}{weak convexity constant: $\rho=0.1$}\\\hline
						avg. & 8.27e-5 & 5.05e-4 & 1.47e-5& 53.02 & 9642 & 1.40e-4 & 5.78e-5 & 1.94e-4 & 16.3 & 4475 &       5.37e-7 & 2.82e-4 & 8.73e-8 & 208.74 & 51201 & 2.85e-4 & 1.54e-4 & 2.32e-4 & 9.16 & \textbf{2435}\\\hline
                            var. & 1.78e-10 & 9.94e-10 & 8.95e-12 & 3.27 & 15791 & 4.60e-8 & 1.62e-10 & 9.38e-9 & 6.23 & 5.64e5 & 6.13e-15 & 7.92e-10 & 1.57e-16 & 65.92 & 7.68e5 & 5.52e-8 & 6.52e-9 & 2.42e-8 & 1.00 & 6.67e4\\\hline
						\multicolumn{21}{|c|}{weak convexity constant: $\rho=1.0$}\\\hline
						avg. & 3.34e-4 & 8.00e-4 & 9.04e-5& 86.44 & 16225 & 4.34e-5 & 6.81e-4 & 1.13e-5 & 11.51 & 3106 & 1.73e-7 & 8.29e-4 & 4.00e-8 & 260.43 & 65541 & 1.07e-4 & 6.82e-4 & 2.64e-5 & 8.67 & \textbf{2225}\\\hline
                             var. & 4.06e-9 & 1.19e-8 & 3.92e-10 & 7.57 & 1.51e5 & 2.76e-10 & 8.64e-09 & 2.21e-11 & 0.44 & 3.41e4 & 3.01e-15 & 1.88e-8 & 1.52e-16 & 418.77 & 1.01e7  & 1.44e-9 & 1.07e-8 & 7.80e-11 & 0.23 & 1.93e4\\\hline
						\multicolumn{21}{|c|}{weak convexity constant: $\rho=10$}\\\hline
						avg. & 3.81e-4 & 9.05e-4 & 6.60e-4 & 176.28 & 30462 & 1.12e-5 & 8.64e-4 & 2.03e-5 & 21.08 & 8172 & 8.57e-8 & 8.80e-4 & 1.29e-7 & 274.37 & 67090 & 3.91e-5 & 8.45e-4 & 6.09e-5 & 21.08 & \textbf{5173}\\\hline
                        var. & 1.75e-9 & 2.17e-9 & 5.71e-9 & 49.64 & 5.09e5 & 6.70e-11 & 4.00e-9 & 2.13e-10 & 2.33 & 4.39e4 & 5.77e-16 & 3.99e-9 & 1.22e-15 & 68.32 & 2.67e6 & 7.35e-11 & 4.89e-10 & 1.71e-10 & 0.84 & 3.66e4\\\hline
					\end{tabular}
				} \caption{Results by the proposed algorithm DPALM, the HiAPeM method in~\cite{li2021augmented} with $N_0 = 10^4$, and NL-IAPIAL and IPL(A) method \\ in~\cite{kong2022iteration} on solving instances of $\rho$-weakly convex QCQP~\eqref{problem: qcqp} of size $m=10$ and $d=1000$, where $\rho \in\{0.1,1.0,10\}$.}\label{table:qcqp}
		\end{table}
\end{landscape}

		\subsection{Linear Constrained Robust Nonlinear Least Square}\label{subsec: nllsq}
In this subsection, we compare the proposed DPALM method to the inexact Prox-Linear method~\cite[Alg.~2]{drusvyatskiy2019efficiency} on solving a linearly constrained robust nonlinear least square:
		\begin{equation}\label{problem:l-rnls}
			\min_{\vx \in \mathbb{R}^d} \|\vf(\vx)\|_1, \st \vA\vx = \vb, x_i \in [l_i, u_i]=[-5,5], \text{ for all }  i \in [d],
		\end{equation}
		where $\vf:\RR^d \to \RR^m$ is a smooth mapping, and $\vA\in\RR^{n\times d}$. 	
  To apply the method in \cite{drusvyatskiy2019efficiency}, we reformulate the problem in~\eqref{problem:l-rnls} to 
			$\min_{\vx\in [\vl,\vu]}\widehat{h}(\widehat{\vf}(\vx))$,
		where $\vl = [l_1, l_2, \dots, l_d], \vu = [u_1, u_2, \dots, u_d]$, $\widehat{\vf}(\vx) = \left(\begin{matrix} \vf(\vx) \\ \vA\vx-\vb 
		\end{matrix} \right)$, and $\widehat{h}: \RR^{m+n} \to \RR$ is defined as $\widehat{h}\left(\begin{matrix} \vy_1 \\ \vy_2 
		\end{matrix} \right) = \|\vy_1\|_1 + \delta_{\{\mathbf{0}\}}(\vy_2) 
  $ for any $\vy_1 \in \mathbb{R}^m, \vy_2 \in \mathbb{R}^n 
  $. Then we smooth $\widehat{h}$ by its Moreau envelope $\widehat{h}_{\nu}$ for a small $\nu>0$.

  In the experiment, let $\vf(\vx) = (f_1, f_2, \dots, f_m)$ with $f_i(\vx) = \frac{1}{2}\vx^\top \vQ_i\vx + \vc_i^\top \vx$ in~\eqref{problem:l-rnls}.  
  The weak convexity constant of the objective is $\rho= M_lL_f$, where $M_l= \sqrt{m}$ is the Lipschitz constant of $\|\cdot\|_1$ and $L_f$ is the smoothness constant of $\vf(\cdot)$. We set 
  $m=n=10$ and $d=1000$ and generate one random instance with the data $\{\vA, \vb\}$ and $\{\vQ_i, \vc_i\}_{i=1}^m$ generated in the way as described in Sect.~\ref{sec:data-gen}.
  
 From~\eqref{eq:def-x-hat-+},~\eqref{eq:importantine},~\eqref{eq:g2}, and the $\rho$-strong convexity of each subproblem, one can show $$\operatorname{dist}(\vzero , \partial_\vx {\cL}_{\beta_k}( \widetilde{\vx}^{k}; \vy^k,\vz^k)) \leq 4\rho \|\vx^k - \vx^{k-1}\| + 4\sqrt{\frac{\nu_k\rho}{2}} +\frac{\sqrt{2}\varepsilon_k}{\rho},$$ where $\varepsilon_k$ is the error tolerance for solving the $k$-th subproblem. Hence, when $4\sqrt{\frac{\nu_k\rho}{2}} +\frac{\sqrt{2}\varepsilon_k}{\rho}$ is small, we can use $\rho \|\vx^k - \vx^{k-1}\|$ as a measure of stationarity of the iterates. For both methods, we use a small constant smoothing parameter $\nu=10^{-3}$ and $\varepsilon_k  = 10^{-3}$ for each $k$. For DPALM, we simply set $\beta_0 = 1$. They are terminated once $\max\{\|\vA\vx^k-\vb\|, \rho \|\vx^k - \vx^{k-1}\|\} \le \varepsilon = 10^{-2}$ for some $k$. 
  In Fig.~\ref{fig:fairness}, we plot 
the  violation to
   PF and the violation to DF measured by $\rho\|\vx^k - \vx^{k-1}\|$ versus the number of gradient evaluations. It shows that DPALM takes fewer gradient evaluations than the Prox-linear method to reach the same accuracy, 
  though both methods have the same order of oracle complexity in theory. 

  \subsection{Classification with ROC-based fairness}\label{subsec: fairnessROC}
  In this subsection, we compare the proposed DPALM method to the inexact Prox-Linear method~\cite{drusvyatskiy2019efficiency} on solving classification with ROC-based fairness \cite{vogel2021learning}. 
  More specifically, we solve the problem in~\cite[Experiment 6.1]{huang2023single} with a fixed threshold $\theta$. Let $\cP \cup \cU \cup \cD$ be the training data set, where $\cP$, $\cU$, and $\cD$ respectively denote the protected group data, unprotected group data, and other labeled data. 
  We can then learn a fairness model parameterized by $\vx$ via solving 
  \begin{equation}\label{exp4: main}
\begin{aligned}
      \min_{\vx\in \cX} \,\,\,&\ R(\vx):=   \left| \frac{1}{|\cP|} \sum_{(\va, b)\in \cP} \sigma(\vx^\top\va - \theta) - \frac{1}{|\cU|} \sum_{(\va, b)\in \cU} \sigma(\vx^\top\va - \theta) \right| \\
      \st & L(\vx; \cP\cup\cU):=   \frac{1}{2|\cP\cup\cU|}\sum_{(\va, b)\in \cP\cup\cU}(\vx^\top\va - b)^2 \leq \min_{\vx\in \cX}L(\vx; \cD) + \gamma,
  \end{aligned}
  \end{equation}
  where $|\cP|$ denotes the cardinality of $\cP$, $\gamma>0$ is a slackness parameter, $\cX$ is a compact set, and $\sigma(y) = \frac{\exp(y)}{1+\exp(y)}$. Each data sample $(\va, b)$ contains a feature vector $\va$ and a label $b \in \{-1, 1 \}$. The objective of~\eqref{exp4: main} measures the discrepancy between the protected and unprotected groups and thus minimizing it encourages the fairness, while the constraint will maintain a good classification accuracy. 
  
  We use \textit{a9a}~\cite{kohavi1996scaling} and  \textit{COMPAS}~\cite{angwin2022machine} datasets in our experiment, where \textit{a9a} has 32,561 samples of 123 features and \textit{COMPAS} 6,172 samples of 11 features. For each dataset, we randomly select about 1/3 of the samples as $\cD$ and the rest as $\cP\cup \cU$.   The groups are based on gender (male vs. female) in \textit{a9a} and race (African-American vs. non-African-American) in \textit{COMPAS}. We set\footnote{A small radius of the $\infty$-norm ball is used here because the minimizer of $L(\,\cdot\,; \cD)$ is small and the ball constraint will not be active if a large radius is set.} $\cX = \{\vx: \|\vx\|_\infty \le 0.1\}$ and let $\vx_{\text{mat}} = \argmin_{\vx\in \cX} L(\vx; \cD)$. Then, we fix \mbox{$\theta = \frac{1}{|\cD|} \sum_{(\va, b) \in \cD} \va^\top \vx_{\text{mat}}$} and $\gamma = 2L(\vx_{\text{mat}}; \cD)$. We remark here that $\gamma$ is a small positive number, respectively 0.07 for \textit{a9a} data and 0.08 for \textit{COMPAS}.

  \begin{figure}[htbp!]\label{fig : Exp3and4}
			\begin{center}
            \begin{tabular}{ccc}
            {\small{Problem}~\eqref{problem:l-rnls}} &  {\small Problem~\eqref{exp4: main} with \text{a9a}} &  {\small \eqref{exp4: main} with \text{COMPAS}}\\
            \includegraphics[width= 0.28\textwidth]{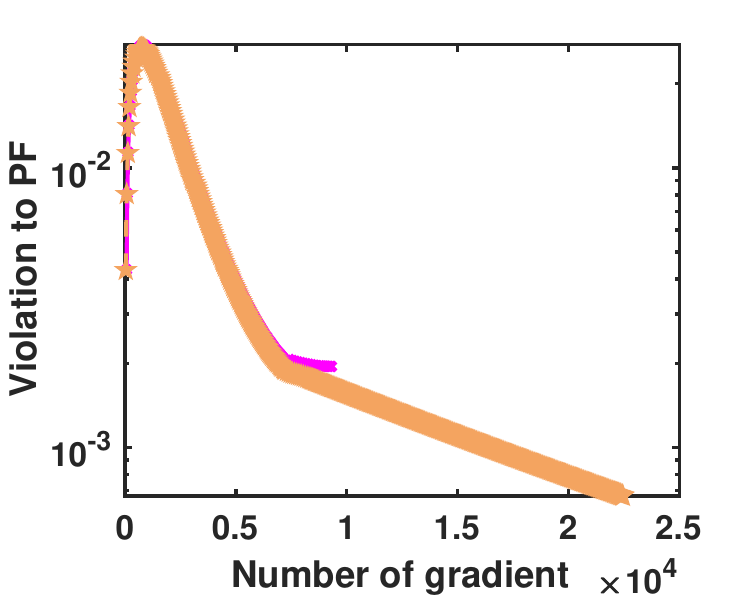}~~ &
				\includegraphics[width= 0.28\textwidth]{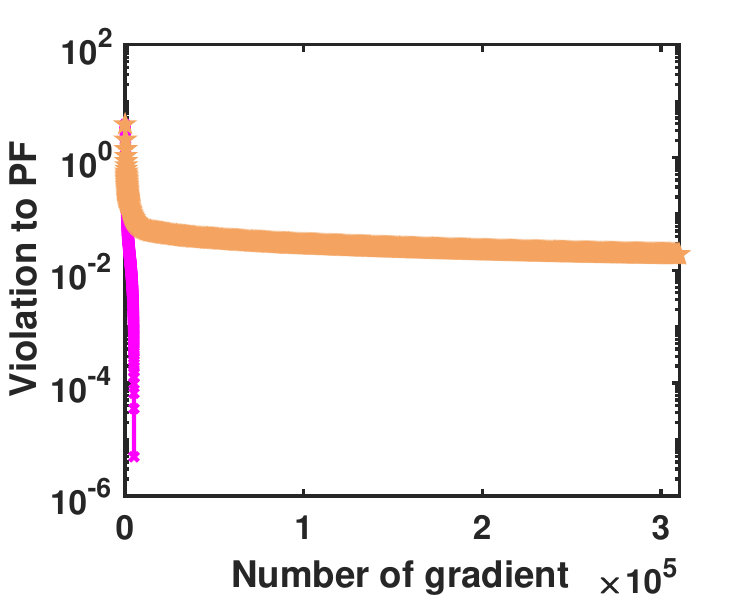} &
                \includegraphics[width= 0.28\textwidth]
                {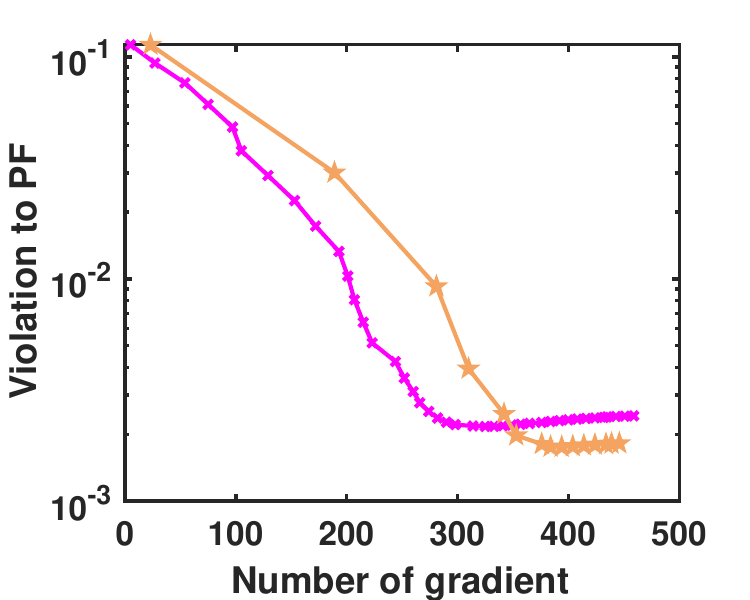}\\
                \includegraphics[width= 0.28\textwidth]{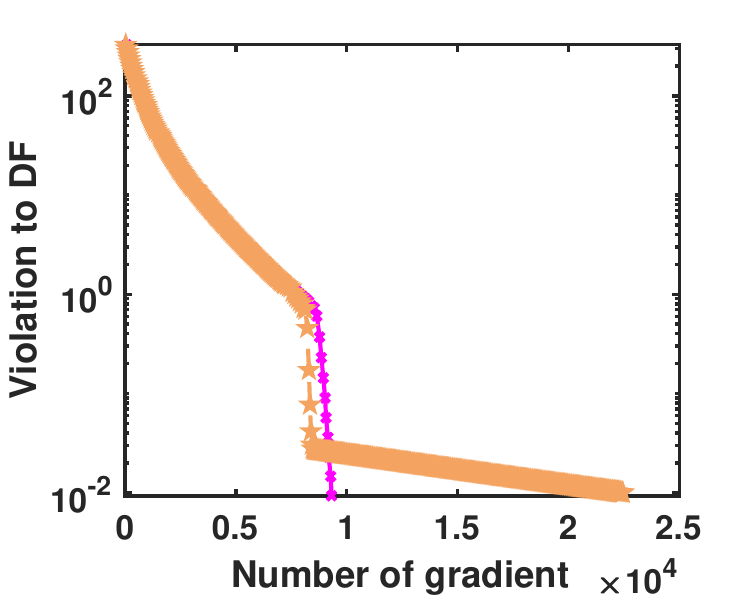} &
                \includegraphics[width= 0.28\textwidth]{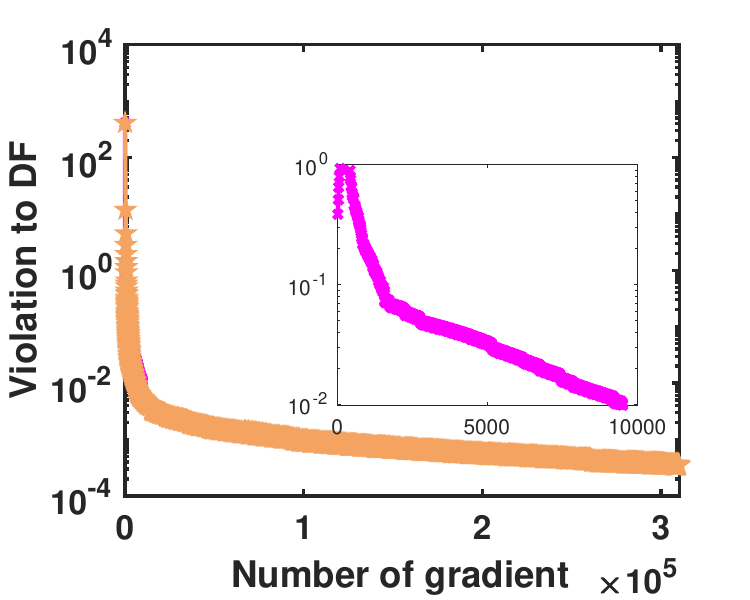} &
                \includegraphics[width= 0.28\textwidth]{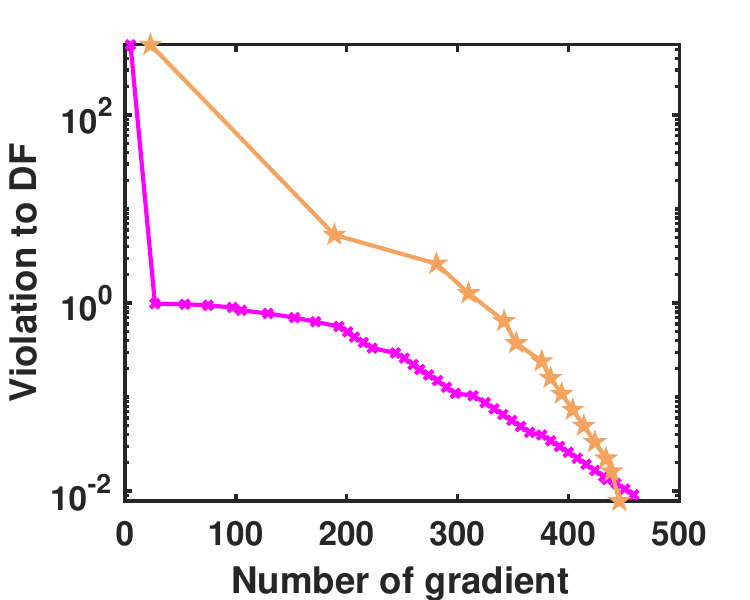}\\
    \end{tabular}
    \includegraphics[width= 0.29\textwidth]{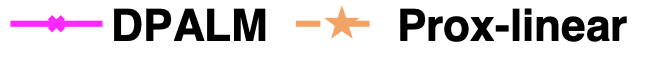}\\
	\end{center}
   \caption{Violation to the PF and DF after each outer iteration vs. Number of gradient evaluations by the proposed DPALM and the Prox-Linear method in \cite{drusvyatskiy2019efficiency} on solving an instance of~\eqref{problem:l-rnls}, and on solving an instance of\eqref{exp4: main} with \textit{a9a}~\cite{kohavi1996scaling} and COMPAS~\cite{angwin2022machine} datasets.}
   \label{fig:fairness}
\end{figure}

For both compared methods, we start from $\vx^0 = \mathbf{1}$ and terminate them once an $\varepsilon$-KKT solution is produced with $\varepsilon = 10^{-2}$. 
We use $\nu = 0.1$ to smooth the absolute function in the objective for both methods and $\nu_c = 10^{-3}$ in smoothing the indicator function of the constraint for the Prox-linear method for it to achieve a small violation to PF. In Fig.~\ref{fig:fairness}, we plot the violation to the PF and DF conditions versus the number of gradients for both methods. 
The results clearly show the higher efficiency of DPALM over Prox-linear in terms of total gradient computation to reach the same level of accuracy.

\section{Conclusions}\label{sec:conclusion}
We propose a damped proximal augmented Lagrangian method (DPALM) for solving a non-convex non-smooth optimization problem, which has a weakly-convex objective and convex affine/nonlinear constraints. We show that DPALM can produce a (near) $\vareps$-KKT point under Slater's condition by solving $\mathcal{O}(\vareps^{-2})$ strongly-convex subproblems, each to a desired accuracy. Also, we have established the overall iteration complexity of DPALM for two cases where~$f$ is smooth or a convex function composed with a smooth mapping.  For the smooth case, with an APG applied to each subproblem, DPALM achieves an $\widetilde \cO(\vareps^{-2.5})$ complexity result to produce an $\vareps$-KKT point, which improves an existing $\widetilde \cO(\vareps^{-3})$ result for proximal ALM based method and matches with the best-known result by quadratic penalty based methods. For the compositional case, with an APG applied to a Moreau-envelope smoothed subproblem, DPALM achieves a complexity result of $\widetilde \cO(\vareps^{-3})$ to produce a near $\vareps$-KKT point, which is new for solving functional constrained compositional problems. 

		
\section*{Acknowledgements} The authors would like to thank three anonymous reviewers and one technical editor for their valuable comments to improve the paper. This work is partly supported by ONR grant N00014-22-1-2573 and NSF grant DMS-2053493.

\section*{Conflict of interest}
The authors declare that they have no conflict of interest.
		
\appendix 
	
\section{Proof of Lemma~\ref{lem:boundxz}} \label{Appendix: multipbound}
In this section, for simplicity of notations, we let $$\widehat{ \vy}^{k+1} := \vy^k + \beta_k\left(\vA \widehat{\vx}^{k+1}-\vb\right),\,\,\widehat{ \vz}^{k+1} :=\left[ \vz^k+\beta_k \vg\left( \widehat{\vx}^{k+1}\right)\right]_+,$$ $$\vp^k = (\vy^k, \vz^k), \text{ and } \widehat{\vp}^{k+1}= (\widehat{\vy}^{k+1}, \widehat{\vz}^{k+1}).$$ 
We first prove the next lemma.
        \begin{lemma}\label{lem: appendbeforeinduction}
        Under Assumptions of Lemma~\ref{lem:boundxz} and with the defined $\overline{B}$ and $\tau$,        it holds
        \begin{align}\label{eq:appendbeforeinduction}
             \frac{1}{\beta_k}\| \widehat{\vp}^{k+1}\|^2 + \min\left\{ \tau \sqrt{\lambda_{\min}^+(\vA\vA^\top)} , \frac{g_{\min}}{2} \right\} \|\widehat{\vp}^{k+1}\| \leq \overline{B} + \frac{1}{\beta_k} \langle \widehat{\vp}^{k+1}, \vp^k \rangle.
        \end{align}   
        \end{lemma}
        
        \begin{proof}
            By~\eqref{eq:approx-cond} and $\overline{\cL}_{\beta_k}$ in~\eqref{eq:barL-func}, there exist $\vxi^{k+1} \in \partial h(\widehat{\vx}^{k+1})$ and $\vxi_f^{k+1} \in \partial \overline{f}(\widehat{\vx}^{k+1})$ such that
            \begin{align}
                \label{eq:vk1}
                \vw^{k+1} = \vxi_f^{k+1} + \vxi^{k+1} + \vA^\top\widehat{\vy}^{k+1} + J_{\vg}(\widehat{\vx}^{k+1})^\top\widehat{\vz}^{k+1} + 2\rho(\widehat{\vx}^{k+1}-\vx^k)
            \end{align}
            satisfies $\|\vw^{k+1} \| \leq \varepsilon_k$.
            For any point $\vx\in \cX$ satisfying $\vA\vx = \vb, \vg(\vx) \leq \mathbf{0}$, it holds by the convexity of $\{g_i\}$ and $\widehat{\vz}^{k+1}\ge\vzero$ that $$\sum_{i=1}^m \widehat{z}_i^{k+1}\langle \vx - \widehat{\vx}^{k+1}, \nabla g_i(\widehat{\vx}^{k+1})\rangle \le \sum_{i=1}^m \widehat{z}_i^{k+1}(g_i(\vx)- g_i(\widehat{\vx}^{k+1})).$$ Hence, we obtain
            \begin{align}\label{eq: appendmain0}
                \langle  \vx - \widehat{\vx}^{k+1}, \vw^{k+1}\rangle \leq& \langle \vx - \widehat{\vx}^{k+1}, \vxi^{k+1} + \vxi_f^{k+1} \rangle + \langle \vb - \vA\widehat{\vx}^{k+1}, \widehat{\vy}^{k+1} \rangle\\\notag
                 &   + \langle \vx - \widehat{\vx}^{k+1}, 2 \rho(\widehat{\vx}^{k+1} - \vx^k)\rangle + \sum_{i=1}^m \widehat{z}_i^{k+1}(g_i(\vx)- g_i(\widehat{\vx}^{k+1})).
             \end{align}
Also, by the definition of $\widehat{\vy}^{k+1}$ and $\widehat{\vz}^{k+1}$, we have 
            \begin{equation}
            \begin{aligned}\label{eq: appendyz2}
                 & \langle \vb - \vA\widehat{\vx}^{k+1}, \widehat{\vy}^{k+1} \rangle = -\frac{1}{\beta_k}\langle \widehat{\vy}^{k+1}, \widehat{\vy}^{k+1} - \vy^k \rangle, \quad \text{and} \\  
                & -\sum_{i=1}^m\widehat{z}_i^{k+1}g_i(\widehat{\vx}^{k+1}) = -\sum_{i=1}^m\widehat{z}_i^{k+1}\frac{\widehat{z}_i^{k+1}-z_i^k}{\beta_k}.
            \end{aligned}
            \end{equation}
            From Assumption~\ref{Assump5}, it follows that $\vxi^{k+1}=\vxi^{k+1}_1 + \vxi^{k+1}_2$, where $\vxi_1^{k+1} \in \cN_{\cX}(\widehat{\vx}^{k+1})$ and $\vxi_2^{k+1} \in B_{r_h}$. By this, we substitute~\eqref{eq: appendyz2} into~\eqref{eq: appendmain0} with $\vx = \vx_{\text{feas}}$ and notice  $\|\vw^{k+1} \| \leq \varepsilon_k \le 1$ 
 to have
            \begin{align}\label{eq: appendmain1}   
            &   \frac{1}{\beta_k}\langle \widehat{\vp}^{k+1}, \widehat{\vp}^{k+1}- \vp^k \rangle + g_{\min}\|\widehat{\vz}^{k+1} \|  
                \le   \frac{1}{\beta_k}\langle \widehat{\vp}^{k+1}, \widehat{\vp}^{k+1}- \vp^k \rangle-\sum_{i=1}^m\widehat{z}_i^{k+1}g_i(\vx_{\text{feas}}) \notag\\ 
                \le &\langle \vx_{\text{feas}} - \widehat{\vx}^{k+1}, \vxi_1^{k+1} \rangle + \langle \vx_{\text{feas}} - \widehat{\vx}^{k+1}, \vxi_2^{k+1} + \vxi_f^{k+1}\rangle +  \langle \vx_{\text{feas}} - \widehat{\vx}^{k+1}, 2\rho(\widehat{\vx}^{k+1} - \vx^k) \rangle \notag \\
                & \quad \quad \quad \quad - \langle \vx_{\text{feas}} - \widehat{\vx}^{k+1}, \vw^{k+1} \rangle\notag \\
                \le &\langle \vx_{\text{feas}} - \widehat{\vx}^{k+1}, \vxi_1^{k+1} \rangle + D(r_h + \overline{B}_f) + 2\rho D^2 + D .
            \end{align}
            We bound the first term in the RHS of~\eqref{eq: appendmain1} as follows. When $\widehat{\vx}^{k+1} \in \text{int}(\cX),$ we have $\cN_{\cX}(\widehat{\vx}^{k+1}) = \{\mathbf{0} \}$ and $\vxi_1^{k+1} = \mathbf0$. When $\widehat{\vx}^{k+1} \in \mathrm{bd}(\cX)$, we see that $$\cM = \{\vx \in \RR^d | (\vx - \widehat{\vx}^{k+1})^\top \vxi_1^{k+1} = 0 \}$$ is a supporting hyperplane of $\cX$ at $\widehat{\vx}^{k+1}$. Hence $$ 0 < \dist(\vx_{\text{feas}}, \mathrm{bd}(\cX)) \le \dist(\vx_{\text{feas}}, \cM) =\frac{|(\widehat{\vx}^{k+1} - \vx_{\text{feas}})^\top \vxi_1^{k+1}|}{\|\vxi_1^{k+1}\|}. $$
            We then have
            \begin{align}\label{ineq: hyperplanexi1}
                (\widehat{\vx}^{k+1} - \vx_{\text{feas}})^\top \vxi_1^{k+1} &= |(\widehat{\vx}^{k+1} - \vx_{\text{feas}})^\top \vxi_1^{k+1}| \notag \\ &= \dist(\vx_{\text{feas}}, \cM) \|\vxi_1^{k+1} \| \geq \dist(\vx_{\text{feas}}, \mathrm{bd}(\cX)) \|\vxi_1^{k+1} \|,
            \end{align}
            where the first equality follows from $\vxi_1^{k+1} \in \cN_{\cX}(\widehat{\vx}^{k+1})$ so that $(\widehat{\vx}^{k+1} - \vx_{\text{feas}})^\top \vxi_1^{k+1} \geq 0$.
            
            Since $\vy^0 \in \text{Range}(\vA)$, it holds $ \widehat{\vy}^{k+1} \in \text{Range}(\vA)$ for all $k \geq 0$. This implies that $\sqrt{\lambda_{\min}^+(\vA\vA^\top)}\|\widehat{\vy}^{k+1}\|\leq \|\vA^\top\widehat{\vy}^{k+1} \|$. 
            Hence, from $\|\vw^{k+1}\| \leq \varepsilon_k \le 1$,~\eqref{eq:vk1}, 
            and the triangle inequality, we obtain 
            \begin{align}\label{eq: appendlambdayy}
               \sqrt{\lambda_{\min}^+(\vA\vA^\top)}\|\widehat{\vy}^{k+1}\| 
               \leq & \varepsilon_k + \| \vxi_1^{k+1}\| + \|\vxi_2^{k+1} + \vxi_f^{k+1} +2\rho(\widehat{\vx}^{k+1}-\vx^k)\| + \|J_{\vg}(\widehat{\vx}^{k+1})^\top \widehat{\vz}^{k+1} \| \notag \\
               \le & \| \vxi_1^{k+1}\| + 1 + {r_h }+ 2\rho D + \overline{B}_f + \sqrt{m}B_g\|\widehat{\vz}^{k+1}\|.
                \end{align}
Now multiplying~\eqref{eq: appendlambdayy} by $\tau$, adding it to~\eqref{eq: appendmain1}, and using~\eqref{ineq: hyperplanexi1}, we have
            \begin{align*}
                &\quad\,\,\dist(\vx_{\text{feas}}, \mathrm{bd}(\cX)) \|\vxi_1^{k+1} \| - \tau \|\vxi_1^{k+1}\|+  \frac{1}{\beta_k}\langle \widehat{\vp}^{k+1}, \widehat{\vp}^{k+1}- \vp^k \rangle \\
                & \quad \quad \quad \quad + \tau \sqrt{\lambda_{\min}^+(\vA\vA^\top)}\|\widehat{\vy}^{k+1} \|
                + g_{\min}\| \widehat{\vz}^{k+1}\|\\  
                &\leq D(r_h + \overline{B}_f) + 2\rho D^2 + D + \tau \left(1 + {r_h }+ 2\rho D + \overline{B}_f + \sqrt{m}B_g\|\widehat{\vz}^{k+1}\|\right).
            \end{align*}
 We obtain the desired result from the inequality above, the definition of $\overline{B}$, and the choice of $\tau$.           
        \end{proof}

        Now, we are ready to prove Lemma~\ref{lem:boundxz}.
        \begin{proof} \textit{(of Lemma~\ref{lem:boundxz})}
            We prove the lemma by induction. The desired result holds trivially for $k = 0$. Suppose that it holds for $k=i$, i.e., 
            $\|\vp^{i}\| \leq B_0 + \sum_{j = 0}^{i-1} \widetilde{\delta}_j$. Then for $k = i+1$, we have 
            \begin{align*}
                &  \left(\min\left\{ \tau \sqrt{\lambda_{\min}^+(\vA\vA^\top)} , \frac{g_{\min}}{2} \right\} + \frac{\|\widehat{\vp}^{i+1}\|}{\beta_{i}}  \right)\|\widehat{\vp}^{i+1} \|
                \overset{\eqref{eq:appendbeforeinduction}}\leq \overline{B} + \frac{\|\widehat{\vp}^{i+1} \| \|\vp^{i}\|}{\beta_{i}}\\
                \leq &   \overline{B} + \frac{\|\widehat{\vp}^{i+1} \|(B_0 + \sum_{j = 0}^{i-1} \widetilde{\delta}_j) }{\beta_{i}}
                =   \min\left\{ \tau \sqrt{\lambda_{\min}^+(\vA\vA^\top)} , \frac{g_{\min}}{2} \right\}\theta + \frac{\|\widehat{\vp}^{i+1} \| (B_0 + \sum_{j = 0}^{i-1} \widetilde{\delta}_j)}{\beta_{i}}\\
                \leq &  \left( \min\left\{ \tau \sqrt{\lambda_{\min}^+(\vA\vA^\top)} , \frac{g_{\min}}{2} \right\} + \frac{\|\widehat{\vp}^{i+1} \|}{\beta_{i}}\right)(B_0 + \sum_{j = 0}^{i-1} \widetilde{\delta}_j),
            \end{align*}
            where 
            the equality follows from the definition of $\theta$, and the third inequality is by 
            $B_0 \ge \theta$. 
            Hence $\|\widehat{\vp}^{i+1}\| \leq B_0 + \sum_{j = 0}^{i-1} \widetilde{\delta}_j$. 
       Moreover, by the update of $\vy^{i+1}$ and $\vz^{i+1}$ and the triangle inequality, we have
            \begin{align*}
                & \quad\,\, \|\vp^{i+1}\| \\ 
                & \leq \left\| \left(\vy^i + \alpha_i(\vA\widehat{\vx}^{i+1} - \vb), \vz^i + \alpha_i \max\left\{\frac{-\vz^i}{\beta_i}, \vg(\widehat{\vx}^{i+1})\right\}\right ) \right\| + \alpha_i \|\vA(\widehat{\vx}^{i+1} - \vx^{i+1}) \| \\
                & \quad \quad \quad
                \quad \quad 
                 + \alpha_i \left\|   \max \left\{ -\frac{\vz^i}{\beta_i}, \vg({{\vx}^{i+1}}) \right\} - \max \left\{ -\frac{\vz^i}{\beta_i}, \vg({\widehat{\vx}^{i+1}}) \right\} \right\|\\
                &  \leq \left\| \left(\frac{\alpha_i}{\beta_i}\widehat{\vy}^{i+1} + \left(1-\frac{\alpha_i}{\beta_i}\right)\vy^i ,  \frac{\alpha_i}{\beta_i}\widehat{\vz}^{i+1} + \left(1-\frac{\alpha_i}{\beta_i} \right)\vz^i \right)\right\| + \alpha_i \| \vA\|\|\widehat{\vx}^{i+1} - \vx^{i+1}\| \\
                & \quad \quad \quad +\alpha_i \| \vg(\vx^{k+1}) - \vg(\widehat{\vx}^{k+1})\|\\
                &  \leq \left\| \left(\frac{\alpha_i}{\beta_i}\widehat{\vy}^{i+1}, \frac{\alpha_i}{\beta_i}\widehat{\vz}^{i+1}\right) + \left( \left(1-\frac{\alpha_i}{\beta_i}\right)\vy^i,  \left(1-\frac{\alpha_i}{\beta_i} \right)\vz^i  \right)  \right\| \\
                & \quad \quad \quad + \alpha_i\| \widehat{\vx}^{i+1} - \vx^{i+1}\| ( \|\vA\| + \sqrt{m} B_g ) \\
                &  \leq \left\| \left(\frac{\alpha_i}{\beta_i}\widehat{\vy}^{i+1}, \frac{\alpha_i}{\beta_i}\widehat{\vz}^{i+1}\right) \right\| + \left\| \left( \left(1-\frac{\alpha_i}{\beta_i}\right)\vy^i,  \left(1-\frac{\alpha_i}{\beta_i} \right)\vz^i  \right) \right\| + \alpha_i \delta_i (\|\vA\| + \sqrt{m} B_g)\\
                &=   \frac{\alpha_i}{\beta_i}\|\widehat{\vp}^{i+1} \| + \left(1-\frac{\alpha_i}{\beta_i}\right) \| \vp^i\| + \alpha_i \delta_i (\|\vA\| + \sqrt{m} B_g) \\
                & \leq B_0 + \sum_{j = 0}^{i-1} \widetilde{\delta}_j + \alpha_i \delta_i (\|\vA \| + \sqrt{m} B_g) = B_0 + \sum_{j = 0}^{i} \widetilde{\delta}_j,
            \end{align*}
            where 
            the third inequality follows from~\eqref{eq: boundg},
            and the last inequality is from the bound of $\|\widehat{\vp}^{i+1}\|$ and $\|\vp^i\|$. Hence, the desired result also holds for $k=i+1$. The proof is then completed.
        \end{proof}

         \section{Proof of Lemma~\ref{lem:tilderhodeltak}}\label{section: primalboundproof}
         The next lemma 
         follows from the second equation in the proof of~\cite[Lemma~7]{xu2021iteration-ialm} by viewing~$\vp^+$ as the dual iterate obtained after one update of the inexact ALM to a strongly convex problem from the initial iterate $\vp$.		
		\begin{lemma}\label{lemma: pbarbound}
			Let $\vx^* = \argmin_\vx \{\widehat{f}(\vx), \st \vA\vx = \vb, \vg(\vx) \leq \vzero\}$, where $\widehat{f}$ is strongly convex, and each component function in $\vg$ is convex. 
			Let $\mathcal{L}_{\beta}(\vx;\vp)$ be the AL function 
			with a multiplier $\vp = (\vy,\vz)$ and a penalty parameter $\beta>0$. 
Suppose $\vx^*$ is a KKT point of the problem with a corresponding multiplier $\vp^*=(\vy^*,\vz^*)$. 			Start from any $\vp$ with $\vz\ge \vzero$; let $\widehat{\vx}$ satisfy $\mathcal{L}_{\beta}\left(\widehat{\vx}; \vp \right) \leq \min_{\vx} \mathcal{L}_{\beta}\left(\vx; \vp \right) + \delta$ for some $\delta \ge 0$; 
			 set $\vy^+ = \vy + \beta\left( \vA\widehat{\vx} - \vb\right), \ \vz^+ = \left[\vz + \beta \vg( \widehat{\vx} )\right]_+ 
    $. 
			Then  
			$\left\|\vp^+-\vp^*\right\|^2 \leq \left\|\vp-\vp^*\right\|^2 + 2 \beta \delta.$ 
		\end{lemma}
  
  Now, we are ready to give the proof of Lemma~\ref{lem:tilderhodeltak}.
	\begin{proof} \textit{(of Lemma~\ref{lem:tilderhodeltak})}
    Consider the following problem
		\begin{equation}\label{problem: subprobleWithConstraints}
			\min_\vx \,\,\overline{f}(\vx) + h(\vx) + \rho\left\|\vx-\vx^k\right\|^2,\quad \st \vA\vx = \vb,\,\, \vg(\vx) \leq \vzero.
		\end{equation}
	Since $\overline{f}$ is $\rho$-weakly convex, the objective function in~\eqref{problem: subprobleWithConstraints} is $\rho$-strongly convex. In addition, the constraints are convex. 
Thus~\eqref{problem: subprobleWithConstraints} has a unique solution $\widehat{\vx}_*^k$, and under Assumption~\ref{Assump3}, there must exist a multiplier $\widehat{\vp}_*^k = (\widehat{\vy}_*^k,\widehat{\vz}_*^k)$ corresponding to $\widehat{\vx}_*^k$ such that the KKT system holds.
		
		 By the definition of $\widehat{\vp}^{k+1}$ and the triangle inequality, we have 
   \begin{equation}
		\begin{aligned}\label{eq: residual}
		 	\left\| \vA \widehat{\vx}^{k+1}- \vb\right\|^2 + \left\|\left[\vg\left( \widehat{\vx}^{k+1}\right)\right]_+\right\|^2 &\leq \frac{1}{\beta_k^2} \left\| \widehat{\vp}^{k+1}- \vp^k\right\|^2 \\ &\leq \frac{1}{\beta_k^2}\left(\left\| \widehat{\vp}^{k+1}-\widehat{\vp}_*^k\right\| + \left\|{\widehat{ \vp}_*^k} - \vp^k\right\|\right)^2.
		\end{aligned}
  \end{equation}
Also, let $\vx^k_*=\argmin_\vx \overline{\mathcal{L}}_{\beta_k}( {\vx}; \vy^k, \vz^k)$; then by the $\rho$-strong convexity of $\overline{\cL}_{\beta_k}( \vx; \vy^k, \vz^k )$ and~\eqref{eq:approx-cond}, 
it holds 
that
		\begin{equation}\label{eq:obj-opt-error}\overline{\mathcal{L}}_{\beta_k}( \widehat{\vx}^{k+1}; \vy^k, \vz^k) - \min_\vx \overline{\mathcal{L}}_{\beta_k}( {\vx}; \vy^k, \vz^k) \leq \langle \widetilde{\vxi}, \widehat{\vx}^{k+1} - \vx^k_*  \rangle - \frac{\rho}{2}\|\widehat{\vx}^{k+1} - \vx^k_*\|^2 \leq \frac{1}{\rho}\| \widetilde{\vxi}\|^2 
		\leq \frac{ \varepsilon_k^2}{\rho},
		\end{equation}
  where  $\widetilde{\vxi} \in \partial_\vx \widetilde{\mathcal{L}}_{\beta_k}( \widehat{\vx}^{k+1}; \vy^k, \vz^k)$ such that $\|\widetilde{\vxi}\|\le\varepsilon_k$. 
  Hence, from Lemma~\ref{lemma: pbarbound}, 
  it follows that
		\begin{align}
			\label{eq:pk1barp}
		 	\left\|\widehat{\vp}^{k+1} - \widehat{\vp}_*^k \right\|^2 \leq \left\|\vp^{k}-\widehat{\vp}_*^k \right\|^2 + \frac{2\beta_k\varepsilon_k^2}{\rho}.
		\end{align}		
		Moreover, 
		using 
		\cite[Lemma~3]{lin2022complexity} to~\eqref{problem: subprobleWithConstraints} gives 
		$
		\left\|\widehat{\vz}_*^k\right\| \leq \frac{Q}{g_{\min}},\,\,
		\left\|\widehat{\vy}_*^k\right\| \leq Q\left\|( \vA \vA^{\top})^{\dagger} \vA\right\|C_1,
		$
		where $Q$ and $C_1$ are defined in~\eqref{eq:Q-C-1-Case1}.		
			Noticing
			$
			\left\|\vp^k\right\| \leq B_p,
			$
			we obtain that
				$$\left\|\widehat{\vp}^k_* - \vp^k\right\| \leq B_p + \sqrt{\frac{Q^2}{g_{\min}^2} + Q^2\left\|( \vA \vA^{\top})^{\dagger} \vA\right\|^2C_1^2},$$
			which together with~\eqref{eq: residual} and~\eqref{eq:pk1barp} imply
   \begin{equation}
			\begin{aligned}\label{eq:bd-pf-hat-k}
			 	\sqrt{\left\| \vA \widehat{\vx}^{k+1}- \vb\right\|^2 + \left\|\left[\vg\left( \widehat{\vx}^{k+1}\right)\right]_+\right\|^2} 
				&\leq \frac{1}{\beta_k}\left(2\left\|\widehat{\vp}_*^k-\vp^k\right\| + \varepsilon_k \sqrt{\frac{2\beta_k}{\rho}}\right) \\
            &\leq \frac{1}{\beta_k}\left(2\left\|\widehat{ \vp}_*^k-\vp^k\right\| + 1\right) 
                \le \frac{\widehat{C}_p}{\beta_k}.
			\end{aligned}
   \end{equation}
where in the second inequality, we have used $\vareps_k \le \sqrt{\frac{\rho}{2\beta_k}}$.		Now	we have
            \begin{align*}
                &\|\vA\vx^{k+1} - \vb\| + \|[\vg(\vx^{k+1})]_+ \| \\
                =&   \|\vA\widehat{\vx}^{k+1} - \vb + \vA(\vx^{k+1} - \widehat{\vx}^{k+1})\| + \|[\vg(\widehat{\vx}^{k+1}) + (\vg(\vx^{k+1}) - \vg(\widehat{\vx}^{k+1}))]_+ \|\\
                \leq & \|\vA\widehat{\vx}^{k+1} - \vb\|  + \|\vA(\vx^{k+1} - \widehat{\vx}^{k+1})\| + \|[\vg(\widehat{\vx}^{k+1})]_+\| + \|[(\vg(\vx^{k+1}) - \vg(\widehat{\vx}^{k+1}))]_+ \| \\
                \leq& \frac{2\widehat{C}_p}{\beta_k} + \delta_k( \|\vA \| + \sqrt{m}B_g),
            \end{align*}
            where the first inequality is by the triangle inequality and $[a + b]_+ \leq [a]_+ + [b]_+$, and the second one holds from~\eqref{eq: boundg} and~\eqref{eq:bd-pf-hat-k}. The desired result then follows from the inequality above.
		\end{proof}
  
        		\section{Proof of Lemma~\ref{lem:boundggg}} \label{AppendixA}
		 \begin{proof}
		Denote 
		$$
			J_1^k = \{i: -\frac{z_i^k}{\beta_k} \geq g_i(\vx^{k+1})\},\,\,J_2^k = \{i: -\frac{z_i^k}{\beta_k} < g_i(\vx^{k+1})\},
			$$ $$
       J_+^k := \{i : g_i( \vx^{k+1}) \geq 0\}, \text{ and }J_-^k := \{i : g_i( \vx^{k+1}) < 0\}.$$
			From the updating rule of $\vz^{k+1}$, we have 
            \begin{align*}
                & \text{$z_i^{k+1} = \left(1-\frac{\alpha_k}{\beta_k}\right) z_i^k$ and $g_i \left(\vx^{k+1} \right) < 0$ for all $i \in J_1^k$;}\\
                & \text{$z_i^{k+1} = z_i^k + \alpha_k g_i\left( \vx^{k+1}\right)$ and $g_i \left(\vx^{k+1} \right) < 0$  for all $ i \in {J}_2^k \cap J_-^k$;}  \\
                & \text{$z_i^{k+1} = z_i^k + \alpha_k g_i\left( \vx^{k+1}\right) \text{ and } g_i \left(\vx^{k+1} \right) \geq 0$ for all $i \in {J}_2^k \cap J_+^k$.}
            \end{align*}
			Below, we look at $\frac{\beta_{k+1}}{2}\big\|[\vg(\vx^{k+1})+\frac{ \vz^{k+1}}{\beta_{k+1}}]_+\big\|^2 - \frac{\beta_k}{2}\big\|[\vg(\vx^{k+1})+\frac{ \vz^k}{\beta_k}]_+\big\|^2$ for these three cases.
			
			\underline{Case I: $i \in J_1^k $}:
		We have $$g_i\left( \vx^{k+1}\right)+\frac{z_i^{k+1}}{\beta_{k+1}}=g_i\left( \vx^{k+1}\right) + \frac{1}{\beta_{k+1}}\left(1-\frac{\alpha_k}{\beta_k}\right) z_i^k
        \leq g_i\left( \vx^{k+1}\right) + \frac{1}{\beta_{k}} z_i^k \leq 0 ,$$ as $0<\beta_k \leq \beta_{k+1}$, $0\le \alpha_k\le \beta_k$, and $\vz^k\ge \vzero$.
			
			Hence,
			$
			\frac{\beta_{k+1}}{2} [g_i(\vx^{k+1}) + \frac{z_i^{k+1}}{\beta_{k+1}}  ]_+^2 - \frac{\beta_k}{2} [g_i(\vx^{k+1}) + \frac{ z_i^k}{\beta_k}]_+^2
			= 0. 
			$
			
			\underline{Case II: $i \in {J}_2^k\cap J_-^k$}: We have $\beta_k g_i \left(\vx^{k+1} \right) + z_i^k > 0$ and $g_i\left(\vx^{k+1}\right) < 0$. These two inequalities imply $g_i(\vx^{k+1}) + \frac{ z_i^k}{\beta_k}\ge 0$. We next consider two subcases based on the sign of $(\beta_{k+1} + \alpha_k) g_i \left(\vx^{k+1} \right) + z_i^k$.
			
			Subcase (i): When $(\beta_{k+1} + \alpha_k) g_i \left(\vx^{k+1} \right) + z_i^k \geq 0$, we have 
			\begin{align}\label{eq: xdifflemmayyy}
				&\quad \frac{\beta_{k+1}}{2} \left[g_i\left(\vx^{k+1}\right) + \frac{1}{\beta_{k+1}} \left( z_i^k + \alpha_kg_i\left(\vx^{k+1}\right)\right)\right]_+^2 - \frac{\beta_k}{2} \left[g_i\left(\vx^{k+1}\right) + \frac{ z_i^k}{\beta_k}\right]_+^2 \notag\\
				&  = \frac{1}{2\beta_{k+1}} \left[ (\beta_{k+1} +\alpha_k)^2 g_i^2(\vx^{k+1}) + 2(\beta_{k+1} + \alpha_k) z_i^kg_i\left(\vx^{k+1}\right)+ \left(z_i^k\right)^2\right] \notag \\ 
                & \quad \quad \quad \quad -\frac{\beta_k}{2} \left(g_i\left(\vx^{k+1}\right) + \frac{ z_i^k}{\beta_k}\right)^2 \notag\\
				& =  \frac{\alpha_k}{\beta_{k+1}}g_i\left(\vx^{k+1}\right)\left( z_i^k + (\beta_{k+1} + \frac{\alpha_k}{2}) g_i\left(\vx^{k+1}\right) \right) + \frac{\beta_{k+1}-\beta_k}{2}(g_i(\vx^{k+1}))^2 \notag \\
                & \quad \quad \quad \quad + \left(\frac{1}{2\beta_{k+1}} - \frac{1}{2\beta_k}\right)\left(z_i^k\right)^2 \\
				&   \leq \frac{\beta_{k+1}-\beta_k}{2}g_i^2(\vx^{k+1}) + \left(\frac{1}{2\beta_{k+1}} - \frac{1}{2\beta_k}\right)\left(z_i^k\right)^2 \notag \\
				 &\leq \frac{\beta_{k+1}-\beta_k}{2} \frac{\left(z_i^k\right)^2}{\beta^2_{k}} + \left(\frac{1}{2\beta_{k+1}} - \frac{1}{2\beta_k}\right)\left(z_i^k\right)^2 
				= \frac{(\beta_k - \beta_{k+1})^2}{2\beta^2_k\beta_{k+1}}\left(z_i^k\right)^2 \notag,
			\end{align}
			where the first inequality holds because 
			the first term in~\eqref{eq: xdifflemmayyy} is negative, 
			and the last inequality follows from $g_i^2(\vx^{k+1}) \leq \frac{\left(z_i^k\right)^2}{\beta^2_{k}}$. 
            
			Subcase (ii): When $(\beta_{k+1} + \alpha_k) g_i \left(\vx^{k+1} \right) + z_i^k < 0$, it is obvious that
   $$
   \frac{\beta_{k+1}}{2} [g_i(\vx^{k+1}) + \frac{1}{\beta_{k+1}} ( z_i^k + \alpha_kg_i(\vx^{k+1}))]_+^2
   - \frac{\beta_k}{2} [g_i(\vx^{k+1}) + \frac{ z_i^k}{\beta_k}]_+^2 \leq 0,
   $$
   as $[g_i(\vx^{k+1}) + \frac{1}{\beta_{k+1}} ( z_i^k + \alpha_kg_i(\vx^{k+1}))]_+=0$.
			
			\underline{Case III: $i \in J_2^k\cap J_+^k$}: We still have \eqref{eq: xdifflemmayyy} by $g_i(\vx^{k+1}) + \frac{ z_i^k}{\beta_k}\ge 0$. We then know from $\beta_k \le \beta_{k+1}$ that 
			\begin{align*}
				& \frac{\beta_{k+1}}{2} \left[g_i\left(\vx^{k+1}\right) + \frac{1}{\beta_{k+1}} \left( z_i^k + \alpha_kg_i\left(\vx^{k+1}\right)\right)\right]_+^2 - \frac{\beta_k}{2} \left[g_i\left(\vx^{k+1}\right) + \frac{ z_i^k}{\beta_k}\right]_+^2\\
				 \leq &\left( \frac{\beta_{k+1}-\beta_k}{2} + \alpha_k + \frac{\alpha_k^2}{2\beta_{k+1}}\right)\left[g_i(\vx^{k+1})\right]_+^2 + \frac{\alpha_k} {\beta_{k+1}} z_i^kg_i\left(\vx^{k+1}\right).
			\end{align*}
   Combining the above three cases, we get 
			\begin{equation}
				\label{eq:eachiter}
				\begin{aligned}
					& \frac{\beta_{k+1}}{2}\left\|\left[\vg\left(\vx^{k+1}\right) + \frac{ \vz^{k+1}}{\beta_{k+1}}\right]_+\right\|^2 - \frac{\beta_k}{2}\left\|\left[\vg\left(\vx^{k+1}\right) + \frac{ \vz^k}{\beta_k}\right]_+\right\|^2\\
					  \leq &\left(\frac{\beta_{k+1}-\beta_k}{2} + \alpha_k + \frac{\alpha_k^2}{2\beta_{k+1}}\right)\left\|\left[\vg\left(\vx^{k+1}\right)\right]_+\right\|^2 + \frac{\alpha_k}{\beta_{k+1}}\left\| \vz^k\right\|\left\|\left[\vg\left( \vx^{k+1}\right)\right]_+\right\| \\
                    & \quad \quad \quad \quad + \frac{(\beta_k - \beta_{k+1})^2}{2\beta^2_k\beta_{k+1}}\left\|\vz^k \right\|^2.
				\end{aligned}
			\end{equation}
			Summing up the inequality in~\eqref{eq:eachiter} from $k = 0$ to $K-1$ yields 
			\begin{align*}
				&\quad\sum_{k=0}^{K-1}\left(\frac{\beta_{k+1}}{2}\left\|\left[\vg\left(\vx^{k+1}\right)+\frac{ \vz^{k+1}}{\beta_{k+1}}\right]_+\right\|^2 - \frac{\beta_k}{2}\left\|\left[\vg\left(\vx^{k+1}\right)+\frac{ \vz^k}{\beta_k}\right]_+\right\|^2\right)\\
				&\leq \sum_{k=0}^{K-1} \left(\left (
				\frac{\beta_{k+1}-\beta_k}{2} + \alpha_k + \frac{\alpha_k^2}{2\beta_{k+1}}\right)\left\|\left[\vg\left(\vx^{k+1}\right)\right]_+ \right\|^2 
				+ \frac{\alpha_k}{\beta_{k+1}}\| \vz^k\|\left\|\left[\vg\left(\vx^{k+1}\right)\right]_+\right\| \right. \\
                &\left. \quad \quad \quad \quad + \frac{(\beta_k - \beta_{k+1})^2}{2\beta^2_k\beta_{k+1}}\left\|\vz^k \right\|^2\right) \\
				&= \sum_{k=0}^{K-1} \left( \frac{\beta_{k+1}-\beta_k}{2}\left\|\left[\vg\left(\vx^{k+1}\right)\right]_+\right\|^2 
				+ \alpha_k \left\|\left[\vg\left(\vx^{k+1}\right)\right]_+\right\|^2 + \frac{\alpha_k}{2\beta_{k+1}} \alpha_k \left\|[\vg(\vx^{k+1})]_+\right\|^2  \right) \\
				&\quad \quad \quad \quad    + \sum_{k=0}^{K-1} \left( \frac{\| \vz^k\|}{\beta_{k+1}}\alpha_k\left\|\left[\vg\left(\vx^{k+1}\right)\right]_+\right\|\right) + \sum_{k=0}^{K-1} \frac{(\beta_k - \beta_{k+1})^2}{2\beta^2_k\beta_{k+1}}\left\|\vz^k \right\|^2 \\
				&  \leq \sum_{k=0}^{K-1}\frac{\beta_{k+1}-\beta_k}{2}\left\|\left[\vg\left(\vx^{k+1}\right)\right]_+\right\|^2 + \sum_{k=0}^{K-1} \left(v_k \left\|\left[\vg\left(\vx^{k+1}\right)\right]_+\right\| + \frac{\beta_kv_k\|[\vg\left(\vx^{k+1}\right)]_+\|}{2\beta_{k+1}} \right)  \\
				& \quad \quad \quad \quad   + \sum_{k=0}^{K-1} \frac{\| \vz^k\|}{\beta_{k+1}}v_k + \sum_{k=0}^{K-1} \frac{(\beta_k - \beta_{k+1})^2}{2\beta^2_k\beta_{k+1}}\left\|\vz^k \right\|^2\\
				&  \leq \sum_{k=0}^{K-1}\frac{\beta_{k+1}-\beta_k}{2}\left\|\left[\vg\left(\vx^{k+1}\right)\right]_+\right\|^2 + \sum_{k=0}^{K-1}\frac{3}{2}v_k \left\|\left[\vg\left(\vx^{k+1}\right)\right]_+\right\| + \sum_{k=0}^{K-1}\frac{v_k}{\beta_{k+1}}\| \vz^k\| \\
                & \quad \quad \quad \quad + \sum_{k=0}^{K-1} \frac{(\beta_k - \beta_{k+1})^2}{2\beta^2_k\beta_{k+1}}\left\|\vz^k \right\|^2\\
				&  \leq\sum_{k=0}^{K-1}\frac{{C}_p^2(\beta_{k+1}-\beta_k)}{2\beta_k^2} + \frac{3}{2}\sum_{k=0}^{K-1} v_k \frac{{C}_p}{\beta_k}  + \sum_{k=0}^{K-1}\frac{v_k}{\beta_{k}}B_p + \sum_{k=0}^{K-1} \frac{\beta_{k+1} - \beta_{k}}{2\beta^2_k}B_p^2,
			\end{align*}
			where the second inequality holds from $\alpha_k \|[\vg(\vx^{k+1})]_+\| \leq v_k$ and $\alpha_k \leq \beta_k$, the third one is due to $\beta_k \leq \beta_{k+1}$, 
			the fourth one uses Lemma~\ref{lem:boundxz} and inequality~\eqref{eq: diffL}.
   
			The proof of~\eqref{eq:lem:boundggg} is then completed.
		{To prove~\eqref{eq:lem: yyybound}, we start by noticing
			$$\left\langle \vy^{k+1}- \vy^k, \vA \vx^{k+1}- \vb\right\rangle = \alpha_k\| \vA \vx^{k+1}- \vb\|^2 \leq v_k\|\vA \vx^{k+1}- \vb\|$$ from the update of $\vy^k$ and definition of $\alpha_k$}. 
			Hence, by Lemma~\ref{lem:boundxz}, the definitions of $v_k,\beta_k$, and~\eqref{eq: diffL}, it follows that
			$$
			\sum_{k=0}^{K-1} \langle \vy^{k+1}- \vy^k, \vA \vx^{k+1}- \vb\rangle \leq {C}_p\sum_{k=0}^{K-1} \frac{v_k}{\beta_k}.
			$$
			In addition, from inequality~\eqref{eq: diffL}, we have
			$$
			 \sum_{k=0}^{K-1} \frac{\beta_{k+1}-\beta_k}{2}\| \vA \vx^{k+1}- \vb\|^2 
			 \leq {C}_p^2\sum_{k=0}^{K-1} \frac{\beta_{k+1}-\beta_k}{2\beta_k^2}.
			$$
			Therefore, we complete the proof.
            \end{proof}
   
  		\section{Nesterov's Accelerated Proximal Gradient (APG) Method} \label{AppendixB}
		In this section, we review Nesterov's APG method in \cite{nesterov2013gradient} for solving composite convex problems in the form of 
		\begin{align}\label{Subproblem: CompositeNesterov}
			\min_{\vx \in \RR^d} \phi(\vx) := \widetilde{f}(\vx) + \widetilde{h}(\vx),
		\end{align}
		where $\widetilde{f}$ is convex and $L_{\widetilde{f}}$-smooth, and $\widetilde{h}$ is closed and $\mu$-strongly convex with  $\mu>0$. 
The algorithm is shown in Alg.~\ref{alg:acceleratedNesterov}, where $$\phi'(\mathcal{M}_{L}(\vy)) :=  L(\vy-\mathcal{M}_{L}(\vy)) + \nabla \widetilde{f}(\mathcal{M}_{L}(\vy)) - \nabla \widetilde{f}(\vy), \text{ }\mathcal{M}_{L}(\vy) := \prox_{\frac{\widetilde{h}}{L}}(\vy - \frac{1}{L}\nabla \widetilde{f}(\vy))$$ for some $L>0$, and $$\psi_{t+1}(\vx) = \psi_t({\vx}) + a_{t+1}[\widetilde{f}(\vx^{t+1}) + \langle \nabla \widetilde{f}(\vx^{t+1}), \vx - \vx^{t+1} \rangle + \widetilde{h}(\vx)]$$ with a positive sequence $\{a_t\}$, and $\psi_0(\vx) = \frac{1}{2}\|\vx-\vx^0\|^2$. 
		\begin{algorithm}[htbp!]
			\caption{{Nesterov's Accelerated Proximal Gradient (APG) Method} for~\eqref{Subproblem: CompositeNesterov}} \label{alg:acceleratedNesterov}
			\DontPrintSemicolon
			{\small
				\textbf{Initialization:} choose $ \vx^0, L:=L_0, \gamma_d \geq \gamma_u > 1, \Delta > 0$ and set $A_0 = 0, \vv^0 = \vy^0 = \vy =  \vx^0.$ 
				
				\For{$t=0,1,\ldots $}{
					\While{$\left\langle \phi'(\mathcal{M}_{L}(\vy)), \vy - \mathcal{M}_{L}(\vy) \right\rangle < \frac{1}{{L}} \|\phi'(\mathcal{M}_{L}(\vy))\|^2$}
					{     $L \gets {L}\gamma_u$;
						let $a>0$ and satisfy $\frac{a^2}{A_t + a} = 2 \frac{1 + \mu A_t}{L}$; 
					        $\vy = \frac{A_t\vx^t + a\vv^t}{A_t + a}$. 
						}
					
					Set $L_{t+1} \gets L$,
					$\vy^t \gets \vy, a_{t+1} = a$, $L \gets \frac{L_{t+1}}{\gamma_d}$, 
					$ \vx^{t+1} \gets \mathcal{M}_{L_{t+1}}(\vy^t)$,
					$A_{t+1} = A_t + a_{t+1}$. 
                    Let $\vv^{t+1} := \argmin_\vx \psi_{t+1}(\vx) = \text{prox}_{A_{t+1}\widetilde{h}(\cdot)} (\vx^0 - \sum_{i=1}^{t+1}a_i\nabla \widetilde{f}(\vx^i))$. 
					
					\textbf{if} {$\dist(\mathbf{0},\partial\phi(\vx^{t+1})) \leq \Delta$} \textbf{then}{
						output $\vx^{t+1}$ and stop.
					}
			}}
		\end{algorithm}
		
		The next theorem gives the number of iterations for Alg.~\ref{alg:acceleratedNesterov} to output the desired solution. 
		\begin{theorem}\label{Th: TboundNesterov}
			Suppose that $\{\vx^t\}$ is the sequence generated by Alg.~\ref{alg:acceleratedNesterov} and $\mathrm{dom}(\widetilde{h})$ is bounded with a diameter {$D = \max_{\vx,\vx'\in \mathrm{dom}(\widetilde{h})} \|\vx-\vx'\| < \infty$}. Then $\dist(\mathbf{0},\partial \phi(\vx^{t+1})) \leq \Delta, $ for all  
                $$
                t\ge T  := \left\lceil\max\left\{\frac{1}{\log 2}, 2{\sqrt{\frac{\gamma_uL_{\widetilde{f}}}{2\mu}}}\right\} \log \frac{3(1+\gamma_u)DL_{\widetilde{f}}\sqrt{\frac{2\gamma_uL_{\widetilde{f}}}{\mu}}}{2\Delta}\right\rceil + 1.
			$$
		\end{theorem}

	\begin{proof}
	Let $\vx^*$ be the minimizer of problem~\eqref{Subproblem: CompositeNesterov}. Then from~\cite[Theorem 6]{nesterov2013gradient}, it holds 
			\begin{equation}\label{eq:mainNesterovResult}
			 	\phi(\vx^{t+1}) - \phi(\vx^*) \leq \frac{\gamma_uL_{\widetilde{f}}}{4}\|\vx^{*}-\vx^{0}\|^2\left[ 1 + \sqrt{\frac{\mu}{2\gamma_uL_{\widetilde{f}}}}\right]^{-2t}.
			\end{equation}
   {By the optimality condition in the definition of $\mathcal{M}_{L_{t+1}}(\vy^t)$, 
   we have $$\nabla \widetilde{f}(\vx^{t+1})-\nabla \widetilde{f}(\vy^t) - L_{t+1}\left(\vx^{t+1}-\vy^t\right) \in \partial \phi(\vx^{t+1}).$$ Also,  from \cite[Eqn.~(4.11)]{nesterov2013gradient}, it holds $L_{t+1} \leq \gamma_u L_{\widetilde{f}}$.  Using these, we obtain} 
			
			\begin{align}\label{ineq:subgradphiBound}
				\dist\left(\vzero,\partial \phi\left(\vx^{t+1}\right)\right) &\leq \left\| \nabla \widetilde{f}(\vx^{t+1})-\nabla \widetilde{f}(\vy^t) - L_{t+1}\left(\vx^{t+1}-\vy^t\right)\right\| \notag \\
				& \leq L_{\widetilde{f}}\left\|\vx^{t+1}-\vy^t \right\|+L_{t+1}\left\|\vx^{t+1}-\vy^t\right\| \notag\\
				&\leq L_{\widetilde{f}}(1+\gamma_u)\left\|\vx^{t+1}-\vy^t\right\|
				\leq L_{\widetilde{f}}(1+\gamma_u)\left(\left\|\vx^{t+1}-\vx^t\right\| + \left\|\vx^{t}-\vy^t\right\|\right).
			\end{align}
			Notice that $\psi_t$ is a $(\mu A_t+1)$-strongly convex function. Hence, 
			\begin{align}\label{eq: psi_x_int}
				\left\|\vx^t-\vv^t\right\|^2 &   \leq \frac{2}{\mu A_t +1}\left(\psi_t(\vx^t) - \psi_t^*\right)
				\leq \frac{2}{\mu A_t+1} \left(A_t\phi(\vx^t) - \psi_t^* + \frac{1}{2}\left\|\vx^t-\vx^{0}\right\|^2\right)\notag\\
				&   \leq \frac{1}{\mu A_t + 1}\left\|\vx^t-\vx^{0}\right\|^2
				\leq \frac{D^2}{\mu A_t+1},
			\end{align}
			where $\psi_t^* := \min_\vx \psi_t(\vx)$, and the second/third inequality is from~\cite[Eqn.~(4.4)]{nesterov2013gradient}. 
			
			Combining the $\mu$-strong convexity of $\phi(\vx)$ and~\eqref{eq:mainNesterovResult} gives 
   \begin{align}
        \frac{\mu}{2} \|\vx^{t+1} - \vx^*\|^2 \leq \phi(\vx^{t+1}) - \phi(\vx^*) \leq \frac{\gamma_uL_{\widetilde{f}}}{4}\|\vx^{*}-\vx^{0}\|^2\left[ 1 + \sqrt{\frac{\mu}{2\gamma_uL_{\widetilde{f}}}}\right]^{-2t},
   \end{align}
   which implies
			\begin{equation}\label{eq: xTbound}
		 		\left\|\vx^{t+1}-\vx^*\right\| \leq \sqrt{\frac{\gamma_uL_{\widetilde{f}}}{2\mu}}\left\|\vx^*-\vx^{0}\right\|\left[ 1 + \sqrt{\frac{\mu}{2\gamma_uL_{\widetilde{f}}}}\right]^{-t}
				\leq D\sqrt{\frac{\gamma_uL_{\widetilde{f}}}{2\mu}}\left[ 1 + \sqrt{\frac{\mu}{2\gamma_uL_{\widetilde{f}}}}\right]^{-t}.
			\end{equation}
			Since 
			$
			\vy^t = \frac{A_t\vx^t+a_{t+1}\vv^t}{A_t + a_{t+1}}
			$
			and $a_{t+1}>0$, it must hold $\|\vy^t-\vx^t\| \leq \|\vv^t-\vx^t\|$. Using this in~\eqref{ineq:subgradphiBound} and combining it with~\eqref{eq: psi_x_int}, we get
			\begin{equation}\label{eq: subgradLtildeIntermediate}
			 	\dist\left(\vzero,\partial \phi\left(\vx^{t+1}\right)\right) \leq (1+\gamma_u)L_{\widetilde{f}}\left(\left\|\vx^{t+1}-\vx^t\right\| + \frac{D}{\sqrt{\mu A_t + 1}}\right).
			\end{equation}
			By the triangle inequality, we have $\left\|\vx^{t+1}-\vx^t\right\| \leq \left\|\vx^{t+1}-\vx^*\right\| + \left\|\vx^t-\vx^*\right\|$, which together with~\eqref{eq: xTbound} gives 
			\begin{align}\label{eq:xTxT+1bound}
			 	\left\|\vx^{t+1}-\vx^t\right\| 
				\leq D\sqrt{\frac{2\gamma_uL_{\widetilde{f}}}{\mu}} \left[ 1 + \sqrt{\frac{\mu}{2\gamma_uL_{\widetilde{f}}}}\right]^{-t+1}.
			\end{align}
			Now by~\eqref{eq: subgradLtildeIntermediate},~\eqref{eq:xTxT+1bound}, and  $ A_t \geq \frac{2}{(\gamma_uL_{\widetilde{f}})} \left[ 1 + \sqrt{\frac{\mu}{(2\gamma_uL_{\widetilde{f}}})}\right]^{2(t-1)}$ from~\cite[Lemma 8]{nesterov2013gradient}, we have 
			\begin{align}\label{eq: finalsubproblemresult}
			 	\dist\left(\vzero,\partial \phi\left(\vx^{t+1}\right)\right) \leq \frac{3}{2}(1+\gamma_u)DL_{\widetilde{f}}\left( \sqrt{\frac{2\gamma_uL_{\widetilde{f}}}{\mu}} \right)\left[ 1 + \sqrt{\frac{\mu}{2\gamma_uL_{\widetilde{f}}}}\right]^{-t+1}.
			\end{align}
		This implies	$\dist(\mathbf{0},\partial \phi(\vx^{t+1})) \leq \Delta, $ for all $ t\ge T$ from the definition of $T$ and $\frac{1}{\log(1+x)} \le \frac{2}{x}$ for all $ x \in (0,1)$.
		\end{proof}

\bibliographystyle{abbrv}
\bibliography{citations.bib}

@article{xu2022first,
  title={First-order methods for problems with ${O} (1)$ functional constraints can have almost the same convergence rate as for unconstrained problems},
  author={Xu, Yangyang},
  journal={SIAM Journal on Optimization},
  volume={32},
  number={3},
  pages={1759--1790},
  year={2022},
  publisher={SIAM}
}

@article{chen2021hyperspectral,
  title={Hyperspectral image denoising using factor group sparsity-regularized non-convex low-rank approximation},
  author={Chen, Yong and Huang, Ting-Zhu and He, Wei and Zhao, Xi-Le and Zhang, Hongyan and Zeng, Jinshan},
  journal={IEEE Transactions on Geoscience and Remote Sensing},
  volume={60},
  pages={1--16},
  year={2021},
  publisher={IEEE}
}

@article{xu2012alternating,
  title={An alternating direction algorithm for matrix completion with nonnegative factors},
  author={Xu, Yangyang and Yin, Wotao and Wen, Zaiwen and Zhang, Yin},
  journal={Frontiers of Mathematics in China},
  volume={7},
  pages={365--384},
  year={2012},
  publisher={Springer}
}

@article{xu2020primal,
  title={Primal-dual stochastic gradient method for convex programs with many functional constraints},
  author={Xu, Yangyang},
  journal={SIAM Journal on Optimization},
  volume={30},
  number={2},
  pages={1664--1692},
  year={2020},
  publisher={SIAM}
}

@article{xu2021first-ALM,
  title={First-order methods for constrained convex programming based on linearized augmented {Lagrangian} function},
  author={Xu, Yangyang},
  journal={INFORMS Journal on Optimization},
  volume={3},
  number = {1},
  pages = {89--117},
  year={2021}
}

@article{xu2021iteration-ialm,
  title={Iteration complexity of inexact augmented {Lagrangian} methods for constrained convex programming},
  author={Xu, Yangyang},
  journal={Mathematical Programming},
  volume={185},
  pages={199--244},
  number = {185},
  year={2021},
  publisher={Springer}
}

@article{lin2022complexity,
  title={Complexity of an inexact proximal-point penalty method for constrained smooth non-convex optimization},
  author={Lin, Qihang and Ma, Runchao and Xu, Yangyang},
  journal={Computational Optimization and Applications},
  volume={82},
  number={1},
  pages={175--224},
  year={2022},
  publisher={Springer}
}

@book{nesterov2003introductory,
  title={Introductory lectures on convex optimization: A basic course},
  author={Nesterov, Yurii},
  volume={87},
  year={2003},
  publisher={Springer Science \& Business Media}
}

@article{nesterov2013gradient,
  title={Gradient methods for minimizing composite functions},
  author={Nesterov, Yu},
  journal={Mathematical Programming},
  volume={140},
  number={1},
  pages={125--161},
  year={2013},
  publisher={Springer}
}

@article{lu2023iteration,
  title={Iteration-complexity of first-order augmented {L}agrangian methods for convex conic programming},
  author={Lu, Zhaosong and Zhou, Zirui},
  journal={SIAM journal on optimization},
  volume={33},
  number={2},
  pages={1159--1190},
  year={2023},
  publisher={SIAM}
}

@article{rigollet2011neyman,
  title={Neyman-pearson classification, convexity and stochastic constraints},
  author={Rigollet, Philippe and Tong, Xin},
  journal={Journal of Machine Learning Research},
  volume={12},
  number={Oct},
  pages={2831--2855},
  year={2011}
}

@inproceedings{li2021rate,
  title={Rate-improved inexact augmented {Lagrangian} method for constrained non-convex optimization},
  author={Li, Zichong and Chen, Pin-Yu and Liu, Sijia and Lu, Songtao and Xu, Yangyang},
  booktitle={International Conference on Artificial Intelligence and Statistics},
  pages={2170--2178},
  year={2021}
}

@inproceedings{sahin2019inexact,
  title={An inexact augmented {Lagrangian} framework for non-convex optimization with nonlinear constraints},
  author={Sahin, Mehmet Fatih and Alacaoglu, Ahmet and Latorre, Fabian and Cevher, Volkan and others},
  booktitle={Advances in Neural Information Processing Systems},
  pages={13965--13977},
  year={2019}
}

@article{kong2019complexity,
  title={Complexity of a quadratic penalty accelerated inexact proximal point method for solving linearly constrained non-convex composite programs},
  author={Kong, Weiwei and Melo, Jefferson G and Monteiro, Renato DC},
  journal={SIAM Journal on Optimization},
  volume={29},
  number={4},
  pages={2566--2593},
  year={2019},
  publisher={SIAM}
}

@article{li2021augmented,
  title={Augmented {Lagrangian}--Based First-Order Methods for Convex-Constrained Programs with Weakly Convex Objective},
  author={Li, Zichong and Xu, Yangyang},
  journal={INFORMS Journal on Optimization},
  volume={3},
  number={4},
  pages={373--397},
  year={2021},
  publisher={INFORMS}
}

@inproceedings{lin2018level,
  title={Level-set methods for finite-sum constrained convex optimization},
  author={Lin, Qihang and Ma, Runchao and Yang, Tianbao},
  booktitle={International conference on machine learning},
  pages={3112--3121},
  year={2018}
}

@article{ghadimi2016accelerated,
  title={Accelerated gradient methods for non-convex nonlinear and stochastic programming},
  author={Ghadimi, Saeed and Lan, Guanghui},
  journal={Mathematical Programming},
  volume={156},
  number={1-2},
  pages={59--99},
  year={2016},
  publisher={Springer}
}

@article{lan2013iteration-penalty,
  title={Iteration-complexity of first-order penalty methods for convex programming},
  author={Lan, Guanghui and Monteiro, Renato DC},
  journal={Mathematical Programming},
  volume={138},
  number={1-2},
  pages={115--139},
  year={2013},
  publisher={Springer}
}

@book{bazaraa2006nonlinear,
  title={Nonlinear programming: theory and algorithms},
  author={Bazaraa, Mokhtar S and Sherali, Hanif D and Shetty, Chitharanjan M},
  year={2006},
  publisher={John Wiley \& Sons}
}

@article{yu2017simple,
  title={A Simple Parallel Algorithm with an ${O}(1/t)$ Convergence Rate for General Convex Programs},
  author={Yu, Hao and Neely, Michael J},
  journal={SIAM Journal on Optimization},
  volume={27},
  number={2},
  pages={759--783},
  year={2017},
  publisher={SIAM}
}

@book{clarke1990optimization,
  title={{Optimization and non-smooth Analysis}},
  author={Clarke, Frank H},
  volume={5},
  year={1990},
  publisher={SIAM},
  address={Philadelphia}
}

@article{kong2022iteration,
  title={Iteration complexity of a proximal augmented {Lagrangian} method for solving non-convex composite optimization problems with nonlinear convex constraints},
  author={Kong, Weiwei and Melo, Jefferson G and Monteiro, Renato DC},
  journal={Mathematics of Operations Research},
  year={2022},
  publisher={INFORMS}
}

@article{zeng2022moreau,
  title={Moreau envelope augmented {Lagrangian} method for non-convex optimization with linear constraints},
  author={Zeng, Jinshan and Yin, Wotao and Zhou, Ding-Xuan},
  journal={Journal of Scientific Computing},
  volume={91},
  number={2},
  pages={61},
  year={2022},
  publisher={Springer}
}

@article{drusvyatskiy2019efficiency,
  title={Efficiency of minimizing compositions of convex functions and smooth maps},
  author={Drusvyatskiy, Dmitriy and Paquette, Courtney},
  journal={Mathematical Programming},
  volume={178},
  pages={503--558},
  year={2019},
  publisher={Springer}
}

@inproceedings{allen2017natasha,
  title={Natasha: Faster non-convex stochastic optimization via strongly non-convex parameter},
  author={Allen-Zhu, Zeyuan},
  booktitle={International Conference on Machine Learning},
  pages={89--97},
  year={2017}
}

@article{davis2019stochastic,
  title={Stochastic model-based minimization of weakly convex functions},
  author={Davis, Damek and Drusvyatskiy, Dmitriy},
  journal={SIAM Journal on Optimization},
  volume={29},
  number={1},
  pages={207--239},
  year={2019},
  publisher={SIAM}
}

@article{davis2018stochastic,
  title={Stochastic subgradient method converges at the rate $ \mathcal{O}(k^{-1/4}) $ on weakly convex functions},
  author={Davis, Damek and Drusvyatskiy, Dmitriy},
  journal={ArXiv, preprint:1802.02988},
  year={2018}
}

@article{davis2019proximally,
  title={Proximally guided stochastic subgradient method for non-smooth, non-convex problems},
  author={Davis, Damek and Grimmer, Benjamin},
  journal={SIAM Journal on Optimization},
  volume={29},
  number={3},
  pages={1908--1930},
  year={2019},
  publisher={SIAM}
}

@article{ghadimi2013stochastic,
  title={Stochastic first-and zeroth-order methods for non-convex stochastic programming},
  author={Ghadimi, Saeed and Lan, Guanghui},
  journal={SIAM Journal on Optimization},
  volume={23},
  number={4},
  pages={2341--2368},
  year={2013},
  publisher={SIAM}
}

@article{lan2019accelerated,
  title={Accelerated stochastic algorithms for non-convex finite-sum and multi-block optimization},
  author={Lan, Guanghui and Yang, Yu},
  journal={SIAM Journal on Optimization},
  volume={29},
  number={4},
  pages={2753--2784},
  year={2019},
  publisher={SIAM}
}

@inproceedings{reddi2016stochastic,
  title={Stochastic variance reduction for non-convex optimization},
  author={Reddi, Sashank J and Hefny, Ahmed and Sra, Suvrit and Poczos, Barnabas and Smola, Alex},
  booktitle={International conference on machine learning},
  pages={314--323},
  year={2016}
}

@article{zhang2018convergence,
  title={On the convergence rate of stochastic mirror descent for non-smooth non-convex optimization},
  author={Zhang, Siqi and He, Niao},
  journal={ArXiv, preprint:1806.04781},
  year={2018}
}

@article{aybat2011first,
  title={A first-order smoothed penalty method for compressed sensing},
  author={Aybat, Necdet S and Iyengar, Garud},
  journal={SIAM Journal on Optimization},
  volume={21},
  number={1},
  pages={287--313},
  year={2011},
  publisher={SIAM}
}

@article{lan2016iteration,
  title={Iteration-complexity of first-order augmented {Lagrangian} methods for convex programming},
  author={Lan, Guanghui and Monteiro, Renato DC},
  journal={Mathematical Programming},
  volume={155},
  number={1-2},
  pages={511--547},
  year={2016},
  publisher={Springer}
}

@article{liu2019nonergodic,
  title={On the nonergodic convergence rate of an inexact augmented {Lagrangian} framework for composite convex programming},
  author={Liu, Ya-Feng and Liu, Xin and Ma, Shiqian},
  journal={Mathematics of Operations Research},
  volume={44},
  number={2},
  pages={632--650},
  year={2019},
  publisher={INFORMS}
}

@article{necoara2019complexity,
  title={Complexity of first-order inexact {Lagrangian} and penalty methods for conic convex programming},
  author={Necoara, Ion and Patrascu, Andrei and Glineur, Francois},
  journal={Optimization Methods and Software},
  volume={34},
  number={2},
  pages={305--335},
  year={2019},
  publisher={Taylor \& Francis}
}

@article{patrascu2017adaptive,
  title={Adaptive inexact fast augmented {Lagrangian} methods for constrained convex optimization},
  author={Patrascu, Andrei and Necoara, Ion and Tran-Dinh, Quoc},
  journal={Optimization Letters},
  volume={11},
  pages={609--626},
  year={2017},
  publisher={Springer}
}

@article{zhang2022global,
  title={A global dual error bound and its application to the analysis of linearly constrained non-convex optimization},
  author={Zhang, Jiawei and Luo, Zhi-Quan},
  journal={SIAM Journal on Optimization},
  volume={32},
  number={3},
  pages={2319--2346},
  year={2022},
  publisher={SIAM}
}

@article{zhang2020proximal,
  title={A proximal alternating direction method of multiplier for linearly constrained non-convex minimization},
  author={Zhang, Jiawei and Luo, Zhi-Quan},
  journal={SIAM Journal on Optimization},
  volume={30},
  number={3},
  pages={2272--2302},
  year={2020},
  publisher={SIAM}
}

@article{hajinezhad2019perturbed,
  title={Perturbed proximal primal--dual algorithm for non-convex non-smooth optimization},
  author={Hajinezhad, Davood and Hong, Mingyi},
  journal={Mathematical Programming},
  volume={176},
  number={1-2},
  pages={207--245},
  year={2019},
  publisher={Springer}
}

@article{huang2023single,
  title={Single-loop switching subgradient methods for non-smooth weakly convex optimization with non-smooth convex constraints},
  author={Huang, Yankun and Lin, Qihang},
  journal={ArXiv, preprint:2301.13314},
  year={2023}
}

@article{goncalves2017convergence,
  title={Convergence rate bounds for a proximal {ADMM} with over-relaxation stepsize parameter for solving non-convex linearly constrained problems},
  author={Goncalves, Max LN and Melo, Jefferson G and Monteiro, Renato DC},
  journal={ArXiv, preprint:1702.01850},
  year={2017}
}

@article{lan2020algorithms,
  title={Algorithms for stochastic optimization with function or expectation constraints},
  author={Lan, Guanghui and Zhou, Zhiqiang},
  journal={Computational Optimization and Applications},
  volume={76},
  number={2},
  pages={461--498},
  year={2020},
  publisher={Springer}
}

@article{bayandina2018mirror,
  title={Mirror descent and convex optimization problems with non-smooth inequality constraints},
  author={Bayandina, Anastasia and Dvurechensky, Pavel and Gasnikov, Alexander and Stonyakin, Fedor and Titov, Alexander},
  journal={Large-scale and distributed optimization},
  pages={181--213},
  year={2018},
  publisher={Springer}
}

@article{tran2014primal,
  title={A primal-dual algorithmic framework for constrained convex minimization},
  author={Tran-Dinh, Quoc and Cevher, Volkan},
  journal={ArXiv, preprint:1406.5403},
  year={2014}
}

@article{wei2018primal,
  title={Primal-dual Frank-Wolfe for constrained stochastic programs with convex and non-convex objectives},
  author={Wei, Xiaohan and Neely, Michael J},
  journal={ArXiv, preprint:1806.00709},
  year={2018}
}

@article{wei2018solving,
  title={Solving Non-smooth Constrained Programs with Lower Complexity than $\mathcal{O}(\frac{1}{\varepsilon}) $: A Primal-Dual Homotopy Smoothing Approach},
  author={Wei, Xiaohan and Yu, Hao and Ling, Qing and Neely, Michael},
  journal={Advances in Neural Information Processing Systems},
  volume={31},
  year={2018},
  publisher={The MIT press}
}

@article{hong2016decomposing,
  title={Decomposing linearly constrained non-convex problems by a proximal primal-dual approach: algorithms, convergence, and applications},
  author={Hong, Mingyi},
  journal={ArXiv, preprint:1604.00543},
  year={2016}
}

@article{jiang2019structured,
  title={Structured non-convex and non-smooth optimization: algorithms and iteration complexity analysis},
  author={Jiang, Bo and Lin, Tianyi and Ma, Shiqian and Zhang, Shuzhong},
  journal={Computational Optimization and Applications},
  volume={72},
  number={1},
  pages={115--157},
  year={2019},
  publisher={Springer}
}

@article{melo2017iteration,
  title={Iteration-complexity of a {Jacobi}-type non-{Euclidean} {ADMM} for multi-block linearly constrained non-convex programs},
  author={Melo, Jefferson G and Monteiro, Renato DC},
  journal={ArXiv, preprint:1705.07229},
  year={2017}
}

@incollection{danilova2022recent,
  title={Recent theoretical advances in non-convex optimization},
  author={Danilova, Marina and Dvurechensky, Pavel and Gasnikov, Alexander and Gorbunov, Eduard and Guminov, Sergey and Kamzolov, Dmitry and Shibaev, Innokentiy},
  booktitle={High-Dimensional Optimization and Probability: With a View Towards Data Science},
  pages={79--163},
  year={2022},
  publisher={Springer}
}

@article{curtis2012sequential,
  title={A sequential quadratic programming algorithm for non-convex, non-smooth constrained optimization},
  author={Curtis, Frank E and Overton, Michael L},
  journal={SIAM Journal on Optimization},
  volume={22},
  number={2},
  pages={474--500},
  year={2012},
  publisher={SIAM}
}

@inproceedings{vogel2021learning,
  title={Learning fair scoring functions: Bipartite ranking under roc-based fairness constraints},
  author={Vogel, Robin and Bellet, Aur{\'e}lien and Cl{\'e}men{\c{c}}on, Stephan},
  booktitle={International conference on artificial intelligence and statistics},
  pages={784--792},
  year={2021},
  organization={PMLR}
}

@inproceedings{kohavi1996scaling,
  title={Scaling up the accuracy of naive-bayes classifiers: A decision-tree hybrid.},
  author={Kohavi, Ron and others},
  booktitle={KDD},
  volume={96},
  pages={202--207},
  year={1996}
}

@incollection{angwin2022machine,
  title={Machine bias},
  author={Angwin, Julia and Larson, Jeff and Mattu, Surya and Kirchner, Lauren},
  booktitle={Ethics of data and analytics},
  pages={254--264},
  year={2022},
  publisher={Auerbach Publications}
}

@article{kong2023accelerated,
  title={An accelerated inexact dampened augmented {Lagrangian} method for linearly-constrained nonconvex composite optimization problems},
  author={Kong, Weiwei and Monteiro, Renato DC},
  journal={Computational Optimization and Applications},
  volume={85},
  number={2},
  pages={509--545},
  year={2023},
  publisher={Springer}
}

@article{melo2024proximal,
  title={A proximal augmented {Lagrangian} method for linearly constrained nonconvex composite optimization problems},
  author={Melo, Jefferson G and Monteiro, Renato DC and Wang, Hairong},
  journal={Journal of Optimization Theory and Applications},
  volume={202},
  number={1},
  pages={388--420},
  year={2024},
  publisher={Springer}
}

@book{wong2011active,
  title={Active-set methods for quadratic programming},
  author={Wong, Elizabeth},
  year={2011},
  publisher={University of California, San Diego}
}
		 
\end{document}